\documentclass[10pt,twoside, a4paper, english, reqno]{amsart}
\usepackage[dvips]{epsfig}
\usepackage{amscd}
\usepackage{a4wide}
\usepackage{amssymb}
\usepackage{amsthm}
\usepackage{amsmath}
\usepackage{bm} 
\usepackage{latexsym}
\usepackage{enumitem}
\usepackage{booktabs}
\usepackage{multirow}
\usepackage{makecell}
\usepackage{tabularx}
\numberwithin{figure}{section}
\numberwithin{table}{section}
\usepackage{graphicx,caption}
\DeclareMathOperator{\sech}{sech}

\usepackage{calc}
\usepackage{graphicx}
\usepackage{stmaryrd}
\usepackage{amsmath,amsthm,bm,mathrsfs}
\usepackage[utf8]{inputenc} 
\usepackage{url}            
\usepackage{hyperref}
\hypersetup{
    colorlinks=true,
    linkcolor=blue,
    filecolor=red,      
    urlcolor=blue,
    citecolor=red,
    }
\usepackage{amsfonts}       
\usepackage{nicefrac}       
\usepackage{microtype}      
\theoremstyle{plain}
\usepackage{doi}

\usepackage{upref}
\usepackage{hyperref}
\usepackage{ulem}
\usepackage{enumitem}
\newcommand{\norm}[1]{\left\Vert#1\right\Vert}
\newcommand{\abs}[1]{\left|#1\right|}

\newcommand{\R}{\mathbb R}

\newcommand{\N}{\mathbb N}

\newtheorem{theorem}{Theorem}[section]
\newtheorem{proposition}[theorem]{Proposition}
\newtheorem{lemma}[theorem]{Lemma}
\newtheorem{definition}[theorem]{Definition}

\newtheorem{remark}[theorem]{Remark}
\newtheorem{corollary}[theorem]{Corollary}

\numberwithin{equation}{section}     
\numberwithin{figure}{section}
\numberwithin{table}{section}

\allowdisplaybreaks

\newcounter{asnr}

{\ifnum\value{asnr}=0 \stepcounter{asnr} 
  \begin{enumerate}[label=\textbf{A}.\arabic{enumi}]
    \else
    \begin{enumerate}[label=\textbf{A}.\arabic{enumi},resume] \fi}
{\end{enumerate}}

\newcounter{defnr}

{\ifnum\value{defnr}=0 \stepcounter{defnr} 
  \begin{enumerate}[label=\textbf{D}.\arabic{enumi}]
    \else
    \begin{enumerate}[label=\textbf{D}.\arabic{enumi},resume] \fi}
{\end{enumerate}}

\numberwithin{equation}{section} \allowdisplaybreaks

\title[Convergence of LDG scheme for fractional KdV equation]
{Analysis of a local discontinuous Galerkin scheme for fractional Korteweg-de Vries equation}
\date{}

\author[M. Dwivedi]{Mukul Dwivedi}
\address[Mukul Dwivedi]{\newline
Department of Mathematics, 
	Indian Institute of Technology Jammu,
	Jagti, NH-44 Bypass Road, Post Office Nagrota,
	Jammu - 181221, India}
\email[]{2020rma1031@iitjammu.ac.in}

\author[T. Sarkar]{Tanmay Sarkar}
\address[Tanmay Sarkar]{\newline
	Department of Mathematics, 
	Indian Institute of Technology Jammu,
	Jagti, NH-44 Bypass Road, Post Office Nagrota,
	Jammu - 181221, India}
\email[]{tanmay.sarkar@iitjammu.ac.in}

\subjclass[2020]{Primary: 65M60, 35Q53; Secondary: 65M12.}

\keywords{Fractional Korteweg-de Vries equation; fractional Laplacian; fractional Sobolev spaces; local discontinuous Galerkin method.
}

\thanks{}

\begin{document}

\begin{abstract}
We propose a local discontinuous Galerkin (LDG) method for the fractional Korteweg-de Vries (KdV) equation, involving the fractional Laplacian with exponent $\alpha \in (1,2)$ in one and multiple space dimensions. By decomposing the fractional Laplacian into first-order derivatives and a fractional integral, we prove the $L^2$-stability of the semi-discrete LDG scheme incorporating suitable interface and boundary fluxes. We derive the optimal error estimate for linear flux and demonstrate an error estimate with an order of convergence $\mathcal{O}(h^{k+\frac{1}{2}})$ for general nonlinear flux utilizing the Gauss-Radau projections. Moreover, we extend the stability and error analysis to the multiple space dimensional case. Additionally, we discretize time using the Crank-Nicolson method to devise a fully discrete stable LDG scheme, and obtain a similar order error estimate as in the semi-discrete scheme. Numerical illustrations are provided to demonstrate the efficiency of the scheme, confirming an optimal order of convergence.
\end{abstract}

\maketitle

\section{Introduction}\label{sec1}
We consider the following degenerate nonlinear non-local integral partial differential equation known as the fractional Korteweg-de Vries (KdV) equation: 
\begin{equation}\label{fkdv}
\begin{cases}
     U_t+f(U)_x-(-\Delta)^{\alpha/2}U_x=0,  \qquad (x,t) \in \mathbb{R}\times (0,T],\\
     U(x,0) = U_0(x), \qquad x \in \mathbb{R},
\end{cases}
\end{equation}
where $T>0$ is fixed, $U_0$ is the prescribed initial data, $f$ is the given flux function of $U$ and $\alpha\in (1,2)$.
The non-local integro-differential operator $-(-\Delta)^{\alpha/2}$ in \eqref{fkdv} is the fractional Laplacian  with the values $\alpha\in(0,2)$ defined for all $\phi\in C_c^\infty(\R)$ by the Fourier transform as
 \begin{equation}\label{fracL}
     \widehat{(-\Delta)^{\alpha/2}\phi}(\xi) = |\xi|^\alpha \hat{\phi}(\xi).
 \end{equation}
The fractional KdV equation \eqref{fkdv}, featuring the fractional Laplacian, finds applications in various fields such as nonlinear dispersive equations, plasma physics, and inverse scattering methods \cite{abdelouhab1989nonlocal, dutta2021operator, fonseca2013ivp, kenig1991well, klein2015numerical}. In particular, whenever $\alpha=2$, \eqref{fkdv} reduces to the classical KdV equation \cite{Bona1975KdV,kato1983cauchy, kenig1991well, korteweg1895xli, sjoberg1970korteweg}, while for $\alpha=1$, it represents the Benjamin-Ono equation \cite{fonseca2013ivp, kenig1994generalized, thomee1998numerical}.

The fractional Laplacian operator \eqref{fracL} is a non-local operator, known for its efficacy in various localized computations such as image segmentation, modeling water flow in narrow channels, plasma physics, and other related applications \cite{podlubny1998fractional, di2012hitchhikers}. Its computational properties make it a valuable tool to study the phenomena where non-local effects play a crucial role. For further insights and applications of \eqref{fkdv}, refer to \cite{kenig1991well,dutta2021operator, dwivedi2023stability, dwivedi2024fully} and references therein. 

Over the years, extensive research has been conducted to study the well-posedness of the fractional KdV equation \eqref{fkdv} for various values of $\alpha$. For the classical case where $\alpha = 2$, seminal work by Kato \cite{kato1983cauchy} established the well-posedness and local smoothing effects of the solution of the KdV equation for the initial data in $L^2(\R)$. 
On the other hand, for $\alpha = 1$, corresponding to the Benjamin-Ono equation, foundational contributions on well-posedness were made by Fokas et al. in  \cite{fokas1981hierarchy}. Further improvements on the global well-posedness of the Benjamin-Ono equation were developed by Tao \cite{tao2004global} and Fonseca et al. \cite{fonseca2013ivp}. Whenever $\alpha\in (1,2)$, the approach to local well-posedness of equation \eqref{fkdv} differs significantly from the classical KdV equation as remarked in \cite{kenig1991well}.
However, Kenig et al. \cite{kenig1991well} explored the local well-posedness of the generalized dispersion model \eqref{fkdv} for initial data in $H^s(\R)$, $s>(9-3\alpha)/4$. These studies utilized compactness arguments based on the smoothing effect and\textit{ a priori }estimates of the solution. Recent advancement in the well-posedness of equation \eqref{fkdv} have been achieved through the frequency-dependent renormalization technique as demonstrated by Herr et al. \cite{herr2010differential} for the initial data in $L^2(\mathbb{R})$. These ongoing efforts underscore the continual endeavor to enhance the theoretical foundations of generalized dispersive equation \eqref{fkdv}. 

Several numerical methods have been developed and are employed in practice to study \eqref{fkdv}. However, we refer only to the literature which are relevant for our study. For $\alpha = 2$, recent years have seen the emergence of fully discrete finite difference schemes \cite{dwivedi2023convergence,holden2015convergence}, and continuous Galerkin schemes \cite{dutta2015convergence,dutta2016note}. In case of $\alpha = 1$, Thom{\'e}e et al. \cite{thomee1998numerical} introduced a fully implicit finite difference scheme. Afterwards, Dutta et al. \cite{dutta2016convergence} established the convergence of the fully discrete Crank-Nicolson scheme and Galtung \cite{galtung2018convergent} devised a continuous Galerkin scheme. However, for $\alpha \in (1,2)$, the literature on numerical methods for \eqref{fkdv} is relatively limited. An operator splitting scheme for \eqref{fkdv} was introduced by Dutta et al. \cite{dutta2021operator}. Furthermore, Dwivedi et al. \cite{dwivedi2023stability} designed a continuous Galerkin scheme and demonstrated its convergence to the weak solution and in \cite{dwivedi2024fully}, they developed a fully discrete finite difference scheme by discretizing the fractional Laplacian with an exponent $\alpha \in [1,2)$ and shown convergence to the unique solution of the fractional KdV equation \eqref{fkdv}. However, the existing numerical methods for fractional KdV equation \eqref{fkdv} lacks higher order accuracy in space. Hence there is a need for further study of stable and convergent numerical methods that are higher order accurate and efficient to compute the discontinuous solution for fractional KdV equation \eqref{fkdv}.

The discontinuous Galerkin (DG) method, a class of finite element method was first introduced by Reed and Hill in \cite{reed1973triangular}, employing discontinuous piecewise polynomial space for approximating the solution and test functions. 
In recent decades, Cockburn and Shu have significantly advanced the DG method for hyperbolic problems, see \cite{cockburn2012discontinuous} and references therein. However, due to instability and inconsistency, usual DG methods for the equations containing higher-order derivative terms may not perform efficiently, refer to \cite{hesthaven2007nodal, cockburn1998local, xu2010local} and references therein.
The Local Discontinuous Galerkin (LDG) method, a special type of the DG method, is often used for tackling higher-order time-dependent problems by reformulating the equation into a first-order system by introducing auxiliary variables to approximate the lower derivatives and then applying the DG method to the resulting system \cite{cockburn1998local}. This extension of DG method was motivated by the successful numerical experiments of Bassi and Rebay \cite{bassi1997high}.
The `local' term associated to LDG comes in a sense that auxiliary variables are superficial, and can be eliminated locally. The design of numerical fluxes at the interfaces play a crucial role to the success of such methods. Therefore, all the numerical fluxes must be designed in such a way that the method becomes stable and locally solvable for the auxiliary variables.

The LDG method has been developed to address equations featuring higher-order derivative terms while retaining its advantages in efficient adaptivity and higher order accuracy. For instance, Yan and Shu \cite{yan2002local} introduced the LDG method for KdV-type equations containing third-order spatial derivatives, achieving an error estimate of order $k+\frac{1}{2}$ in the linear case. Subsequently, Xu and Shu \cite{xu2007error} extended the error analysis to the nonlinear case and obtained the same order of convergence. Xu and Shu \cite{xu2010local} further expanded the LDG method to equations involving fourth and fifth order spatial derivatives, including nonlinear Schrödinger equations \cite{xu2005local}. Levy et al. \cite{levy2004local} developed a LDG method tailored for nonlinear dispersive equations with compactly supported traveling wave solutions.
The LDG method has gained attention for partial differential equations involving the fractional Laplacian in recent years. Notably, Xu and Hesthaven \cite{xu2014discontinuous} introduced a LDG method for the fractional convection-diffusion equation, decomposing the fractional Laplacian of order $\alpha$ into a first-order derivative and a fractional integral of order $2-\alpha$, achieving optimal error of $\mathcal{O}(h^{k+1})$ in the linear case and $\mathcal{O}(h^{k+\frac{1}{2}})$ in the nonlinear case. 

In this paper, our approach to design the LDG scheme for the fractional KdV equation \eqref{fkdv} involves decomposing the fractional Laplacian into a lower-order derivative and a fractional integral, and introducing auxiliary variables to represent \eqref{fkdv} into a system involving first order derivatives and fractional integral terms. This transformation modifies \eqref{fkdv} to an equation containing third-order derivatives along with a fractional integral of order $2-\alpha$ term instead of $-(-\Delta)^{\alpha/2}U_x$. 
Moreover, we discretize time using the Crank-Nicolson (CN) method to devise a fully discrete CN-LDG scheme.
Since the equation involves a non-local operator defined over the entire real line, we restrict our problem to a bounded domain for designing an efficient numerical scheme for \eqref{fkdv}, see \cite{xu2014discontinuous}. We employ the usual DG method within each element to all equations of the system. A critical aspect of this process is the construction of appropriate numerical fluxes at the interfaces and at the boundary with the help of boundary conditions. Furthermore, proving spatial stability in this case requires careful choice of test functions, unlike in the usual stability approach of existing LDG schemes for higher-order equations, see \cite{xu2014discontinuous,yan2002local}. Considering a linear flux we achieved the optimal error estimates utilizing the Gauss-Radau projections, and assuming $\norm{U-u_h}_{L^2(\Omega)}\leq h$, we establish an error estimate of $\mathcal{O}(h^{k+\frac{1}{2}})$ for nonlinear flux functions.  
Additionally, this approach is extended to do the stability and error analysis for the fully discrete CN-LDG scheme. Up to our knowledge, the LDG scheme has not been developed for the fractional KdV equation \eqref{fkdv}, which is stable and convergent. Furthermore, adapting the approach from the one dimension case and assuming $\norm{U-u_h}_{L^2(\Omega)}\leq h^{\frac{3}{2}}$, we analyze the stability and error analysis for the two space dimensions and obtain an error estimate of order $\mathcal{O}(h^{k+\frac{1}{2}})$ for $k\geq 2$. We provide an efficient algorithm to compute the fractional integral of order $2-\alpha$ using the Gauss-Jacobi quadrature formula in the numerical section and consequently, obtained the optimal convergence rates.

The rest of the paper is organized as follows: we commence our investigation by introducing key definitions and preliminary lemmas of fractional calculus in Section \ref{sec2}. Section \ref{sec3} outlines the development of the LDG scheme, accompanied by the proof of its spatial stability and the derivation of error estimates. In Section \ref{sec4}, we extend our analysis to the multiple space dimension as well, presenting the LDG scheme tailored for this equation along with its stability analysis.
Section \ref{sec5} includes the stability and error analysis for fully discrete CN-LDG scheme and a remark on higher order Runge-Kutta LDG scheme.
The order of convergence is validated through several numerical examples presented in Section \ref{sec6}. Finally, we end up with a few concluding remarks in Section \ref{sec7}.



\section{Preliminary results and fractional calculus}\label{sec2}
We introduce an alternative definition of the fractional Laplacian using the framework of fractional calculus \cite{el2006finite, muslih2010riesz, yang2010numerical}. Let us consider $u\in C_c^\infty(\R)$. Then the left and right integrals of fractional order $\alpha > 0$ are defined as follows:
\begin{align}\label{leftfrac}
    {}_{-\infty}I_x^\alpha u(x)&:={}_{-\infty}D_x^{-\alpha} u(x) = \frac{1}{\Gamma(\alpha)}\int_{-\infty}^x(x-t)^{\alpha -1}u(t)\,dt,\\
\label{rightfrac}
    {}_{x}I_\infty^\alpha u(x)&:= {}_{x}D_\infty^{-\alpha} u(x) = \frac{1}{\Gamma(\alpha)}\int^{\infty}_x(t-x)^{\alpha -1}u(t)\,dt,
\end{align}
where $I^\alpha u$ denotes the $\alpha$ times fractional integration of $u$. Whenever $\alpha$ is a natural number, these definitions correspond to exactly $\alpha$ times integration in the classical sense.

The left and right Riemann-Liouville derivatives of fractional order $\alpha$ $(n-1<\alpha<n)$, $n\in\N$ are described as
\begin{align}
    \label{leftfracD} {}_{-\infty}D_x^\alpha u(x) &= \frac{1}{\Gamma(n-\alpha)}\left(\frac{d}{dx}\right)^n\int_{-\infty}^x (x-t)^{n-\alpha -1}u(t)\,dt = \left(\frac{d}{dx}\right)^n\left({}_{-\infty}I_x^{n-\alpha} u(x)\right),\\
    \label{rightfracD}{}_{x}D_\infty^\alpha u(x) &= \frac{1}{\Gamma(n-\alpha)}\left(-\frac{d}{dx}\right)^n\int^{\infty}_x (t-x)^{n-\alpha -1}u(t)\,dt = \left(-\frac{d}{dx}\right)^n\left({}_{x}I_\infty^{n-\alpha} u(x)\right).
\end{align}
The fractional Laplacian using Riemann-Liouville fractional derivative (also known as a Riesz derivative) is defined as
\begin{equation}\label{fracLfrac}
    \frac{\partial^\alpha}{\partial|x|^\alpha}u(x) = -(-\Delta)^{\alpha/2}u(x) = -\frac{{}_{-\infty}D_{x}^\alpha u(x) + {}_{x}D_{\infty}^\alpha u(x)}{2\cos\left(\frac{\alpha\pi}{2}\right)}, \quad 1<\alpha<2.
\end{equation}
Building upon the aforementioned definition, we extend the notion to introduce the fractional Laplacian of negative order. When $\alpha < 0$, the negative order fractional Laplacian reduces to a fractional integral operator. As a consequence, for any $0 < s < 1$, we define
\begin{equation}\label{NegFracL}
\Delta_{-s/2}u(x) :=  \frac{{}_{-\infty}I_x^{s}u(x) + {}_xI^{s}_{\infty}u(x)}{2\cos\left(\frac{s\pi}{2}\right)}.
\end{equation}
This definition serves as the basis for various analytical investigations. Subsequently, we review a set of relevant definitions and results that will constitute the fundamental components of our subsequent analysis.

\begin{lemma}[See \cite{podlubny1998fractional}]\label{fracint_P}
Let $u,$ $v \in C^n(\R)$, $n\in \N$ be such that $\frac{d^j}{dx^j}v(x)= 0$ and $\frac{d^j}{dx^j}u(x)= 0$ as $x \to \pm\infty,$ $\forall ~0\leq j\leq n$. Then the fractional integrals and derivatives defined in \eqref{leftfrac}-\eqref{rightfracD} satisfy the following properties:
\begin{enumerate}[label=\roman*)]
    \item (\text{Linearity}) Let $\lambda,\mu\in \R$. Then
    \begin{align}\label{Lin_1}
        {}_{-\infty}I_x^{\alpha}(\lambda u(x) + \mu v(x) ) &= \lambda {}_{-\infty}I_x^{\alpha} u(x) + \mu {}_{-\infty}I_x^{\alpha} v(x),\\
        \label{Lin_2} 
        {}_{-\infty}D_x^{\alpha}(\lambda u(x) + \mu v(x) ) &= \lambda {}_{-\infty}D_x^{\alpha} u(x) + \mu {}_{-\infty}D_x^{\alpha} v(x).
    \end{align}
    \item (\text{Semi-group property}) Let $\alpha,$ $\beta>0$. Then
       \begin{align}\label{SemiG}
        {}_{-\infty}I_x^{\alpha+\beta}u(x) = {}_{-\infty}I_x^{\alpha} \left({}_{-\infty}I_x^{\beta} u(x)\right) = {}_{-\infty}I_x^{\beta} \left({}_{-\infty}I_x^{\alpha} u(x)\right).
    \end{align}
    \item Let $n-1<\alpha <n$ and $m=1, 2,\cdots,n$ for $n\in\N$. Then 
     \begin{align}\label{Deri_pull1}
       {}_{-\infty}D_x^{\alpha}u(x) &= \left(\frac{d}{dx}\right)^n\left({}_{-\infty}I_x^{n-\alpha} u(x)\right) = \left(\frac{d}{dx}\right)^{n-m}\left({}_{-\infty}I_x^{n-\alpha} \frac{d^mu(x)}{dx^m}\right),\\
       \label{Deri_pull2} {}_{x}D_\infty^{\alpha}u(x) &= \left(-\frac{d}{dx}\right)^n\left({}_{x}I_\infty^{n-\alpha} u(x)\right) = \left(-\frac{d}{dx}\right)^{n-m}\left({}_{x}I_\infty^{n-\alpha}\left(-\frac{d}{dx}\right)^mu(x)\right).
    \end{align}
\end{enumerate}
\end{lemma}
In order to decompose the fractional Laplacian into a lower-order derivative and a fractional integral, a crucial step for expressing the problem into a system, we apply the properties outlined in Lemma \ref{fracint_P}. This application yields the following result for $1<\alpha<2$:
\begin{align*}
    -(-\Delta)^{\alpha/2}u(x) &=  -\frac{1}{2\cos\left(\alpha\pi/2\right)}\left({}_{-\infty}D_{x}^\alpha u(x) + {}_{x}D_{\infty}^\alpha u(x)\right)\\
    &=\frac{1}{2\cos\left((2-\alpha) \pi/2\right)}\frac{d^2}{dx^2}\left({}_{-\infty}I_x^{2-\alpha}u(x) + {}_{x}I_\infty^{2-\alpha} u(x)\right) =\frac{d^2}{dx^2}\left(\Delta_{\frac{\alpha-2}{2}}u(x)\right)\\
    &= \frac{1}{2\cos\left((2-\alpha) \pi/2\right)} \frac{d}{dx}\left({}_{-\infty}I_x^{2-\alpha}\frac{du(x)}{dx} + {}_{x}I_\infty^{2-\alpha} \frac{du(x)}{dx}\right) =  \frac{d}{dx}\left(\Delta_{\frac{\alpha-2}{2}}\frac{du(x)}{dx}\right).
\end{align*}
Hence the fractional Laplacian can be written in the decomposed form
\begin{equation}\label{frac_dcmp}
     -(-\Delta)^{\alpha/2}u(x) = \frac{d}{dx}\left(\Delta_{\frac{\alpha-2}{2}}\frac{du(x)}{dx}\right) = \frac{d^2}{dx^2}\left(\Delta_{\frac{\alpha-2}{2}}u(x)\right).
\end{equation}
\begin{definition}[Fractional space \cite{ervin2006variational}]\label{Defn_fracnorm} Let $u\in C_c^\infty(\R)$ and $\mu>0$.
    We define the semi-norms 
    \begin{align}\label{SeminormfracL}
        \abs{u}_{J_L^\mu{(\mathbb R)}} &= \norm{{}_{-\infty}D_x^\mu u}_{L^2(\mathbb{R})},  \quad\qquad \qquad \abs{u}_{J_L^{-\mu}{(\mathbb R)}} = \norm{{}_{-\infty}I_x^\mu u}_{L^2(\mathbb{R})}, \\
         \abs{u}_{J_R^\mu(\mathbb{R})} &= \norm{{}_{x}D_\infty^\mu u}_{L^2(\mathbb{R})}, \qquad \qquad \qquad\abs{u}_{J_R^{-\mu}(\mathbb{R})} = \norm{{}_{x}I_\infty^\mu u}_{L^2(\mathbb{R})},\\
         \abs{u}_{J_S^\mu(\R)} &= \abs{\left({}_{-\infty}D_x^\mu u,{}_{x}D_{\infty}^\mu u\right)}^{\frac{1}{2}},  \qquad\quad \abs{u}_{J_S^{-\mu}(\R)} = \abs{\left({}_{-\infty}I_x^\mu u,{}_{x}I_{\infty}^\mu u\right)}^{\frac{1}{2}},
    \end{align}
    where $(\cdot,\cdot)$ is the standard $L^2$-inner product. We define the norms
    \begin{align}\label{NormfracL}
       \norm{u}_{J_L^\mu{(\mathbb R)}}  &= \abs{u}_{J_L^\mu{(\mathbb{R})}} + \norm{u}_{L^2(\mathbb{R})}, \qquad\qquad \norm{u}_{J_L^{-\mu}{(\mathbb R)}}  = \abs{u}_{J_L^{-\mu}{(\mathbb{R})}} + \norm{u}_{L^2(\mathbb{R})}, \\
       \norm{u}_{J_R^\mu(\mathbb{R})}  &= \abs{u}_{J_R^\mu(\mathbb{R})} + \norm{u}_{L^2(\mathbb{R})}, \qquad\qquad  \norm{u}_{J_R^{-\mu}(\mathbb{R})}  = \abs{u}_{J_R^{-\mu}(\mathbb{R})} + \norm{u}_{L^2(\mathbb{R})},\\
       \norm{u}_{J_S^\mu(\R)}  &= \left(\abs{u}_{J_S^\mu(\R)}^2 + \norm{u}^2_{L^2(\R)}\right)^{\frac{1}{2}},\qquad \norm{u}_{J_S^{-\mu}(\R)}  = \left(\abs{u}_{J_S^{-\mu}(\R)}^2 + \norm{u}^2_{L^2(\R)}\right)^{\frac{1}{2}}
    \end{align}
    and let $J_L^\mu{(\mathbb{R})}$, $J_R^\mu(\mathbb{R})$ and $J_S^\mu(\R)$ denote the closure of $C_c^{\infty}(\mathbb{R})$ with respect to $\norm{\cdot}_{J_L^\mu{(\mathbb{R})}}$, $\norm{\cdot}_{J_R^\mu(\mathbb{R})}$ and $\norm{\cdot}_{J_S^\mu(\R)}$ respectively. 
\end{definition}
\begin{remark}
It is important to note that these definitions can be naturally extended to the case where $\mu=0$. In this particular case, all the norms defined in \ref{Defn_fracnorm} reduce to the $L^2$-norm.
\end{remark}
Utilizing the Definition \ref{Defn_fracnorm}, we have the following lemma.
\begin{lemma}[See \cite{ervin2006variational}]\label{Bound_Fracint}
    Let $u\in C_c^\infty(\R)$ and $\mu>0$. Then there holds
    \begin{enumerate}[label=\roman*)]
        \item Adjoint property:$$({}_{-\infty}I_x^\mu u,u) = (u,{}_{x}I_{\infty}^\mu u).$$
        \item Inner-product between fractional integrals: 
        $$({}_{-\infty}I_x^\mu u,{}_{x}I_{\infty}^\mu u) = \cos(\mu\pi)|u|_{J^{-\mu}_L(\R)}^2 = \cos(\mu\pi)|u|_{J^{-\mu}_R(\R)}^2.$$
    \end{enumerate}
\end{lemma}
\begin{lemma}[See \cite{xu2014discontinuous}]\label{Inner_frac}
    Let $0<s<1$. Then the negative fractional Laplacian defined by \eqref{NegFracL} satisfies the following identity:
    \begin{equation}
        (\Delta_{-s/2} u,u) = |u|_{J^{-s}_L(\R)}^2 = |u|_{J^{-s}_R(\R)}^2.
    \end{equation}
\end{lemma}
We formulate the scheme for \eqref{fkdv} within a bounded domain instead of $\R$. Consequently, we confine the definitions and identities to the bounded domain $\Omega=[a,b]$ in the following analysis.
\begin{definition}\label{Defn_fracnorm2}
    The fractional derivative spaces $J_{L,0}^\mu(\Omega)$, $J_{R,0}^\mu(\Omega)$ and $J_{S,0}^\mu(\Omega)$ of order $\mu\in \R$ are the closure of $C_c^\infty(\Omega)$ under their respective norms defined in Definition \ref{Defn_fracnorm}.
\end{definition}

\begin{lemma}[See \cite{kilbas2006theory}]\label{kilbas}
    Let $0<s<1$. Then there is a constant $C$ such that the fractional integrals defined by \eqref{leftfrac} and \eqref{rightfrac} satisfy the following:
    \begin{equation*}
       \norm{{}_{a}I_x^s u}_{L^2(\Omega)} \leq C \norm{u}_{L^2(\Omega)},\qquad \norm{{}_{x}I_{b}^s u} \leq C \norm{u}_{L^2(\Omega)}.
    \end{equation*}
    Moreover, we have
    \begin{equation*}
       \norm{\Delta_{-s/2} u}_{L^2(\Omega)} \leq C \norm{u}_{L^2(\Omega)}.
    \end{equation*}
\end{lemma}

\begin{lemma}[Fractional Poincar\'e-Friedrichs \cite{ervin2006variational}]\label{fpfred}
    Let $u\in J_{L,0}^\mu(\Omega)$ and $\mu\in\R$. Then, we have the following estimate
    \begin{equation*}
        \norm{u}_{L^2(\Omega)}\leq C |u|_{J_{L,0}^\mu(\Omega)},
    \end{equation*}
    and for $u\in J_{R,0}^\mu(\Omega)$, we have 
    \begin{equation*}
        \norm{u}_{L^2(\Omega)}\leq C |u|_{J_{R,0}^\mu(\Omega)}.
    \end{equation*}
\end{lemma}

\begin{lemma}[See Theorem 2.8 in \cite{deng2013local}]\label{lemma_fracnorm2}
    Let $\mu_2<\mu_1\leq 0$. Then the fractional derivative spaces $J_{L,0}^{\mu_1}(\Omega)$ (or $J_{R,0}^{\mu_1}(\Omega)$ or $J_{S,0}^{\mu_1}(\Omega)$) is embedded into $J_{L,0}^{\mu_2}(\Omega)$ (or $J_{R,0}^{\mu_2}(\Omega)$ or $J_{S,0}^{\mu_2}(\Omega)$). In particular, $L^2(\Omega)$ is embedded into both spaces.
\end{lemma}

\section{LDG scheme for the one-dimensional fractional KdV equation}\label{sec3}

To address the complexity introduced by the higher order derivative with fractional Laplacian, we decompose it into a lower-order derivative and a fractional integral. Regarding this, we introduce the suitable auxiliary variables. Consequently, the equation \eqref{fkdv} transforms into a system of fractional integral equation and first-order derivative equations. To solve this system, we employ the discontinuous Galerkin \cite{cockburn2012discontinuous} method. This approach is introduced by Cockburn et al. \cite{cockburn1998local} for the development of LDG method to the problems involving higher order derivatives. The following analysis can be considered as an extension of the LDG  scheme developed for the fractional convection-diffusion equation \cite{xu2014discontinuous} to the fractional KdV equation \eqref{fkdv}.
However, the subsequent analysis of the proposed LDG scheme differs significantly due to the treatment of additional auxiliary variable.

Considering the decomposition of the fractional Laplacian by the equation  \eqref{frac_dcmp}, we introduce three auxiliary variables, namely, $P$, $Q$, and $R$ such that
\begin{align*}
    P = \Delta_{\frac{(\alpha-2)}{2}}Q, \qquad Q = R_x , \qquad R= U_x.
\end{align*}
As a consequence, the equation \eqref{fkdv} can be rewritten in the following form
\begin{equation}
\begin{split}\label{systemfkdv}
    U_t & = -\left(f(U) + P\right)_x ,\\
              P &= \Delta_{\frac{(\alpha-2)}{2}}Q,\\
              Q&= R_x,\\
              R&= U_x.
\end{split}
\end{equation}

Our approach to implement the DG scheme for the system \eqref{systemfkdv} begins with introducing the weak formulation of the system. This involves multiplying the system \eqref{systemfkdv} by test functions from the appropriate finite element space and integrating over an element. Subsequently, flux terms emerge by the use of integration by parts. To initiate this process, we introduce the finite element space denoted by $V_h^k$ and subsequent discretization of the domain.

\begin{remark}\label{remk1}
The equation \eqref{fkdv} is defined over $\mathbb{R}$ due to non-local operator fractional Laplacian. However, to design an efficient numerical scheme, following the approach in \cite{xu2014discontinuous,liu2006local}, we restrict it to a sufficient large bounded domain $\Omega = [a, b]\subset \mathbb{R}$ by assuming that $U$ has a compact support within $\Omega$. Consequently, we impose homogeneous boundary conditions to the problem \eqref{fkdv}, i.e. $U(a,t) = 0 = U(b,t)$ for all $t<T$.
Moreover, in order to establish the $L^2$-stability estimate, we encounter with the term $UU_{xx} = (U,R_x)$ on the right boundary, where $U_{xx}$ is not known a priori. This fact leads us to impose the mixed boundary conditions to get the stability. More precisely, we consider the  homogeneous Dirichlet boundary data along with Neumann boundary data on the right, i.e. $U_x(b,t) = 0$, for all $t<T$. Hence our original Cauchy problem transformed into the initial boundary value problem with homogeneous boundary data.
\end{remark}
We discretize the domain $\Omega=[a,b]$ using the partition $a=x_{\frac{1}{2}}<x_{\frac{3}{2}}<\cdots<x_{N+\frac{1}{2}}=b$, where $N$ is the number of elements. We denote the mesh of elements by $ \mathcal{I}:= \{I_i = (x_{i-\frac{1}{2}},x_{i+\frac{1}{2}})|~i=1, 2,\cdots, N\} $, and $h_i = x_{i+\frac{1}{2}}-x_{i-\frac{1}{2}}$ is the spatial step size with $h=\max\limits_{1\leq i\leq N} h_i$. We define the following broken Sobolev spaces associated with the mesh of elements $\mathcal{I}$:
\[ L^2(\Omega,{\mathcal{I}}) := \{v:\Omega\to \mathbb{R} \text{ such that } v|_{I_i}\in L^2(I_i),\,i=1,2,\cdots,N\}; \]
and 
\[ H^1(\Omega,{\mathcal{I}}) := \{v:\Omega\to \mathbb{R} \text{ such that } v|_{I_i}\in H^1(I_i),\,i=1,2,\cdots,N\}. \]
For a function $v\in H^1(\Omega,{\mathcal{I}})$, we denote the value at nodes $\{x_{i+\frac{1}{2}}\}$ as follows
$$v_{i+\frac{1}{2}}:=v(x_{i+\frac{1}{2}}), \qquad v^\pm_{i+\frac{1}{2}} = v(x^\pm_{i+\frac{1}{2}}):= \lim_{x\to x^\pm_{i+\frac{1}{2}}}v(x).$$
Furthermore, we define the local inner product and local $L^2(I_i)$ norm as follows:
\begin{equation*}
    (u,v)_{I_i} = \int_{I_i} uv\,dx,\qquad \|u\|_{I_i} = (u,u)_{I_i}^{\frac{1}{2}}.
\end{equation*}
Prior to introduce the LDG scheme, we assume that the exact solution $(U,P,Q,R)$ of the system \eqref{systemfkdv} lies in
\[\mathcal{H}(\Omega,\mathcal{I}):= H^1(0,T;H^1(\Omega,{\mathcal{I}}))\times L^2(0,T;H^1(\Omega,{\mathcal{I}})) \times L^2(0,T;L^2(\Omega,{\mathcal{I}})) \times L^2(0,T;H^1(\Omega,{\mathcal{I}})). \]
This assumption carries no ambiguity, as both $L^2(\Omega)$ and $H^1(\Omega)$ is embedded in the fractional spaces defined in Definition \ref{Defn_fracnorm2} by Lemma \ref{lemma_fracnorm2}. Thus, the solution $(U,P,Q,R)$ satisfies the system: for $i=1,2,\cdots,N$,
\begin{equation}\label{exactsolun}
\begin{split}
     \left(U_t,v\right)_{I_i} & = \left(f(U) + P,v_x\right)_{I_i} - \left( f v +  P v\right)\Big|_{x_{i-\frac{1}{2}}^+}^{x_{i+\frac{1}{2}}^-},\\
              \left(P,w\right)_{I_i} &= \left(\Delta_{\frac{(\alpha-2)}{2}}Q,w\right)_{I_i},\\
              \left(Q,z\right)_{I_i}&= -\left(R,z_x\right)_{I_i} + \left( R z\right)\Big|_{x_{i-\frac{1}{2}}^+}^{x_{i+\frac{1}{2}}^-},\\
              \left(R,s\right)_{I_i}&= -\left(U,s_x\right)_{I_i} +\left( U s\right)\Big|_{x_{i-\frac{1}{2}}^+}^{x_{i+\frac{1}{2}}^-},
\end{split}
\end{equation}
for all $w\in L^2(\Omega,{\mathcal{I}})$ and $v,z,s\in H^1(\Omega,{\mathcal{I}})$.

We define the finite dimensional discontinuous piecewise polynomial space $V_h^k\subset H^1(\Omega,{\mathcal{I}})$ by
\begin{equation}\label{elem_space}
    V_h^k = \{v:\Omega \to \R \text{ such that }v|_{I_i}\in P^k (I_i),~ \forall i=1,2,\cdots,N\},
\end{equation}
where $P^k(I_i)$ is the space of polynomials of degree up to order $k$ $(\geq 1)$ on $I_i$. We have the following result as a consequence of Lemma \ref{lemma_fracnorm2}.
\begin{proposition}\label{fracemb}
    The finite element space $V_h^k$ is embedded into the fractional derivative spaces $J_{L,0}^\mu(\Omega)$, $J_{R,0}^\mu(\Omega)$ and $J_{S,0}^\mu(\Omega)$ of order $\mu \leq 0$. 
\end{proposition}


\textbf{LDG scheme:} We seek an approximations $(u_h,p_h,q_h,r_h)\in H^1(0,T;V_h^k)\times L^2(0,T;V_h^k)\times L^2(0,T;V_h^k)\times L^2(0,T;V_h^k) =: \mathcal{T}_4\times\mathcal{V}_h^k$ to $(U,P,Q,R)$, where $U$ is an exact solution of \eqref{fkdv} with \[P = \Delta_{\frac{(\alpha-2)}{2}}Q,\qquad Q = R_x ,\qquad R= U_{x},\] such that for all test functions $(v,w,z,s)\in \mathcal{T}_4\times\mathcal{V}_h^k$ and $i=1,2,\cdots,N$, we have
\begin{equation}\label{LDGscheme1}
\begin{split}
     \left((u_h)_t,v\right)_{I_i} & = \left(f(u_h) + p_h,v_x\right)_{I_i} - \left(\hat f_h v + \hat p_h v\right)|_{x_{i-\frac{1}{2}}^+}^{x_{i+\frac{1}{2}}^-},\\
              \left(p_h,w\right)_{I_i} &= \left(\Delta_{\frac{(\alpha-2)}{2}}q_h,w\right)_{I_i},\\
              \left(q_h,z\right)_{I_i}&= -\left(r_h,z_x\right)_{I_i} + \left(\hat r_h z\right)|_{x_{i-\frac{1}{2}}^+}^{x_{i+\frac{1}{2}}^-},\\
              \left(r_h,s\right)_{I_i}&= -\left(u_h,s_x\right)_{I_i} +\left(\hat u_h s\right)|_{x_{i-\frac{1}{2}}^+}^{x_{i+\frac{1}{2}}^-},\\
              \left(u_h^0,v\right)_{I_i} &= \left(U_0,v\right)_{I_i}.
\end{split}
\end{equation}

To complete the LDG scheme \eqref{LDGscheme1}, it is necessary to define the numerical fluxes $\hat u_h,$ $\hat p_h,$ $\hat r_h,$ and $\hat f_h.$ To do this, we introduce the following notations:
\begin{equation*}
   \{\!\!\{u\}\!\!\} = \frac{u^-+u^+}{2},\qquad \llbracket u \rrbracket = u^+-u^-.
\end{equation*}
For the numerical fluxes, we opt for alternative fluxes at all interfaces ${x_{j+\frac{1}{2}}},~j=1,2,\cdots,N-1$:
\begin{align}\label{fluxes}
   &\hat p_h = p_h^+,\qquad \hat r_h = r_h^+,\qquad \hat u_h = u_h^- 
\end{align}
or alternatively,
\begin{align*}
   &\hat p_h = p_h^-,\qquad \hat r_h = r_h^+,\qquad \hat u_h = u_h^+.
\end{align*}
Sign of $\hat r_h$ usually depends on the sign of higher order derivative term, which is positive in  equation \eqref{fkdv}. We define the fluxes at the boundary
\begin{equation}\label{Boundaryfluxes}
    \begin{split}
    (\hat r_h)_{\frac{1}{2}}  =(r_h)^+_{\frac{1}{2}},\qquad
    (\hat p_h)_{N+\frac{1}{2}} = (p_h)^-_{N+\frac{1}{2}}, \qquad (\hat p_h) _{\frac{1}{2}} = (p_h)^+_{\frac{1}{2}},
    \end{split}
\end{equation}
   and rest of the boundary fluxes can be chosen by imposing the mixed boundary conditions and compact support of $U$ in $\Omega$, that we have already mentioned in Remark \ref{remk1}. In this case, we have
   \begin{equation}\label{Boundaryfluxes2}
       (\hat u_h)_{\frac{1}{2}} = U(a,t) = 0,  \qquad (\hat u_h)_{N+\frac{1}{2}} = U(b,t) = 0,\qquad(\hat r_h)_{N+\frac{1}{2}} = (r_h)^-_{N+\frac{1}{2}}= U_x(b,t)= 0.
   \end{equation}

For the numerical flux function $\hat f_h$, we can use any monotone flux \cite{yan2002local}. In particular, we use the following  Lax-Friedrichs flux 
\begin{equation}\label{LF}
    \hat f_h = \hat f (u_h^-,u_h^+) = \frac{1}{2}(f(u_h^-) + f(u_h^+) -\delta \llbracket u_h\rrbracket),\qquad \delta = \max\limits_u|f'(u)|,
\end{equation}
where the maximum is taken over a range of $u$ in a relevant element.

\begin{remark}
It is important to note that, for the sake of writing convenience, we adopt the notation $(v, w, z, s)\in V_h^k$ instead of $(v, w, z, s)\in \mathcal{T}_4\times\mathcal{V}_h^k$. The notation $(v, w, z, s)\in V_h^k$ signifies that each element within the tuple belongs to $V_h^k$. This choice is motivated by the fact that the time variable does not play a role in defining the finite element space.
\end{remark}
\subsection{Stability estimates}
In this section, we analyze the $L^2$-stability of the semi-discrete scheme \eqref{LDGscheme1}-\eqref{fluxes} designed for the fractional KdV equation \eqref{fkdv}.
\begin{lemma}[\textbf{Stability}]\label{stablemma}
Let $u_h$ be the approximate solution obtained by the LDG scheme \eqref{LDGscheme1}-\eqref{Boundaryfluxes2} with the auxiliary variables $p_h,$ $q_h$ and $r_h$. Then the LDG scheme \eqref{LDGscheme1}-\eqref{Boundaryfluxes2} is $L^2$-stable. Moreover, there holds
\begin{equation}\label{stabbound}
    \norm{u_h(\cdot,T)}_{L^2(\Omega)} \leq e^{-2CT}\norm{u^0_h}_{L^2(\Omega)},
\end{equation}
    for any $T>0$ and a constant $C$.
\end{lemma}
\begin{proof}
To carry out the stability analysis, incorporating \eqref{LDGscheme1}, we define
    \begin{align}\label{perturbation}
           \nonumber \mathcal{B}(u,p,q,r;& v,w,z,s)  := \sum\limits_{i=1}^N\Big[\left(u_t,v\right)_{I_i} - \left(f(u) + p,v_x\right)_{I_i}  +\left(p,w\right)_{I_i} -  \left(\Delta_{\frac{(\alpha-2)}{2}}q,w\right)_{I_i}+ \left(q,z\right)_{I_i}
             \\&  + \left(r,z_x\right)_{I_i} 
       +\left(r,s\right)_{I_i} +\left(u,s_x\right)_{I_i} + \left((\hat f  + \hat p )v\right)|_{x_{i-\frac{1}{2}}^+}^{x_{i+\frac{1}{2}}^-}-\left(\hat r z\right)|_{x_{i-\frac{1}{2}}^+}^{x_{i+\frac{1}{2}}^-} -
        \left(\hat u s\right)|_{x_{i-\frac{1}{2}}^+}^{x_{i+\frac{1}{2}}^-}\Big],
    \end{align}
for all $u,p,q,r \in \mathcal{H}(\Omega,\mathcal{I})$ and $v,w,z,s \in V_h^k$. We observe that the interface fluxes can be simplified as follows
\begin{align*}
    \sum\limits_{i=1}^N \big((\hat f+\hat p)v\big)\big|_{x_{i-\frac{1}{2}}^+}^{x_{i+\frac{1}{2}}^-} = -(\hat f_{\frac{1}{2}}+\hat p_{\frac{1}{2}})v^+_{\frac{1}{2}} + (\hat f_{N+\frac{1}{2}}+\hat p_{N+\frac{1}{2}})v^-_{N+\frac{1}{2}} - \sum\limits_{i=1}^{N-1} (\hat f_{i+\frac{1}{2}}+\hat p_{i+\frac{1}{2}})\llbracket v \rrbracket_{i+\frac{1}{2}},
\end{align*}
\begin{align*}
    \sum\limits_{i=1}^N (\hat rz)|_{x_{i-\frac{1}{2}}^+}^{x_{i+\frac{1}{2}}^-} = -\hat r_{\frac{1}{2}}z^+_{\frac{1}{2}} + \hat r_{N+\frac{1}{2}}z^-_{N+\frac{1}{2}} - \sum\limits_{i=1}^{N-1} \hat r_{i+\frac{1}{2}}\llbracket z \rrbracket_{i+\frac{1}{2}},
\end{align*}
and 
\begin{align*}
    \sum\limits_{i=1}^N (\hat us)|_{x_{i-\frac{1}{2}}^+}^{x_{i+\frac{1}{2}}^-} = -\hat u_{\frac{1}{2}}s^+_{\frac{1}{2}} + \hat u_{N+\frac{1}{2}}s^-_{N+\frac{1}{2}} - \sum\limits_{i=1}^{N-1} \hat u_{i+\frac{1}{2}}\llbracket s \rrbracket_{i+\frac{1}{2}}. 
\end{align*}
Using the numerical fluxes \eqref{fluxes}-\eqref{Boundaryfluxes2} and the above identities,  \eqref{perturbation} can be rewritten as 
\begin{align}\label{perturbation2}
            \nonumber\mathcal{B}(u,p,q,r;v,w,z,s) &= \sum\limits_{i=1}^N\Big[\left(u_t,v\right)_{I_i} - \left(\left(f(u) +p\right),v_x\right)_{I_i} +\left(p,w\right)_{I_i} -  \left(\Delta_{\frac{(\alpha-2)}{2}}q,w\right)_{I_i}
        \\&\quad  + \left(q,z\right)_{I_i} + \left(r,z_x\right)_{I_i} 
      +\left(r,s\right)_{I_i} +\left(u,s_x\right)_{I_i}\Big] + \mathcal{IF}(u,p,r;v,z,s),
    \end{align}
     where numerical flux $\mathcal{IF}$ at interfaces is given by
 \begin{align}\label{Interfacesflux}
       \nonumber \mathcal{IF}(u,p,r;v,z,s) &= -\left(\hat f_{\frac{1}{2}}+p_{\frac{1}{2}}^+\right)v^+_{\frac{1}{2}}  + \left(\hat f_{N+\frac{1}{2}}+ p_{N+\frac{1}{2}}^-\right)v^-_{N+\frac{1}{2}}- \sum\limits_{i=1}^{N-1} (\hat f_{i+\frac{1}{2}}+p_{i+\frac{1}{2}}^+)\llbracket v \rrbracket_{i+\frac{1}{2}}
        \\ &\quad +r_{\frac{1}{2}}^+z^+_{\frac{1}{2}}-r_{N+\frac{1}{2}}^-z^-_{N+\frac{1}{2}} + \sum\limits_{i=1}^{N-1} r_{i+\frac{1}{2}}^+\llbracket z \rrbracket_{i+\frac{1}{2}}
       +\sum\limits_{i=1}^{N-1} u_{i+\frac{1}{2}}^- \llbracket s \rrbracket_{i+\frac{1}{2}}.
    \end{align}
In order to estimate $\mathcal{B}$, we choose test functions $(v, w, z, s) = (u, -q+p+r, u+r, r-p)$ in \eqref{perturbation2} to obtain
\begin{align}\label{perturbation3}
           \nonumber \mathcal{B}(u,p,&q,r;u,-q+p+r,u+r,r-p) \\& \nonumber= \sum\limits_{i=1}^N\Big[\left(u_t,u\right)_{I_i} - \left(f(u), u_x\right)_{I_i} - \left(p,u_x\right)_{I_i}  -\left(p,q\right)_{I_i} + \left(p,p\right)_{I_i}  + \left(p,r\right)_{I_i} +  \left(\Delta_{\frac{(\alpha-2)}{2}}q,q\right)_{I_i} \\&\quad \nonumber - \left(\Delta_{\frac{(\alpha-2)}{2}}q,p\right)_{I_i}   - \left(\Delta_{\frac{(\alpha-2)}{2}}q,r\right)_{I_i} 
             + \left(q,u\right)_{I_i}  + \left(q,r\right)_{I_i} 
         + \left(r,u_x\right)_{I_i} +\left(r,r_x\right)_{I_i}  \\&\quad +\left(r,r\right)_{I_i} -\left(r,p\right)_{I_i}
      +\left(u,r_x\right)_{I_i} - \left(u,p_x\right)_{I_i}\Big] 
        + \mathcal{IF}(u,p,r;u,u+r,r-p).
    \end{align}
Applying the integration by parts, we have $$(p,u_x)_{I_i} + (u,p_x)_{I_i} = (up)|_{x_{i-\frac{1}{2}}^+}^{x_{i+\frac{1}{2}}^-} \quad \text{and} \quad(r,u_x)_{I_i} + (u,r_x)_{I_i} = (ur)|_{x_{i-\frac{1}{2}}^+}^{x_{i+\frac{1}{2}}^-},$$ 
and substituting these identities in equation \eqref{perturbation3}, the compact form $\mathcal{B}$ reduces to 
    \begin{align}\label{perturbation4}
       \nonumber \mathcal{B}(u,p,&q,r;u,-q+p+r,u+r,r-p)\\ \nonumber=& 
         \sum\limits_{i=1}^N\Big[\left(u_t,u\right)_{I_i} - \left(f(u), u_x\right)_{I_i} -\left(p,q\right)_{I_i} + \left(p,p\right)_{I_i} +  \left(\Delta_{\frac{(\alpha-2)}{2}}q,q\right)_{I_i}- \left(\Delta_{\frac{(\alpha-2)}{2}}q,p\right)_{I_i}  \\&\quad \nonumber- \left(\Delta_{\frac{(\alpha-2)}{2}}q,r\right)_{I_i} + \left(q,u\right)_{I_i}+ \left(q,r\right)_{I_i} 
         +\left(r,r_x\right)_{I_i} +\left(r,r\right)_{I_i}
           \Big] \\&\quad \underbrace{-\sum\limits_{i=1}^N (up)|_{x_{i-\frac{1}{2}}^+}^{x_{i+\frac{1}{2}}^-} + \sum\limits_{i=1}^N (ur)|_{x_{i-\frac{1}{2}}^+}^{x_{i+\frac{1}{2}}^-}
             + \mathcal{IF}(u,p,r;u,u+r,r-p)}_{\mathcal{E}_1}.
    \end{align}
In order to simplify $\mathcal{E}_1$, we observe that
\begin{equation*}
    \sum\limits_{i=1}^N (up)|_{x_{i-\frac{1}{2}}^+}^{x_{i+\frac{1}{2}}^-} = -u^+_{\frac{1}{2}}p^+_{\frac{1}{2}} + u^-_{N+\frac{1}{2}} p^-_{N+\frac{1}{2}} - \sum\limits_{i=1}^{N-1}u^-_{i+\frac{1}{2}} \llbracket p \rrbracket_{i+\frac{1}{2}} -\sum\limits_{i=1}^{N-1}p^+_{i+\frac{1}{2}} \llbracket u \rrbracket_{i+\frac{1}{2}},
\end{equation*}
and in a similar manner,
\begin{equation*}
    \sum\limits_{i=1}^N (ur)|_{x_{i-\frac{1}{2}}^+}^{x_{i+\frac{1}{2}}^-} = -u^+_{\frac{1}{2}}r^+_{\frac{1}{2}}+u^-_{N+\frac{1}{2}}r^-_{N+\frac{1}{2}}- \sum\limits_{i=1}^{N-1}u^-_{i+\frac{1}{2}} \llbracket r \rrbracket_{i+\frac{1}{2}} -\sum\limits_{i=1}^{N-1}r^+_{i+\frac{1}{2}} \llbracket u \rrbracket_{i+\frac{1}{2}}.
\end{equation*}
As a result, the term $\mathcal{E}_1$ becomes
    \begin{equation*}
        \begin{split}
            \mathcal{E}_1&=-\sum\limits_{i=1}^N (up)|_{x_{i-\frac{1}{2}}^+}^{x_{i+\frac{1}{2}}^-}+ \sum\limits_{i=1}^N (ur)|_{x_{i-\frac{1}{2}}^+}^{x_{i+\frac{1}{2}}^-}
             + \mathcal{IF}(u,p,r;u,u+r,r-p) \\&=u^+_{\frac{1}{2}}p^+_{\frac{1}{2}} - u^-_{N+\frac{1}{2}} p^-_{N+\frac{1}{2}} + \sum\limits_{i=1}^{N-1}u^-_{i+\frac{1}{2}} \llbracket p \rrbracket_{i+\frac{1}{2}} +\sum\limits_{i=1}^{N-1}p^+_{i+\frac{1}{2}} \llbracket u \rrbracket_{i+\frac{1}{2}} -u^+_{\frac{1}{2}}r^+_{\frac{1}{2}} + u^-_{N+\frac{1}{2}}r^-_{N+\frac{1}{2}}\\&\qquad 
             - \sum\limits_{i=1}^{N-1}u^-_{i+\frac{1}{2}} \llbracket r \rrbracket_{i+\frac{1}{2}} -\sum\limits_{i=1}^{N-1}r^+_{i+\frac{1}{2}} \llbracket u \rrbracket_{i+\frac{1}{2}} -\left(\hat f_{\frac{1}{2}}+ p^+_{\frac{1}{2}}\right)u^+_{\frac{1}{2}} + \left(\hat f_{N+\frac{1}{2}}+p^-_{N+\frac{1}{2}}\right)u^-_{N+\frac{1}{2}} \\&\qquad 
              - \sum\limits_{i=1}^{N-1} (\hat f_{i+\frac{1}{2}}+p^+_{i+\frac{1}{2}})\llbracket u \rrbracket_{i+\frac{1}{2}}+ r^+_{\frac{1}{2}}(u+r)^+_{\frac{1}{2}} - r^-_{N+\frac{1}{2}}(u+r)^-_{N+\frac{1}{2}} + \sum\limits_{i=1}^{N-1}r^+_{i+\frac{1}{2}} \llbracket u+r \rrbracket_{i+\frac{1}{2}} 
            \\ &\qquad + \sum\limits_{i=1}^{N-1}u^-_{i+\frac{1}{2}} \llbracket r-p \rrbracket_{i+\frac{1}{2}}\\
            &= -\hat f_{\frac{1}{2}}u^+_{\frac{1}{2}} + \hat f_{N+\frac{1}{2}}u^-_{N+\frac{1}{2}}  -\sum\limits_{i=1}^{N-1} \hat f_{i+\frac{1}{2}}\llbracket u \rrbracket_{i+\frac{1}{2}}+\sum\limits_{i=1}^{N-1}r^+_{i+\frac{1}{2}}  \llbracket r \rrbracket_{i+\frac{1}{2}} + (r^+_{\frac{1}{2}})^2 -(r^-_{N+\frac{1}{2}})^2 \\&
            = -\hat f_{\frac{1}{2}}u^+_{\frac{1}{2}} + \hat f_{N+\frac{1}{2}}u^-_{N+\frac{1}{2}}  -\sum\limits_{i=1}^{N-1} \hat f_{i+\frac{1}{2}}\llbracket u \rrbracket_{i+\frac{1}{2}}+\sum\limits_{i=1}^{N-1}\big((r^+_{i+\frac{1}{2}})^2 -  r^+_{i+\frac{1}{2}}r^-_{i+\frac{1}{2}}\big) + (r^+_{\frac{1}{2}})^2-(r^-_{N+\frac{1}{2}})^2.
        \end{split}
    \end{equation*}
Furthermore, we observe that
    \begin{equation*}
        \sum\limits_{i=1}^N \left(r,r_x\right)_{I_i} = \frac{1}{2}\sum\limits_{i=1}^N r^2 \big|_{x^+_{i-\frac{1}{2}}}^{x^-_{i+\frac{1}{2}}} = -\frac{1}{2}(r^+_{\frac{1}{2}})^2+\frac{1}{2}(r^-_{N+\frac{1}{2}})^2 +\frac{1}{2} \sum\limits_{i=1}^{N-1}\big((r^-_{i+\frac{1}{2}})^2 -  (r^+_{i+\frac{1}{2}})^2\big).
    \end{equation*}
As a consequence, using $r^-_{N+\frac{1}{2}}=0$, we get
     \begin{align}\label{E1+rr_x}
        \nonumber\mathcal{E}_1+\sum\limits_{i=1}^N \left(r,r_x\right)_{I_i} &= -\hat f_{\frac{1}{2}}u^+_{\frac{1}{2}} + \hat f_{N+\frac{1}{2}}u^-_{N+\frac{1}{2}}  -\sum\limits_{i=1}^{N-1} \hat f_{i+\frac{1}{2}}\llbracket u \rrbracket_{i+\frac{1}{2}} +\frac{1}{2}(r^+_{\frac{1}{2}})^2 -\frac{1}{2}(r^-_{N+\frac{1}{2}})^2 \\&\qquad \nonumber+\sum\limits_{i=1}^{N-1}\left(\frac{1}{2}(r^-_{i+\frac{1}{2}})^2 +  \frac{1}{2}(r^+_{i+\frac{1}{2}})^2 -  r^+_{i+\frac{1}{2}}r^-_{i+\frac{1}{2}}\right)\\&
        = -\hat f_{\frac{1}{2}}u^+_{\frac{1}{2}} + \hat f_{N+\frac{1}{2}}u^-_{N+\frac{1}{2}}  -\sum\limits_{i=1}^{N-1} \hat f_{i+\frac{1}{2}}\llbracket u \rrbracket_{i+\frac{1}{2}} +\frac{1}{2}(r^+_{\frac{1}{2}})^2 + \frac{1}{2}\sum\limits_{i=1}^{N-1} \llbracket r \rrbracket_{i+\frac{1}{2}}^2.
    \end{align}
Let us define $F(u) = \int^u f(u)\,du$. Hence we get
    \begin{equation}\label{F_u}
        \sum\limits_{i=1}^{N} \left(f(u), u_x\right)_{I_i} = \sum\limits_{i=1}^{N} F(u)|_{u^+_{i-\frac{1}{2}}}^{u^-_{i+\frac{1}{2}}} = - \sum\limits_{i=1}^{N-1} \llbracket F(u)\rrbracket_{i+\frac{1}{2}} - F(u)_{\frac{1}{2}} + F(u)_{N+\frac{1}{2}}.
    \end{equation}
Taking into account \eqref{E1+rr_x} and \eqref{F_u} in \eqref{perturbation4}, we end up with
\begin{equation}\label{B_nlinear}
        \begin{split}
         \mathcal{B}(u,p, & q,r; u,-q+p+r,u+r,r-p) \\
         =& \sum\limits_{i=1}^N\Big[\left(u_t,u\right)_{I_i}-\left(p,q\right)_{I_i} + \left(p,p\right)_{I_i} +  \left(\Delta_{\frac{(\alpha-2)}{2}}q,q\right)_{I_i}   -\left(\Delta_{\frac{(\alpha-2)}{2}}q,p\right)_{I_i}- \left(\Delta_{\frac{(\alpha-2)}{2}}q,r\right)_{I_i}  
         \\ & \quad + \left(q,u\right)_{I_i} + \left(q,r\right)_{I_i} 
       +\left(r,r\right)_{I_i}
           \Big]+  F(u)_{\frac{1}{2}}- F(u)_{N+\frac{1}{2}} + \sum\limits_{i=1}^{N-1} \llbracket F(u)\rrbracket_{i+\frac{1}{2}} 
           \\& \quad -\hat f_{\frac{1}{2}}u^+_{\frac{1}{2}}
            + \hat f_{N+\frac{1}{2}}u^-_{N+\frac{1}{2}} -  \sum\limits_{i=1}^{N-1}\hat f_{i+\frac{1}{2}} \llbracket u\rrbracket_{i+\frac{1}{2}}  +
           \frac{1}{2}\sum\limits_{i=1}^{N-1} \llbracket r \rrbracket_{i+\frac{1}{2}}^2 +\frac{1}{2}(r^+_{\frac{1}{2}})^2.
        \end{split}
    \end{equation}
    Clearly, if $(u_h, p_h, q_h, r_h)$ is a solution of scheme \eqref{LDGscheme1}-\eqref{Boundaryfluxes2}, then $$\mathcal{B}(u_h, p_h, q_h, r_h; v, w, z, s)=0 \text{ for any }(v, w, z, s) \in V_h^k.$$ By using the Young's inequality in \eqref{B_nlinear} and Proposition \ref{fracemb}, we have
    \begin{equation}\label{perturbation5}
        \begin{split}
         \left((u_h)_t,u_h\right)_{L^2(\Omega)}&  +\sum\limits_{i=1}^{N-1} \llbracket F(u_h)\rrbracket_{i+\frac{1}{2}} + F(u_h)_{\frac{1}{2}} - F(u_h)_{N+\frac{1}{2}}+ \norm{p_h}_{L^2(\Omega)}^2 + \left(\Delta_{\frac{(\alpha-2)}{2}}q_h,q_h\right)_{I_i}\\&+
         \norm{r_h}^2_{L^2(\Omega)} 
           - (\hat f_h)_{\frac{1}{2}}(u_h)^+_{\frac{1}{2}} + (\hat f_h)_{N+\frac{1}{2}}(u_h)^-_{N+\frac{1}{2}} - \sum\limits_{i=1}^{N-1}  (\hat f_h)_{i+\frac{1}{2}}\llbracket u_h \rrbracket_{i+\frac{1}{2}}\\&+\frac{1}{2}((r_h)^+_{\frac{1}{2}})^2  + \frac{1}{2}\sum\limits_{i=1}^{N-1} \llbracket r_h \rrbracket_{i+\frac{1}{2}}^2\\ \leq & \varepsilon \norm{p_h}_{L^2(\Omega)}^2 +c_1(\varepsilon)\norm{q_h}_{L^2(\Omega)}^2 +\varepsilon\norm{r_h}_{L^2(\Omega)}^2 +c_2(\varepsilon) \norm{u_h}_{L^2(\Omega)}^2,
        \end{split}
    \end{equation}
where we have used the Lemma \ref{kilbas}, and $c_i$, $i=1,2$ are constants and $\varepsilon>0$.
Given that $\hat f(u_h^-,u_h^+)$ is a monotone flux, exhibiting non-decreasing behavior in its first argument and non-increasing behavior in its second argument, we have $$\llbracket F(u_h)\rrbracket_{i+\frac{1}{2}} -  (\hat f_h)_{i+\frac{1}{2}} \llbracket u_h\rrbracket_{i+\frac{1}{2}} > 0 \text{ for } i=1,2,\ldots,N-1.$$ With the help of Lemma \ref{Inner_frac}, Lemma \ref{fpfred} and Proposition \ref{fracemb}, the estimate \eqref{perturbation5} reduces to
  \begin{equation*}
        \begin{split}
        \left((u_h)_t,u_h\right)_{L^2(\Omega)} + F(u_h)_{\frac{1}{2}} - F(u_h)_{N+\frac{1}{2}}
       -(\hat f_h)_{\frac{1}{2}}u^+_{\frac{1}{2}} +  (\hat f_h)_{N+\frac{1}{2}}(u_h)^-_{N+\frac{1}{2}}\leq  C\norm{u_h}_{L^2(\Omega)}^2,
        \end{split}
    \end{equation*}
where we have omitted the positive terms on the left-hand side. Imposing the boundary conditions, we obtain
 \begin{equation}\label{perturbation6}
        \begin{split}
         \frac{1}{2}\frac{d}{dt}\norm{u_h}^2_{L^2(\Omega)}\leq  C\norm{u_h}_{L^2(\Omega)}^2.
        \end{split}
    \end{equation}
Using the Gronwall’s inequality, \eqref{perturbation6} yields the estimate
 \begin{equation*}
    \norm{u_h(\cdot,T)}_{L^2(\Omega)} \leq e^{-2CT}\norm{u^0_h}_{L^2(\Omega)},
    \end{equation*}
    where the constant $C$ is independent of $u_h$.
Hence the $L^2$-stability is established.
    \end{proof}

\begin{remark}
Following the Remark \ref{remk1}, we can extend our stability analysis for an alternate choice of Neumann boundary condition consider in \eqref{Boundaryfluxes2}. More precisely, let us consider the boundary data $U_x(a,t) = U_x(b,t)$. We observe that setting $\hat{r}_{\frac{1}{2}} = r^+_{\frac{1}{2}}$ and  $\hat{r}_{N+\frac{1}{2}} = r^-_{N+\frac{1}{2}}$ implies $r^+_{\frac{1}{2}} = r^-_{N+\frac{1}{2}}$, resulting in $\frac{1}{2}(r^+_{\frac{1}{2}})^2 - \frac{1}{2}(r^-_{N+\frac{1}{2}})^2 = 0$. From \eqref{E1+rr_x}, the stability estimate follows. Alternatively, if we consider the non-homogeneous boundary data $U_x(b,t) = \theta(t)$, we can choose flux $\hat{r}_{N+\frac{1}{2}} = \theta(t)$, leading to $\frac{1}{2}(r^-_{N+\frac{1}{2}})^2 = \frac{1}{2}\theta(t)^2$ and subsequently, the stability estimate can be derived from \eqref{E1+rr_x}.
These alternative choices of flux functions enrich our understanding of the stability under various boundary conditions.
\end{remark}

\subsection{Error analysis}
We would like to derive an estimate for the $L^2$-error of the approximate solution $u_h$ by using the aforementioned stability result. This process unfolds in two steps: initially focusing on the linear convection term $f(u) = \lambda u$, $\lambda\geq0$ (see \cite{yan2002local}), and then  utilizing the linear case estimate to establish the error estimate for the non-linear case \cite{xu2007error}. In particular, we can take $\lambda\leq 0$ with the modification in the numerical flux by setting $\hat p_h = p_h^-$ and $\hat u_h = u^+_h$. It is worthwhile to mention that error analysis for \eqref{fkdv} differs from \cite{yan2002local, xu2007error} in choosing the test functions and estimating the fractional term.

We begin by defining the projection operators into $V_h^k$ (see \cite{xu2014discontinuous}). For any sufficiently smooth function $g$, we define
\begin{align}\label{projectionprop}
    \int_{I_i}(\mathcal{P}^\pm g(x) - g(x))y(x)\,dx &= 0 \quad \forall ~y\in P^{k-1}(I_i), \quad i=1,2,\cdots,N, ~\text{ and } (\mathcal{P}^\pm g)^\pm_{i\mp\frac{1}{2}} = g(x^\pm_{i\mp\frac{1}{2}}), \nonumber\\
     \int_{I_i}(\mathcal{P} g(x) - g(x))y(x)\,dx &= 0 \quad \forall~y\in P^k(I_i), \quad i=1,2,\cdots,N,
\end{align}
where $\mathcal{P}^\pm$ are special projection (Gauss-Radau projection) operators and $\mathcal{P}$ is the standard $L^2$ projection. Let $U$ be an exact solution of \eqref{fkdv} and $u_h$ be an approximate solution obtained by the LDG scheme \eqref{LDGscheme1}-\eqref{Boundaryfluxes2}. For the writing convenience, we denote some notations corresponding to differences between projections and certain approximations or exact solutions as follows
    $$\mathcal{P}^-_hu = \mathcal{P}^-U-u_h ,\quad \mathcal{P}^+_hp = \mathcal{P}^+P-p_h,\quad \mathcal{P}_hq = \mathcal{P}Q- q_h, \quad \mathcal{P}_h^+r = \mathcal{P}^+R-r_h,$$
    and 
     $$\mathcal{P}^-_eU = \mathcal{P}^-U-U ,\quad \mathcal{P}^+_eP = \mathcal{P}^+P-P,\quad \mathcal{P}_eQ = \mathcal{P}Q-Q , \quad \mathcal{P}_e^+R = \mathcal{P}^+R-R,$$
     where subscripts $h$ and $e$ indicate that the difference with approximate and exact solution respectively.
From the $L^2$-projections and Gauss-Radau projection defined above, it is easy to show the following interpolation estimate \cite{Ciarlet}
\begin{equation}\label{interpest}
    \norm{\mathcal Q_eg}_{L^2(\Omega)} + h\norm{\mathcal{Q}_e g}_{L^{\infty}(\Omega)} + h^{\frac{1}{2}}\norm{\mathcal{Q}_eg}_{\Gamma_h}\leq Ch^{k+1},
\end{equation}
     where $\mathcal{Q}_e = \mathcal{P}_e$ or $\mathcal{P}^\pm_e$ and constant $C$ only depends on $g$.
     
If $f(u)=\lambda u$, then $\hat f_h = \hat f(u^-_h,u^+_h) = \frac{\lambda}{2}(u^-_h+u^+_h)-\frac{|\lambda|}{2}(u^+_h-u^-_h)$. This is the Lax-Friedrichs flux with $|f'(u)|=|\lambda|$ and the scheme \eqref{LDGscheme1} reduces to 
\begin{equation}\label{LDGscheme2}
\begin{split}
     \left((u_h)_t,v\right)_{I_i} & = \left(\lambda u_h + p_h,v_x\right)_{I_i} - \left(\hat f_h v + \hat p_h v\right)|_{x_{i-\frac{1}{2}}^+}^{x_{i+\frac{1}{2}}^-},\\
              \left(p_h,w\right)_{I_i} &= \left(\Delta_{\frac{(\alpha-2)}{2}}q_h,w\right)_{I_i},\\
              \left(q_h,z\right)_{I_i}&= -\left(r_h,z_x\right)_{I_i} + \left(\hat r_h z\right)|_{x_{i-\frac{1}{2}}^+}^{x_{i+\frac{1}{2}}^-},\\
              \left(r_h,s\right)_{I_i}&= -\left(u_h,s_x\right)_{I_i} +\left(\hat u_h s\right)|_{x_{i-\frac{1}{2}}^+}^{x_{i+\frac{1}{2}}^-},\\
              \left(u^0_h,v\right)_{I_i} &= \left( \mathcal P^-U_0,v\right)_{I_i},
\end{split}
\end{equation}
and the associated compact form can be represented as
\begin{align}\label{F_linear}
           \nonumber \mathcal{B}_\lambda(u,p,q,r;& v,w,z,s) = \sum\limits_{i=1}^N\Big[\left(u_t,v\right)_{I_i} - \left(\lambda u + p,v_x\right)_{I_i} +\left(p,w\right)_{I_i} -  \left(\Delta_{\frac{(\alpha-2)}{2}}q,w\right)_{I_i}
         + \left(q,z\right)_{I_i} \\&+ \left(r,z_x\right)_{I_i}  
          +\left(r,s\right)_{I_i} +\left(u,s_x\right)_{I_i} 
       + \left((\hat f + \hat p )v\right)|_{x_{i-\frac{1}{2}}^+}^{x_{i+\frac{1}{2}}^-} -\left(\hat r z\right)|_{x_{i-\frac{1}{2}}^+}^{x_{i+\frac{1}{2}}^-} -
        \left(\hat u s\right)|_{x_{i-\frac{1}{2}}^+}^{x_{i+\frac{1}{2}}^-}\Big],
    \end{align}
 where numerical fluxes $\hat u$, $\hat p$, $\hat r$ are defined by \eqref{fluxes} and boundary fluxes by \eqref{Boundaryfluxes2}. 
Let us consider equations \eqref{B_nlinear} with $f(u)=\lambda u, \lambda\geq 0$ and $\hat f(u^-,u^+) = \frac{\lambda}{2}(u^-+u^+)-\frac{|\lambda|}{2}(u^+-u^-) = \lambda u^-$. As a consequence, the compact form transforms to 
\begin{align}\label{B_linear}
        \nonumber \mathcal{B}_{\lambda}(u,p,&q,r;u,-q+p+r,u+r,r-p)\\& \nonumber = \sum\limits_{i=1}^N\Big[\left(u_t,u\right)_{I_i}-\left(p,q\right)_{I_i} + \left(p,p\right)_{I_i} +  \left(\Delta_{\frac{(\alpha-2)}{2}}q,q\right)_{I_i}- \left(\Delta_{\frac{(\alpha-2)}{2}}q,p\right)_{I_i} \\&\nonumber - \left(\Delta_{\frac{(\alpha-2)}{2}}q,r\right)_{I_i}  + \left(q,u\right)_{I_i} + \left(q,r\right)_{I_i}
       +\left(r,r\right)_{I_i}
           \Big]+  F(u)_{\frac{1}{2}}- F(u)_{N+\frac{1}{2}} + \sum\limits_{i=1}^{N-1} \llbracket F(u)\rrbracket_{i+\frac{1}{2}} \\&  -\hat f_{\frac{1}{2}}u^+_{\frac{1}{2}} + \hat f_{N+\frac{1}{2}}u^-_{N+\frac{1}{2}} -  \sum\limits_{i=1}^{N-1}\hat f_{i+\frac{1}{2}} \llbracket u\rrbracket_{i+\frac{1}{2}} +
           \frac{1}{2}\sum\limits_{i=1}^{N-1} \llbracket r \rrbracket_{i+\frac{1}{2}}^2 +\frac{1}{2}(r^+_{\frac{1}{2}})^2.
    \end{align}

The following lemma will be instrumental for the subsequent error analysis. 
\begin{lemma}[See Lemma 2.14 in \cite{xu2014discontinuous}]\label{Negfracapprox}
    Let $U\in C_c^\infty (\Omega)$ and $u_h$ be an approximation of $U$ in $V_h^k(\Omega)$. Further assume that $(U-u_h,v)_{I_i} = 0,~i=1,2,\cdots,N$, $\forall v\in V_h^k(\Omega)$. For $0\leq s<1$, the following estimate holds
    \begin{equation}\label{fraclapbound_app}
        \norm{\Delta_{-s/2}U - \Delta_{-s/2}u_h}_{L^2(\Omega)} \leq Ch^{k+1},
    \end{equation}
    where $C$ is a constant independent of $h$.
\end{lemma}
\begin{theorem}\label{Lemmalin_err}
    Let $U$ be a sufficiently smooth exact solution of equation \eqref{fkdv} with $f(u)=\lambda u$, $\lambda\geq0$. Let $u_h$ be an approximate solution obtained by the LDG scheme \eqref{F_linear}. Then the following error estimate holds
    \begin{equation}\label{errest1}
        \norm{U-u_h}_{L^2(\Omega)}\leq C h^{k+1},
    \end{equation}
    provided $h$ is sufficiently small, and $C$ is a constant independent of $h$.
\end{theorem}
\begin{proof}
 We define 
    \begin{align*}
    P = \Delta_{\frac{(\alpha-2)}{2}}Q, \qquad Q = R_x , \qquad R= U_x.
    \end{align*}
Then $U,$ $P,$ $Q$ and $R$ satisfy  \eqref{F_linear}, i.e.
    \begin{equation*}
        \mathcal{B}_\lambda(U,P,Q,R;v,w,z,s)=0, \quad \text{ for all } v,w,z,s\in V_h^k.
    \end{equation*}
Since $f(u)=\lambda u$, the equation \eqref{fkdv} becomes linear and $\mathcal{B}_\lambda$ transforms into a bilinear operator. Assume that $(u_h,p_h,q_h,r_h) \in V_h^k$ satisfies \eqref{LDGscheme2}. Then we have
$$\mathcal{B}_{\lambda}(u_h,p_h,q_h,r_h;v,w,z,s)=0, \quad \text{ for all } (v,w,z,s)\in V_h^k.$$ 
Hence we formulate the error equation
    \begin{equation*}
        \mathcal{B}_\lambda(U-u_h,P-p_h,Q-q_h,R-r_h;v,w,z,s)=0 \quad \text{ for all } (v,w,z,s)\in V_h^k.
    \end{equation*}
The error can be represented as $U-u_h =\mathcal{P}^-_hu -  \mathcal{P}^-_eU$. Consequently, for all $(v,w,z,s)\in V_h^k$, we have the following identity
    \begin{equation}\label{errorequation}
    \begin{split}
        \mathcal{B}_\lambda(\mathcal{P}^-_hu,\mathcal{P}^+_hp,\mathcal{P}_hq ,\mathcal{P}_h^+r;v,w,z,s)=\mathcal{B}_\lambda(\mathcal{P}^-_eU,\mathcal{P}^+_eP,\mathcal{P}_eQ ,\mathcal{P}_e^+R;v,w,z,s).
    \end{split}
    \end{equation}
We choose $(v,w,z,s) = (\mathcal{P}^-_hu, -\mathcal{P}_hq + \mathcal{P}^+_hp+\mathcal{P}_h^+r, \mathcal{P}^-_hu+ \mathcal{P}_h^+r, \mathcal{P}_h^+r - \mathcal{P}^+_hp) $ in equation \eqref{errorequation}. Taking into account the bilinear form \eqref{B_linear}, the left hand side of the identity \eqref{errorequation} can be represented as 
    \begin{align}\label{bilinear1}
            \nonumber \mathcal{B}_\lambda(\mathcal{P}^-_hu&,\mathcal{P}^+_hp,\mathcal{P}_hq ,\mathcal{P}_h^+r;\mathcal{P}^-_hu, -\mathcal{P}_hq + \mathcal{P}^+_hp + \mathcal{P}_h^+r, \mathcal{P}^-_hu+ \mathcal{P}_h^+r, \mathcal{P}_h^+r - \mathcal{P}^+_hp)\\ \nonumber
              =& \sum\limits_{i=1}^N\Big[\left((\mathcal{P}^-_hu)_t,\mathcal{P}^-_hu\right)_{I_i}-\left(\mathcal{P}^+_hp,\mathcal{P}_hq\right)_{I_i} + \left(\mathcal{P}^+_hp,\mathcal{P}^+_hp\right)_{I_i} +  \left(\Delta_{\frac{(\alpha-2)}{2}}\mathcal{P}_hq,\mathcal{P}_hq\right)_{I_i}  \\&\quad  \nonumber- \left(\Delta_{\frac{(\alpha-2)}{2}}\mathcal{P}_hq,\mathcal{P}^+_hp\right)_{I_i}-\left(\Delta_{\frac{(\alpha-2)}{2}}\mathcal{P}_hq,\mathcal{P}^+_hr\right)_{I_i} + \left(\mathcal{P}_hq,\mathcal{P}^-_hu\right)_{I_i}+ \left(\mathcal{P}_hq,\mathcal{P}_h^+r\right)_{I_i}\\&\quad \nonumber+\left(\mathcal{P}_h^+r,\mathcal{P}_h^+r\right)_{I_i}
           \Big] + \sum\limits_{i=1}^{N-1} \Big(\llbracket F(\mathcal{P}^-_hu)\rrbracket_{i+\frac{1}{2}} - \hat f_{i+\frac{1}{2}} \llbracket \mathcal{P}^-_hu\rrbracket_{i+\frac{1}{2}}\Big)  + F(\mathcal{P}^-_hu)_{\frac{1}{2}}- F(\mathcal{P}^-_hu)_{N+\frac{1}{2}} \\& \quad  -\hat f_{\frac{1}{2}}(\mathcal{P}^-_hu)^+_{\frac{1}{2}}  + \hat f_{N+\frac{1}{2}}(\mathcal{P}^-_hu)^-_{N+\frac{1}{2}} + \frac{1}{2}\sum\limits_{i=1}^{N-1} \llbracket \mathcal{P}_h^+r \rrbracket_{i+\frac{1}{2}}^2 +\frac{1}{2}((\mathcal{P}_h^+r)^+_{\frac{1}{2}})^2,
    \end{align}
and from equation \eqref{F_linear}, the right hand side of the identity \eqref{errorequation} can be written as
\begin{align}\label{bilinear2}
       \nonumber \mathcal{B}_\lambda(\mathcal{P}^-_eU,&\mathcal{P}^+_eP,\mathcal{P}_eQ ,\mathcal{P}_e^+R;\mathcal{P}^-_hu, -\mathcal{P}_hq + \mathcal{P}^+_hp + \mathcal{P}_h^+r, \mathcal{P}^-_hu+ \mathcal{P}_h^+r, \mathcal{P}_h^+r - \mathcal{P}^+_hp) \\& \nonumber
        =\sum\limits_{i=1}^N\Big[\left((\mathcal{P}^-_eU)_t,\mathcal{P}^-_hu\right)_{I_i} - \left(\lambda\mathcal{P}^-_eU + \mathcal{P}^+_eP,(\mathcal{P}^-_hu)_x\right)_{I_i}  +\left(\mathcal{P}^+_eP, -\mathcal{P}_hq + \mathcal{P}^+_hp + \mathcal{P}_h^+r\right)_{I_i} \\&\quad \nonumber -  \left(\Delta_{\frac{(\alpha-2)}{2}}(\mathcal{P}_eQ ), -\mathcal{P}_hq + \mathcal{P}^+_hp + \mathcal{P}_h^+r\right)_{I_i}
         + \left(\mathcal{P}_eQ ,\mathcal{P}^-_hu+ \mathcal{P}_h^+r\right)_{I_i}\\&\quad \nonumber + \left(\mathcal{P}_e^+R,(\mathcal{P}^-_hu+ \mathcal{P}_h^+r)_x\right)_{I_i}
          +\left(\mathcal{P}_e^+R,\mathcal{P}_h^+r - \mathcal{P}^+_hp\right)_{I_i}  +\left(\mathcal{P}^-_eU,(\mathcal{P}_h^+r - \mathcal{P}^+_hp)_x\right)_{I_i} \Big]\\&\quad
      + \mathcal{IF}(\mathcal{P}^-_eU,\mathcal{P}^+_eP,\mathcal{P}_e^+R;\mathcal{P}^-_hu,\mathcal{P}^-_hu+ \mathcal{P}_h^+r,\mathcal{P}_h^+r - \mathcal{P}^+_hp),
\end{align}
    where numerical flux at interfaces $\mathcal{IF}$ is given by \eqref{Interfacesflux} with $\hat f_h =\hat f(u^-_h,u^+_h) = \frac{\lambda}{2}(u^-_h+u^+_h)-\frac{|\lambda|}{2}(u^+_h-u^-_h)$. We observe that
      \begin{align*}
          (\mathcal{P}^-_hu)_x \in P^{k-1}(I_i), &\quad \mathcal{P}^-_hu+ \mathcal{P}_h^+r \in P^k(I_i),\\  (\mathcal{P}^-_hu+ \mathcal{P}_h^+r)_x \in P^{k-1}(I_i), &\quad
          (\mathcal{P}_h^+r - \mathcal{P}^+_hp)_x \in P^{k-1}(I_i),
      \end{align*}
for all $i=1,2,\cdots,N$.  Since $U$, $P$, $Q$ and $R$ are sufficiently smooth, applying the projection operators introduced in \eqref{projectionprop}, we have 
      \begin{align*}
           \left(\lambda\mathcal{P}^-_eU + \mathcal{P}^+_eP, (\mathcal{P}^-_hu)_x\right)_{I_i} = 0,  &\quad \left(\mathcal{P}_eQ ,\mathcal{P}^-_hu+ \mathcal{P}_h^+r\right)_{I_i} = 0, \\\quad \left(\mathcal{P}_e^+R,(\mathcal{P}^-_hu+ \mathcal{P}_h^+r)_x\right)_{I_i}=0, & \quad \left(\mathcal{P}^-_eU,(\mathcal{P}_h^+r - \mathcal{P}^+_hp)_x\right)_{I_i} =0,
      \end{align*}
      and at the interfaces
      \begin{align*}
         (\lambda\mathcal{P}^-_eU)^-_{i+\frac{1}{2}} = 0, \qquad (\mathcal{P}^+_eP)^+_{i-\frac{1}{2}} =0, \qquad (\mathcal{P}_e^+R)^+_{i-\frac{1}{2}} =0,
      \end{align*}
     and the approximation theory on the point values \cite[Section 3.2]{Ciarlet} yields
       \begin{align*}
         (\lambda\mathcal{P}^-_eU)^-_{i-\frac{1}{2}} \leq Ch_{i}^{k+1}, \qquad (\mathcal{P}^+_eP)^+_{i+\frac{1}{2}} \leq Ch_{i}^{k+1}, \qquad (\mathcal{P}_e^+R)^+_{i+\frac{1}{2}} \leq Ch_{i}^{k+1},
      \end{align*}
With the help of the above estimate, the bilinear form \eqref{bilinear2} can be estimated as
     \begin{align}\label{Bilinear3}
 \mathcal{B}_\lambda(\mathcal{P}^-_eU&,\mathcal{P}^+_eP,\mathcal{P}_eQ ,\mathcal{P}_e^+R;\mathcal{P}^-_hu, -\mathcal{P}_hq + \mathcal{P}^+_hp + \mathcal{P}_h^+r, \mathcal{P}^-_hu+ \mathcal{P}_h^+r, \mathcal{P}_h^+r - \mathcal{P}^+_hp) \nonumber\\
             & \leq \sum\limits_{i=1}^N\Big[\left((\mathcal{P}^-_eU)_t,\mathcal{P}^-_hu\right)_{I_i}+\left(\Delta_{\frac{(\alpha-2)}{2}}(\mathcal{P}_eQ ) - (\mathcal{P}^+_eP), \mathcal{P}_hq - \mathcal{P}^+_hp - \mathcal{P}_h^+r\right)_{I_i} \nonumber\\ &\qquad  \qquad+\left(\mathcal{P}_e^+R,\mathcal{P}_h^+r - \mathcal{P}^+_hp\right)_{I_i} \Big] +C({\Omega})h^{2k+2}.
\end{align}
Using the Young's inequality and also the triangle inequality, Lemma \ref{Negfracapprox} and approximation theorem in \cite[Section 3.2]{Ciarlet} in equation \eqref{Bilinear3}, we get
    \begin{align}\label{Bilinear4}
           \nonumber  \mathcal{B}_\lambda(\mathcal{P}^-_eU&,\mathcal{P}^+_eP,\mathcal{P}_eQ ,\mathcal{P}_e^+R;\mathcal{P}^-_hu,-\mathcal{P}_hq + \mathcal{P}^+_hp + \mathcal{P}_h^+r, \mathcal{P}^-_hu+ \mathcal{P}_h^+r, \mathcal{P}_h^+r - \mathcal{P}^+_hp)\\ \nonumber
             & \leq \left((\mathcal{P}^-_eU)_t,\mathcal{P}^-_hu\right)_{L^2(\Omega)} + c_1(\epsilon)h^{2k+2} + \epsilon\left(\norm{\mathcal{P}_hq}^2_{L^2(\Omega)}+\norm{\mathcal{P}^+_hp}^2_{L^2(\Omega)}+\norm{\mathcal{P}^+_hr}^2_{L^2(\Omega)}\right)\\ &\quad + c_2(\epsilon)h^{2k+2} +\epsilon\left(\norm{\mathcal{P}^+_hr}^2_{L^2(\Omega)}+\norm{\mathcal{P}^+_hp}^2_{L^2(\Omega)}\right) + C({\Omega})h^{2k+2}.
    \end{align}
 where $c_1$ and $c_2$ are constants and $\epsilon>0$. Equation \eqref{errorequation} implies
\begin{align}\label{Bilinear45}
           \nonumber \mathcal{B}_\lambda(\mathcal{P}^-_hu&,\mathcal{P}^+_hp,\mathcal{P}_hq ,\mathcal{P}_h^+r;\mathcal{P}^-_hu, -\mathcal{P}_hq + \mathcal{P}^+_hp + \mathcal{P}_h^+r, \mathcal{P}^-_hu+ \mathcal{P}_h^+r, \mathcal{P}_h^+r - \mathcal{P}^+_hp)\\ \nonumber
             & \leq \left((\mathcal{P}^-_eU)_t,\mathcal{P}^-_hu\right)_{L^2(\Omega)} + c_1(\epsilon)h^{2k+2} + \epsilon\left(\norm{\mathcal{P}_hq}^2_{L^2(\Omega)}+\norm{\mathcal{P}^+_hp}^2_{L^2(\Omega)}+\norm{\mathcal{P}^+_hr}^2_{L^2(\Omega)}\right)\\ &\quad + c_2(\epsilon)h^{2k+2} +\epsilon\left(\norm{\mathcal{P}^+_hr}^2_{L^2(\Omega)}+\norm{\mathcal{P}^+_hp}^2_{L^2(\Omega)}\right) + C({\Omega})h^{2k+2}.
    \end{align}
Incorporating \eqref{bilinear1} in \eqref{Bilinear45}, and omitting the positive terms on the left-hand side, yields
    \begin{align}\label{bilinear5}
            \nonumber&\left((\mathcal{P}^-_hu)_t,\mathcal{P}^-_hu\right)_{L^2(\Omega)}+\norm{\mathcal{P}^+_hr}^2_{L^2(\Omega)}+\norm{\mathcal{P}^+_hp}^2_{L^2(\Omega)} +  \left(\Delta_{\frac{(\alpha-2)}{2}}\mathcal{P}_hq,\mathcal{P}_hq\right)_{L^2(\Omega)} \\ \nonumber
            &\quad\leq \left((\mathcal{P}^-_eU)_t,\mathcal{P}^-_hu\right)_{L^2(\Omega)} + C({\Omega})h^{2k+2} +\epsilon\norm{\mathcal{P}^+_hp}^2_{L^2(\Omega)}+c_1(\epsilon)\norm{\mathcal{P}_hq}^2_{L^2(\Omega)}\\&\qquad+\epsilon\norm{\mathcal{P}^+_hr}^2_{L^2(\Omega)} + c_2(\epsilon) \norm{\mathcal{P}^-_hu}^2_{L^2(\Omega)},
    \end{align}
where we have again used the Young's inequality, Lemma \ref{Inner_frac} and  Lemma \ref{fpfred}. Furthermore, we have 
     \begin{equation*}
        \begin{split}
            \frac{1}{2}\frac{d}{dt}\norm{\mathcal{P}^-_hu}^2_{L^2(\Omega)}\leq \left((\mathcal{P}^-_eU)_t,\mathcal{P}^-_hu\right)_{L^2(\Omega)} + C({\Omega})h^{2k+2} + C \norm{\mathcal{P}^-_hu}^2_{L^2(\Omega)}.
        \end{split}
    \end{equation*}
   Since $\norm{\mathcal{P}^-_hu(\cdot,0)}^2_{L^2(\Omega)}=0$, using the standard approximation theory associated to the projection operator \cite[Section 3.2]{Ciarlet} and Gronwall’s inequality, we have
    \begin{equation*}
        \norm{U-u_h}_{L^2(\Omega)}\leq C h^{k+1}.
    \end{equation*}
    This completes the proof.
\end{proof}

We extend the previous error estimate to accommodate a more general non-linear convection term. In the process of formulating the error estimate for the nonlinear fractional KdV equation \eqref{fkdv}, we introduce few lemmas concerning the measure between physical flux $f$ and numerical flux $\hat f_h$. Subsequently, we present our main result.

\begin{lemma}[See Lemma 3.1 in \cite{zhang2004error}]\label{Shulemma1}
    Let $\xi\in L^2(\Omega)$ be any piecewise smooth function. On each interface of elements and the boundary point, we define
    \begin{align*}
        \beta(\hat f;\xi):= \beta(\hat f;\xi^-,\xi^+) = \begin{cases}
            \llbracket\xi\rrbracket^{-1}(f(\{\!\!\{\xi\}\!\!\}) - \hat f(\xi)),\qquad &if~\llbracket \xi \rrbracket \neq 0,\\
            \frac{1}{2}|f'(\{\!\!\{\xi\}\!\!\})|, \qquad &if~\llbracket \xi \rrbracket = 0,
        \end{cases}
    \end{align*}
    where $\hat f(\xi) = \hat f(\xi^-,\xi^+)$ is a monotone numerical flux consistent with the given flux $f$. Then $\beta(\hat f;\xi)$ is bounded and nonnegative for any $(\xi^-,\xi^+)\in\R$. Moreover, we have
    \begin{align*}
        \frac{1}{2}|f'(\{\!\!\{\xi\}\!\!\})| &\leq \beta(\hat f;\xi) + C_{\ast}|\llbracket \xi \rrbracket|,\\
        -\frac{1}{8}f''(\{\!\!\{\xi\}\!\!\})|\llbracket \xi \rrbracket &\leq \beta(\hat f;\xi) + C_{\ast}|\llbracket \xi \rrbracket|^2.
    \end{align*}
\end{lemma}
For nonlinear flux $f(u)$, we define
\begin{align}
    \nonumber\sum\limits_{i=1}^N \mathcal{F}_i(f;U,u_h;v) &:= \sum\limits_{i=1}^N \int_{I_i} \big(f(U) - f(u_h)\big)v_x\,dx + \sum\limits_{i=1}^N \Big(\big(f(U) - f(\{\!\!\{u_h\}\!\!\})\big)\llbracket v\rrbracket\Big)_{i+\frac{1}{2}}\\
   \label{Nonlinearpart} &\quad  + \sum\limits_{i=1}^N \big((f(\{\!\!\{u_h\}\!\!\}) - \hat f)\llbracket v\rrbracket\big)_{i+\frac{1}{2}}.
\end{align}
\begin{lemma}[See Corollary 3.6 in \cite{xu2007error}]\label{Shulemma2}
    Let the operator $\mathcal{F}_i$ be defined by \eqref{Nonlinearpart}. Then the following estimate holds:
    \begin{align*}
        \nonumber\sum\limits_{i=1}^N \mathcal{F}_i(f;U,u_h;v) &\leq -\frac{1}{4}\beta(\hat f;u_h)\sum\limits_{i=1}^N\llbracket v\rrbracket^2_{i+\frac{1}{2}} + \big(C+C_{\ast}(\norm{v}_{L^{\infty}(\Omega} + h^{-1}\norm{U-u_h}_{L^\infty(\Omega)}^2)\big)\norm{v}_{L^2(\Omega)}^2\\
        &\quad + (C+C_{\ast}h^{-1}\norm{U-u_h}_{L^\infty(\Omega)}^2)h^{2k+1}.
    \end{align*}
\end{lemma}
Similar to \cite{xu2007error}, we deal with the nonlinear flux $f(u)$ by making an \emph{a priori} assumption. Let $h$ be small enough and $k\geq 1$ such that 
\begin{equation}\label{prioriass}
    \norm{U-u_h}_{L^2(\Omega)} \leq h.
\end{equation}
The above assumption is unnecessary for linear flux $f(u) = \lambda u$. 
The rationality of the \textit{a priori} assumption \eqref{prioriass} can be justified as follows. For $k \geq 1$, we can choose $h$ to be small enough such that $Ch^{k + \frac{1}{2}} < \frac{1}{2} h$, where $C$ is a constant. 
Define 
\[
\bar{t} = \sup \left\{ t : \norm{U(t) - u_h(t)}_{L^2(\Omega)} \leq h \right\}.
\]
By continuity, we then have $\norm{U(\bar{t}) - u_h(\bar{t})}_{L^2(\Omega)} = h$ if $\bar{t}$ is finite.
On the other hand, we show in Theorem \ref{NLerror} that
\[
\norm{U - u_h}_{L^2(\Omega)} \leq Ch^{k + \frac{1}{2}} < \frac{1}{2} h \quad \text{for all } t \leq T.
\]
This results in a contradiction to Theorem \ref{NLerror} if $\bar{t} < T$. Therefore, $\bar{t} \geq T$, and our \textit{a priori} assumption \eqref{prioriass} is justified.
\begin{corollary}\label{infcoro}
    Suppose the interpolation estimate \eqref{interpest} holds. Then, under the\textit{ a priori }assumption \eqref{prioriass}, we have
    \begin{equation}\label{Linfty_est}
        \norm{U-u_h}_{L^\infty(\Omega)}\leq Ch^{\frac{1}{2}} \quad \text{and} \quad \norm{\mathcal{Q}_h u}_{L^\infty(\Omega)} \leq Ch^{\frac{1}{2}},
    \end{equation}
    where $\mathcal{Q}_h = \mathcal{P}_h$ or $\mathcal{P}^{\pm}_h$.
\end{corollary}

\begin{proof}
    Writing $\mathcal{Q}_h u$ as \( \mathcal{Q}_h u = U - u_h + \mathcal{Q}_e U,\) and applying the interpolation estimate \eqref{interpest}, we get
    \begin{align*}
        \norm{\mathcal{Q}_h u}_{L^2(\Omega)} \leq \norm{U - u_h}_{L^2(\Omega)} + \norm{\mathcal{Q}_e U}_{L^2(\Omega)}
        &\leq h + Ch^{k+1} \leq Ch.
    \end{align*}
    Using the inverse inequality $\norm{u_h}_{L^\infty(\Omega)}\leq h^{-\frac{1}{2}}\norm{u_h}_{L^2(\Omega)}$, we obtain
    \begin{equation*}
        \norm{\mathcal{Q}_h u}_{L^\infty(\Omega)} \leq 
        Ch^{-\frac{1}{2}} \norm{\mathcal{Q}_e U}_{L^2(\Omega)} \leq Ch^{\frac{1}{2}}.
    \end{equation*}
    Finally, since $U - u_h = \mathcal{Q}_h u - \mathcal{Q}_e U$, we find
    \begin{align*}
        \norm{U - u_h}_{L^\infty(\Omega)} \leq \norm{\mathcal{Q}_h u}_{L^\infty(\Omega)} + \norm{\mathcal{Q}_e U}_{L^\infty(\Omega)}
        \leq Ch^{\frac{1}{2}} + Ch^{k} \leq Ch^{\frac{1}{2}}.
    \end{align*}
\end{proof}
\begin{theorem}\label{NLerror}
Let $U$ be sufficiently smooth exact solution of \eqref{fkdv} in $\Omega$. Assume that the nonlinear flux $f\in C^3(\Omega)$. Let $u_h$ be an approximate solution obtained by the semi-discrete LDG scheme \eqref{LDGscheme1}-\eqref{Boundaryfluxes2} with auxiliary variables $p_h,q_h,r_h$. Suppose $V_h^k$ is a space of piecewise polynomials of degree $k\geq1$ defined by \eqref{elem_space}. Then for small enough $h$, there holds the following error estimate
\begin{equation}\label{errestimateNL}
    \norm{U-u_h}_{L^2(\Omega)}\leq Ch^{k+\frac{1}{2}},
\end{equation}
where $C$ is a constant depending on time $T>0$, $k$ and the bounds on the derivatives $|f^{(m)}|$, $m=1,2,3$.
\end{theorem}
\begin{proof}
We define $P,Q$ and $R$ as
    \begin{align*}
    P = \Delta_{\frac{(\alpha-2)}{2}}Q, \qquad Q = R_x , \qquad R= U_x.
    \end{align*}
It is straightforward to observe that for any $(v,w,z,s)\in V_h^k$,
    \begin{align*}
        \mathcal{B}(U,P,Q,R;v,w,z,s) = \mathcal{B}(u_h,p_h,q_h,r_h;v,w,z,s) =0,
    \end{align*}
where the compact form $\mathcal{B}$ is defined in 
\eqref{perturbation}.
Taking $\lambda = 0$ in \eqref{F_linear}, we obtain
    \begin{align*}
         \mathcal{B}(U,P,&Q,R;v,w,z,s) - \mathcal{B}(u_h,p_h,q_h,r_h;v,w,z,s)\\
          &=\mathcal{B}_\lambda(U,P,Q,R;v,w,z,s) - \mathcal{B}_\lambda(u_h,p_h,q_h,r_h;v,w,z,s)- \sum\limits_{i=1}^N \mathcal{F}_i(f;U,u_h;v)\\
          &=\mathcal{B}_\lambda(U-u_h,P-p_h,Q-q_h,R-r_h;v,w,z,s) - \sum\limits_{i=1}^N \mathcal{F}_i(f;U,u_h;v)=0.
    \end{align*}
Choosing the test function $(v,w,z,s) = (\mathcal{P}^-_hu, -\mathcal{P}_hq + \mathcal{P}^+_hp+\mathcal{P}_h^+r, \mathcal{P}^-_hu+ \mathcal{P}_h^+r, \mathcal{P}_h^+r - \mathcal{P}^+_hp)$ and using the fact that $U-u_h =\mathcal{P}^-_hu -  \mathcal{P}^-_eU$, the above identity yields
    \begin{align*}
        \mathcal{B}_\lambda(&\mathcal{P}^-_hu,\mathcal{P}^+_hp,\mathcal{P}_hq ,\mathcal{P}_h^+r;\mathcal{P}^-_hu, -\mathcal{P}_hq + \mathcal{P}^+_hp + \mathcal{P}_h^+r, \mathcal{P}^-_hu+ \mathcal{P}_h^+r, \mathcal{P}_h^+r - \mathcal{P}^+_hp)\\
        & = \mathcal{B}_\lambda(\mathcal{P}^-_eU,\mathcal{P}^+_eP,\mathcal{P}_eQ ,\mathcal{P}_e^+R;\mathcal{P}^-_hu, -\mathcal{P}_hq + \mathcal{P}^+_hp + \mathcal{P}_h^+r, \mathcal{P}^-_hu+ \mathcal{P}_h^+r, \mathcal{P}_h^+r - \mathcal{P}^+_hp)\\ &\qquad+ \sum\limits_{i=1}^N \mathcal{F}_i(f;U,u_h;\mathcal{P}^-_hu).
    \end{align*}
    From equation \eqref{B_linear}, \eqref{bilinear5} and Lemma \ref{Shulemma2}, we end up with the following inequality
    \begin{align}
        \nonumber &\left((\mathcal{P}^-_hu)_t,\mathcal{P}^-_hu\right)_{L^2(\Omega)}+\norm{\mathcal{P}^+_hr}^2_{L^2(\Omega)}+\norm{\mathcal{P}^+_hp}^2_{L^2(\Omega)} +  \left(\Delta_{\frac{(\alpha-2)}{2}}\mathcal{P}_hq,\mathcal{P}_hq\right)_{L^2(\Omega)} \\&\nonumber\qquad\qquad+\frac{1}{4}\beta(\hat f;\mathcal{P}^-_hu)\nonumber\sum\limits_{i=1}^N\llbracket \mathcal{P}^-_hu\rrbracket^2_{i+\frac{1}{2}}\\
           \nonumber &\leq \left((\mathcal{P}^-_eU)_t,\mathcal{P}^-_hu\right)_{L^2(\Omega)} + C({\Omega})h^{2k+1} +\epsilon\norm{\mathcal{P}^+_hp}^2_{L^2(\Omega)}+c_1(\epsilon)\norm{\mathcal{P}_hq}^2_{L^2(\Omega)}+\epsilon\norm{\mathcal{P}^+_hr}^2_{L^2(\Omega)} \\&\nonumber\qquad+ c_2(\epsilon)\norm{\mathcal{P}^-_hu}^2_{L^2(\Omega)} + \left(C+C_{\ast}\left(\norm{\mathcal{P}^-_hu}_{L^\infty(\Omega)} + h^{-1}\norm{U-u_h}_{L^\infty(\Omega)}^2\right)\right)\norm{\mathcal{P}^-_hu}_{L^2(\Omega)}^2\\
            \label{nonlinerr}&\qquad + (C+C_{\ast}h^{-1}\norm{U-u_h}_{\infty}^2)h^{2k+1},
    \end{align}
    where $c_1,c_2$ are constants and $\epsilon>0$.
    Utilizing the Corollary \ref{infcoro}, Lemma \ref{Inner_frac} and the positivity of $\beta$ from Lemma \ref{Shulemma1}, we obtain
    \begin{align}\label{nonlin2err}
         &\frac{1}{2}\frac{d}{dt}\norm{\mathcal{P}^-_hu}^2_{L^2(\Omega)}\leq \left((\mathcal{P}^-_eU)_t,\mathcal{P}^-_hu\right)_{L^2(\Omega)} + C({\Omega})h^{2k+1} +C \norm{\mathcal{P}^-_hu}^2_{L^2(\Omega)}.
    \end{align}
     Since $\norm{\mathcal{P}^-_hu(\cdot,0)}^2_{L^2(\Omega)}=0$, again integrating with respect to $t$, we have 
      \begin{equation*}
        \begin{split}
            \frac{1}{2}\norm{\mathcal{P}^-_hu(\cdot,T)}^2_{L^2(\Omega)}\leq \int_0^T\left((\mathcal{P}^-_eU)_t,\mathcal{P}^-_hu\right)_{L^2(\Omega)}\,dt + C(T,\Omega)h^{2k+1} + C \int_0^T \norm{\mathcal{P}^-_hu(\cdot,t)}^2_{L^2(\Omega)}\,dt.
        \end{split}
    \end{equation*}
Using the standard approximation theory of the projection operator \cite{Ciarlet} and Gronwall’s inequality, we obtain the desired estimate \eqref{errestimateNL}.
\end{proof}
\begin{remark}
We have chosen an alternative flux at the interfaces of elements for $\hat u_h$, $\hat p_h$, and $\hat r_h$ to obtain the $L^2$-stability and error estimate for the nonlinear flux $f$. Another choice of fluxes is the central flux, sometimes referred to conservative flux, given by
         \begin{align*}
               &\hat p_h = \{\!\!\{p_h\}\!\!\},\qquad \hat r_h = \{\!\!\{r_h\}\!\!\},\qquad \hat u_h = \{\!\!\{u_h\}\!\!\}.
          \end{align*}
    Both choices of fluxes can be used, and the $L^2$-stability and error estimates can be derived for the central flux with straightforward adjustments to the proofs presented earlier. These adjustments involve the following calculations:
           \begin{equation*}
             \sum\limits_{i=1}^N (up)|_{x_{i-\frac{1}{2}}^+}^{x_{i+\frac{1}{2}}^-} = -u^+_{\frac{1}{2}}p^+_{\frac{1}{2}} + u^-_{N+\frac{1}{2}} p^-_{N+\frac{1}{2}} - \sum\limits_{i=1}^{N-1}\{\!\!\{u\}\!\!\}_{i+\frac{1}{2}} \llbracket p \rrbracket_{i+\frac{1}{2}} -\sum\limits_{i=1}^{N-1}\{\!\!\{p\}\!\!\}_{i+\frac{1}{2}} \llbracket u \rrbracket_{i+\frac{1}{2}}
           \end{equation*}
     and 
           \begin{equation*}
                \sum\limits_{i=1}^N (ur)|_{x_{i-\frac{1}{2}}^+}^{x_{i+\frac{1}{2}}^-} = -u^+_{\frac{1}{2}}r^+_{\frac{1}{2}} + u^-_{N+\frac{1}{2}} r^-_{N+\frac{1}{2}} - \sum\limits_{i=1}^{N-1}\{\!\!\{u\}\!\!\}_{i+\frac{1}{2}} \llbracket r \rrbracket_{i+\frac{1}{2}} -\sum\limits_{i=1}^{N-1}\{\!\!\{r\}\!\!\}_{i+\frac{1}{2}} \llbracket u \rrbracket_{i+\frac{1}{2}}.
          \end{equation*}
          To prove conservative properties for the semi-discrete scheme \eqref{LDGscheme1} with conservative fluxes is a complicated task, see \cite{bona2013conservative} for more details.
          However, our numerical implementation have revealed that the alternative fluxes consistently preform better than the central fluxes in terms of computational efficiency. Although our current approach has yielded a convergence rate of $k+\frac{1}{2}$, it is essential to note that different choices of numerical flux (Godunov or upwind flux) for $\hat f$ can impact the convergence rate significantly. This aspect, along with a fully discrete DG method with homogeneous mixed boundary conditions, will be explored in our future work.
     \end{remark}
   
\section{LDG for multi-dimensional fractional \text{KdV} equation}\label{sec4}
In this section, we extend the LDG scheme to multiple space dimensional equations involving fractional Laplacian in each direction. We consider the following Cauchy problem:
\begin{equation}\label{fkdvhd}
\begin{cases}
     U_t+\sum\limits_{i=1}^d f_i(U)_{x_i}-\sum\limits_{i=1}^d (-\Delta)_{i}^{\alpha_i/2}U_{x_i}=0,  \quad \mathbf{x}:=(x_1,x_2,\cdots,x_d) \in \mathbb{R}^d,~t\in(0,T],\\
     U(\mathbf{x},0) = U_{0_d}(\mathbf{x}), \quad \mathbf{x} \in \mathbb{R}^d,
\end{cases}
\end{equation}
where $T>0$ is fixed, $U_{0_d}$ represents the prescribed periodic initial condition, and $U$ is the unknown solution.
The non-local integro-differential operator $-(-\Delta)_{i}^{\alpha_i/2}$ in \eqref{fkdvhd} denotes the one-dimensional fractional Laplacian acting on the $i$-th component $x_i$, for $i=1,2,\cdots,d$, defined as in Section \ref{sec2}. More precisely
\begin{equation}\label{fracLfrachd}
    -(-\Delta)_i^{\alpha/2}u(\mathbf{x})=\frac{\partial^\alpha}{\partial|x_i|^\alpha}u(\mathbf{x}) = -\frac{{}_{-\infty}D_{x_i}^\alpha u(\mathbf{x}) + {}_{x_i}D_{\infty}^\alpha u(\mathbf{x})}{2\cos\left(\frac{\alpha_i\pi}{2}\right)}, \quad 1<\alpha_i<2.
\end{equation}
The left and right Riemann-Liouville derivatives in the $i$-th component are given by
\begin{align}
    \label{leftfracDhd} {}_{-\infty}D_{x_i}^\alpha u(\mathbf{x}) &= \frac{1}{\Gamma(n-\alpha)}\left(\frac{d}{dx_i}\right)^n\int_{-\infty}^{x_i} (x_i-t)^{n-\alpha_i -1}u(x_1,\cdots,x_{i-1},t,x_{i+1},\cdots,x_d)\,dt,\\
    \label{rightfracDhd}{}_{x_i}D_\infty^\alpha u(\mathbf{x}) &= \frac{1}{\Gamma(n-\alpha_i)}\left(-\frac{d}{dx_i}\right)^n\int^{\infty}_{x_i} (t-x_i)^{n-\alpha_i -1}u(x_1,\cdots,x_{i-1},t,x_{i+1},\cdots,x_d)\,dt,
\end{align}
for $i=1,2,\cdots,d$ and $n-1<\alpha_i<n$, respectively. We decompose the fractional Laplacian \eqref{fracLfrachd} as
\begin{equation}\label{frac_dcmphd}
     -(-\Delta)_i^{\alpha/2}u(\mathbf{x}) = \frac{d}{dx_i}\left(\Delta^i_{\frac{\alpha_i-2}{2}}\frac{du(\mathbf{x})}{dx_i}\right) = \frac{d^2}{d{x_i}^2}\left(\Delta^i_{\frac{\alpha_i-2}{2}}u(\mathbf{x})\right),
\end{equation}
where the fractional integral $\Delta^i_{-s/2}$, $0<s<1$ is given by
\begin{equation}\label{NegFracLhd}
\Delta^i_{-s/2}u(\mathbf{x}) =  \frac{{}_{-\infty}D_{x_i}^{-s}u(\mathbf{x}) + {}_{x_i}D^{-s}_{\infty}u(\mathbf{x})}{2\cos\left(\frac{s\pi}{2}\right)}.
\end{equation}
We discretize the domain $\Omega_d=\left(\prod\limits_{i=1}^d I^i\right)$, where $I^i=[-1,1]$, $i=1,2,\cdots,d$. It is worth noting that assuming the unit box as the domain simplifies the understanding of the method, but it is not essential. We denote the triangulation $\mathcal{T}_d$ of $\Omega_d$ as follows:
\[ \mathcal{T}_d:= \{D^\tau : D^\tau \text{ are non-overlapping polyhedra covering $\Omega_d$ completely, } \tau=1,2,\cdots,\mathcal{N}_d\}, \]
where $\mathcal{N}_d$ is the total number of triangles (we write triangle instead of polyhedra). We do not consider hanging nodes. Associated with the triangulation $\mathcal{T}_d$, we define the broken Sobolev spaces as follows: 
\[ H^1(\Omega_d,{\mathcal{T}}) := \{v:\Omega_d\to \mathbb{R} \text{ such that }\, v|_{D^\tau}\in H^1(D^\tau),\,\tau=1,2,\cdots,\mathcal{N}_d\}. \]
For a function $v\in H^1(\Omega_d,{\mathcal{T}})$, we denote the value of $v$ evaluated from the inside of the triangle $D^\tau$ by $v^{int,\tau}$ and from the outside of the triangle $D^\tau$ by $v^{ext,\tau}$.
We define the finite-dimensional discontinuous piecewise polynomial space $V_d^k\subset H^1(\Omega_d,{\mathcal{T}})$ as
\begin{equation}\label{elem_spacehd}
    V_d^k = \{v:\Omega_d \to \R \text{ such that } v|_{D^\tau}\in P^k (D^\tau),~ \forall \tau=1,2,\cdots,\mathcal{N}_d\}
\end{equation}
and the finite element space $W_d^k$ as the space of tensor product of piecewise polynomials with degree at most $k$ in each variable on every element, i.e.
\begin{equation}\label{elem_spacehdW}
    W_d^k = \{v:\Omega_d \to \R \text{ such that } v\in \mathcal{Q}^k \left(\prod\limits_{i=1}^d I^i_{j_i}\right),~ \forall j_i=1,2,\cdots,N_{i} ~\text{and} ~ i=1,\cdots,d\},
\end{equation}
where $I^i_{j_i} = [x_{i_{j_i-\frac{1}{2}}},x_{i_{j_i+\frac{1}{2}}}]$ is a partition of interval $I^i$ for $j_i = 1,2,\cdots,N_{i} $, and $N_{i}$ is the total number of partition of interval $I^i$ for all $i=1,\cdots,d$, and $\mathcal{Q}^k$ is the space of tensor products of one dimensional polynomials of degree up to $k$. 

Now, we introduce the auxiliary variables $\mathbf{P}=(P_1,P_2,\cdots,P_d)$, $\mathbf{Q}=(Q_1,Q_2,\cdots,Q_d)$, and $\mathbf{R}=(R_1,R_2,\cdots,R_d)$ such that
\begin{align*}
    P_i = \Delta^i_{(\alpha_i-2)/2}Q_i, \qquad Q_i = (R_i)_{x_i} , \qquad R_i= U_{x_i}, \qquad i=1,2,\cdots,d.
\end{align*}
With this definition, the equation \eqref{fkdvhd} can be rewritten as a system of equations
\begin{equation}
\begin{split}\label{systemfkdvhd}
    U_t & = -\sum\limits_{i=1}^d\left(f_i(U)+ P_i\right)_{x_i} ,\\
              P_i &= \Delta^i_{(\alpha_i-2)/2}Q_i, \quad i=1,2,\cdots,d,\\
              Q_i &= (R_i)_{x_i}, \quad i=1,2,\cdots,d,\\
              R_i &= U_{x_i}, \quad i=1,2,\cdots,d.
\end{split}
\end{equation}
Now, employing a methodology akin to that of the preceding section, we present the \textbf{LDG scheme}. Our objective is to determine approximations $(u_h,\mathbf{p}_h,\mathbf{q}_h,\mathbf{r}_h):=(u_h,(p_{1_h},\cdots,p_{d_h}),(q_{1_h},\cdots,q_{d_h}),$ $(r_{1_h},\cdots,r_{d_h}))\in H^1(0,T;V_d^k)\times (L^2(0,T;V_d^k))^d\times (L^2(0,T;V_d^k))^d\times (L^2(0,T;V_d^k))^d =: (\mathcal{T}_4\times\mathcal{V}_d^k)^d$ to the exact solution $(U,\mathbf{P},\mathbf{Q},\mathbf{R})$ of \eqref{systemfkdvhd}. This is subject to the condition that for all test functions $v\in H^1(0,T;V_d^k)$ and $w_i,z_i,s_i\in L^2(0,T;V_d^k)$, for all $i=1,2,\cdots,d$, the following system of equations holds for $\tau=1,2,\cdots,\mathcal{N}_d$:
\begin{equation}\label{LDGscheme1hd}
\begin{split}
     \int_{D^\tau}(u_h)_tv\,d\mathbf{x} & = \sum\limits_{i=1}^d\int_{D^\tau}\big(f_i(u_h) + (p_i)_h\big)v_{x_i}\,d\mathbf{x} - \int_{\partial D^\tau}(\hat f_{h,\mathbf n_\tau} + \hat{p}_{h,\tau}) v^{int,\tau}\,d\mathbf{\xi},\\
              \int_{D^\tau}p_{i_h}w_i\,d\mathbf{x} &= \int_{D^\tau}\Delta^i_{(\alpha_i-2)/2}q_{i_h}w_i\,d\mathbf{x}, \qquad i=1,2,\cdots,d,\\
              \int_{D^\tau}q_{i_h}z_i\,d\mathbf{x}&= -\int_{D^\tau}r_{i_h}(z_i)_{x_i}\,d\mathbf{x} + \int_{\partial D^\tau}(\hat{r}_i)_{h,\tau} z^{int,\tau}\,d\mathbf{\xi}, \qquad i=1,2,\cdots,d,\\
              \int_{D^\tau}r_{i_h}s_i\,d\mathbf{x}&= -\int_{D^\tau}u_h(s_i)_{x_i}\,d\mathbf{x} + \int_{\partial D^\tau}(\hat{u}_i)_{h,\tau} s^{int,\tau}\,d\mathbf{\xi}, \qquad i=1,2,\cdots,d,\\
            \int_{D^\tau}u_h^{0_d}v\,d\mathbf{x} &= \int_{D^\tau}U_{0_d}v\,d\mathbf{x},
\end{split}
\end{equation}
where $\partial D^\tau$ is the boundary of $D^\tau$, and ``the hats" are numerical fluxes which are defined by the following:
\begin{equation}\label{num_fluxhd}
    \hat f_{h,\mathbf n_\tau} = \widehat{f_{h,\mathbf n_\tau}}(u^{int,\tau},u^{ext,\tau}),\quad \hat{p}_{h,\tau} = \sum\limits_{i=1}^dp^+_{i_h}n_{i,\tau},\quad (\hat{r}_i)_{h,\tau} = r^+_{i_h}n_{i,\tau},\quad (\hat{u}_i)_{h,\tau} = u_h^-n_{i,\tau},
\end{equation}
for $i=1,2,\cdots,d$ and at the boundary of $\Omega_d$, choice of fluxes is obvious as $U$ has compact support. In this context, $\mathbf{n}_\tau = (n_{1,\tau},n_{2,\tau},\cdots,n_{d,\tau})$ signifies the outward unit normal vector for the triangle $D^\tau$ along its boundary $\partial D^\tau$. The terms $p_{i_h}^+$ and $r_{i_h}^+$ denote the values of $p_{i_h}$ and $r_{i_h}$ respectively, calculated from the plus side along an edge $e$ (of element $D^\tau$), typically formed by the boundaries of two adjacent elements. For a more detailed illustration, refer to \cite[Figure 3.1]{yan2002local}. The function $\widehat{f_{h,\mathbf n_\tau}}(u^{int,\tau},u^{ext,\tau})$ represents a monotone flux, which maintains Lipschitz continuity in both arguments. It aligns with $f_{\mathbf n_\tau}(U) = \sum\limits_{i=1}^d f_i(U)n_{i,\tau}$ and behaves as a decreasing function in $u^{int,\tau}$ while being an increasing function in $u^{ext,\tau}$. As an example, we select the Lax-Friedrichs flux
\begin{align}
    &\label{c}\widehat{f_{h,\mathbf n_\tau}}(u^{int,\tau},u^{ext,\tau}) = \frac{1}{2}\left(\sum\limits_{i=1}^d\big(f_i(u^{int,\tau})+f_i(u^{ext,\tau})\big)n_{i,\tau}-\gamma(u^{ext,\tau}-u^{int,\tau})\right),\\
    &\nonumber\gamma = \max\limits_U|f'_{\mathbf n_\tau}(U)|,
\end{align}
where the maximum is taken over relevant range of $U$. 

\begin{remark}
We would like to remark that our further analysis applies specifically to the finite element space $W_d^k$, and not to the standard $k$-th degree polynomial space $V_d^k$. The reason for this distinction is that our primary technique involves a special tensor projection, denoted as $\mathcal{P}^\pm$, which effectively removes certain jump terms that arise from higher-order derivative terms. However, in our numerical experiments in section \ref{sec6}, we confirm the optimal order of convergence for the LDG method when applied to $V_d^k$.
\end{remark}

\subsection{Error estimate}

We would like to derive an error estimate for the equation \eqref{fkdvhd}. For simplicity, let us consider it in two dimension, i.e. $d=2$  with the notation $(x,y):=(x_1,x_2)$. The following analysis can be extended to the higher dimension case with the same line of technique. Thus the equation \eqref{fkdvhd} reads 
\begin{equation}\label{fkdvhd2}
\begin{cases}
     U_t+f_1(U)_{x}+f_2(U)_{y}-(-\Delta)_{1}^{\alpha_1/2} U_{x}-(-\Delta)_{2}^{\alpha_2/2}U_{y}=0,  \quad (x,y) \in \Omega_2,~t\in(0,T],\\
     U(x,y,0) = U_{0}(x,y), \quad  (x,y) \in \Omega_2.
\end{cases}
\end{equation}
The semi-discrete LDG scheme \eqref{LDGscheme1hd} can be rewritten in the finite element space $W_d^k$ as: find the approximations $u_h,p_{1_h},p_{2_h},q_{1_h},q_{2_h},r_{1_h},r_{2_h}\in W_h^d$ to the exact solution $(U,P_1,P_2,Q_1,Q_2,$ $R_1,R_2)$ of \eqref{systemfkdvhd} such that $\forall v,w_1,w_2,z_1,z_2,s_1,s_2\in W_h^d$ and $j_i = 1,2,\cdots,N_{i},~i=1,2$, we have the following system of equations
\begin{align*}
    \int_{I^1_{j_1}}\int_{I^2_{j_2}}(u_h)_tv\,dy\,dx & = \int_{I^1_{j_1}}\int_{I^2_{j_2}}\big(f_1(u_h) + p_{1_h}\big)v_{x}\,dy\,dx+\int_{I^1_{j_1}}\int_{I^2_{j_2}}\big(f_2(u_h) + p_{2_h}\big)v_{y}\,dy\,dx\\ &\quad - \int_{I^2_{j_2}}\big((\hat f_1  + \hat p_{1_h}) v^{-}\big)_{j_1+\frac{1}{2},y}\,dy + \int_{I^2_{j_2}}\big((\hat f_1  + \hat p_{1_h} )v^{+}\big)_{j_1-\frac{1}{2},y}\,dy\\&\quad - \int_{I^1_{j_1}}\big((\hat f_2  + \hat p_{2_h}) v^{-}\big)_{x,j_2+\frac{1}{2}}\,dx + \int_{I^1_{j_1}}\big((\hat f_2 + \hat p_{2_h}) v^{+}\big)_{x,j_2-\frac{1}{2}}\,dx,
\end{align*}
\begin{align}\label{LDGscheme1hd2}
              \nonumber\int_{I^1_{j_1}}\int_{I^2_{j_2}}p_{i_h}w_i\,dy\,dx &= \int_{I^1_{j_1}}\int_{I^2_{j_2}}\Delta^i_{(\alpha_i-2)/2}q_{i_h}w_i\,dy\,dx, \qquad i=1,2,\\
             \nonumber \int_{I^1_{j_1}}\int_{I^2_{j_2}}q_{1_h}z_1\,dy\,dx&= -\int_{I^1_{j_1}}\int_{I^2_{j_2}}r_{1_h}(z_1)_{x}\,dy\,dx + \int_{I_{j_2}^2}(\hat{r}_1 z^{-}_1)_{j_1+\frac{1}{2},y}\,dy-\int_{I_{j_2}^2}(\hat{r}_2 z^{+}_1)_{j_1-\frac{1}{2},y}\,dy,\\
             \nonumber \int_{I^1_{j_1}}\int_{I^2_{j_2}}q_{2_h}z_2\,dy\,dx&= -\int_{I^1_{j_1}}\int_{I^2_{j_2}}r_{2_h}(z_2)_{y}\,dy\,dx + \int_{I_{j_1}^1}(\hat{r}_2 z^{-}_2)_{x,j_2+\frac{1}{2}}\,dx-\int_{I_{j_1}^1}(\hat{r}_2 z^{+}_2)_{x,j_2-\frac{1}{2}}\,dx,\\
             \nonumber  \int_{I^1_{j_1}}\int_{I^2_{j_2}}r_{1_h}s_1\,dy\,dx&= -\int_{I^1_{j_1}}\int_{I^2_{j_2}}u_{h}(s_1)_{x}\,dy\,dx + \int_{I_{j_2}^2}(\hat{u}_h s^{-}_1)_{j_1+\frac{1}{2},y}\,dy-\int_{I_{j_2}^2}(\hat{u}_h s^{+}_1)_{j_1-\frac{1}{2},y}\,dy,\\
             \nonumber \int_{I^1_{j_1}}\int_{I^2_{j_2}}r_{2_h}s_2\,dy\,dx&= -\int_{I^1_{j_1}}\int_{I^2_{j_2}}u_{h}(s_2)_{y}\,dy\,dx + \int_{I_{j_1}^1}(\hat{u}_h s^{-}_2)_{x,j_2+\frac{1}{2}}\,dx-\int_{I_{j_1}^1}(\hat{u}_h s^{+}_2)_{x,j_2-\frac{1}{2}}\,dx,\\
             \int_{I^1_{j_1}}\int_{I^2_{j_2}}u_h^{0}v\,dy\,dx &=  \int_{I^1_{j_1}}\int_{I^2_{j_2}}U_{0}v\,dy\,dx,
\end{align}
where ``hat'' terms in the above scheme are the numerical fluxes, which can be defined by
\begin{align}\label{fluxhd2}
    \hat u_h= u^-_h, \quad \hat p_{1_h} = p_{1_h}^+, \quad \hat r_1= r_{1_h}^+
\end{align}
at interface $\{j_1+\frac{1}{2},y\}$, $j_1=1,2,\cdots,N_1-1$ and 
\begin{align}\label{fluxhd22}
    \hat u_h = u^-_h, \quad \hat p_{2_h} = p_{2_h}^+, \quad \hat r_{2_h} =  r_{2_h}^+ 
\end{align}
at interface $\{x,j_2+\frac{1}{2}\}$, $j_2=1,2,\cdots,N_2-1$. Since $U$ has a compact support in $\Omega_2$, choice of fluxes at the boundary of $\Omega_2$ is obvious.
Furthermore, we choose any monotone fluxes for $\hat f_1 =  \hat f_1(u^-_h,u_h^+)$  and $\hat f_2 = \hat f_2(u^-_h,u_h^+)$ at interface  $\{j_1+\frac{1}{2},y\}$ and $\{x,j_2+\frac{1}{2}\}$ respectively, which have the uniform dissipation property. To proceed further, we define the linear (in each component) compact form of the scheme \eqref{LDGscheme1hd2}
\begin{align}\label{compactformhd23}
  \nonumber \mathcal B^2(u_h,p_{1_h},p_{2_h},& q_{1_h},q_{2_h},r_{1_h},r_{2_h};v,w_1,w_2,z_1,z_2,s_1,s_2)\\ \nonumber =&  \sum_{j_1=1}^N\sum_{j_2=1}^N \Bigg[ \int_{I^1_{j_1}}\int_{I^2_{j_2}}(u_h)_tv\,dy\,dx  -\int_{I^1_{j_1}}\int_{I^2_{j_2}}p_{1_h}v_{x}\,dy\,dx 
     -\int_{I^1_{j_1}}\int_{I^2_{j_2}}p_{2_h}v_{y}\,dy\,dx  \\& \nonumber
     + \int_{I^2_{j_2}}(\hat p_{1_h} v^{-})_{j_1+\frac{1}{2},y}\,dy 
     - \int_{I^2_{j_2}}(\hat p_{1_h} v^{+})_{j_1-\frac{1}{2},y}\,dy + \int_{I^1_{j_1}}(\hat p_{2_h} v^{-})_{x,j_2+\frac{1}{2}}\,dx \\& \nonumber
     -\int_{I^1_{j_1}}( \hat p_{2_h} v^{+})_{x,j_2-\frac{1}{2}}\,dx +\int_{I^1_{j_1}}\int_{I^2_{j_2}}p_{1_h}w_1\,dy\,dx 
     - \int_{I^1_{j_1}}\int_{I^2_{j_2}}\Delta^1_{\frac{(\alpha_1-2)}{2}}q_{1_h}w_1\,dy\,dx \\&\nonumber
     +\int_{I^1_{j_1}}\int_{I^2_{j_2}}p_{2_h}w_2\,dy\,dx 
     - \int_{I^1_{j_1}}\int_{I^2_{j_2}}\Delta^2_{\frac{(\alpha_2-2)}{2}}q_{2_h}w_2\,dy\,dx +\int_{I^1_{j_1}}\int_{I^2_{j_2}}q_{1_h}z_1\,dy\,dx \\& \nonumber
     +\int_{I^1_{j_1}}\int_{I^2_{j_2}}r_{1_h}(z_1)_{x}\,dy\,dx 
    -\int_{I_{j_2}^2}(\hat{r}_1 z^{-}_1)_{j_1+\frac{1}{2},y}\,dy
    +\int_{I_{j_2}^2}(\hat{r}_2 z^{+}_1)_{j_1-\frac{1}{2},y}\,dy\\& \nonumber
    +\int_{I^1_{j_1}}\int_{I^2_{j_2}}q_{2_h}z_2\,dy\,dx
    +\int_{I^1_{j_1}}\int_{I^2_{j_2}}r_{2_h}(z_2)_{y}\,dy\,dx -\int_{I_{j_1}^1}(\hat{r}_2 z^{-}_2)_{x,j_2+\frac{1}{2}}\,dx\\& \nonumber
    +\int_{I_{j_1}^1}(\hat{r}_2 z^{+}_2)_{x,j_2-\frac{1}{2}}\,dx
    +\int_{I^1_{j_1}}\int_{I^2_{j_2}}r_{1_h}s_1\,dy\,dx+
    \int_{I^1_{j_1}}\int_{I^2_{j_2}}u_{h}(s_1)_{x}\,dy\,dx \\& \nonumber
    - \int_{I_{j_2}^2}(\hat{u}_h s^{-}_1)_{j_1+\frac{1}{2},y}\,dy
    +\int_{I_{j_2}^2}(\hat{u}_h s^{+}_1)_{j_1-\frac{1}{2},y}\,dy
    +\int_{I^1_{j_1}}\int_{I^2_{j_2}}r_{2_h}s_2\,dy\,dx\\&
    +\int_{I^1_{j_1}}\int_{I^2_{j_2}}u_{h}(s_2)_{y}\,dy\,dx 
    - \int_{I_{j_1}^1}(\hat{u}_h s^{-}_2)_{x,j_2+\frac{1}{2}}\,dx+\int_{I_{j_1}^1}(\hat{u}_h s^{+}_2)_{x,j_2-\frac{1}{2}}\,dx\Bigg].
\end{align}
We choose the test functions $(v,w_1,w_2,z_1,z_2,s_1,s_2) =  (u_h,-q_{1_h}+p_{1_h}+r_{1_h},-q_{2_h}+p_{2_h}+r_{2_h},u_h+r_{1_h},u_h+r_{2_h},r_{1_h}-p_{1_h},r_{2_h}-p_{2_h})$ in the linear form $\mathcal{B}^2$ defined by \eqref{compactformhd23} to get the following identity
    \begin{align}\label{B2lemmahd}
  \nonumber\mathcal{B}^2(u_h,&p_{1_h},p_{2_h},q_{1_h},q_{2_h},r_{1_h},r_{2_h};u_h,-q_{1_h}+p_{1_h}+r_{1_h},-q_{2_h}+p_{2_h}+r_{2_h},u_h+r_{1_h},u_h+r_{2_h},\\&\qquad \nonumber r_{1_h}-p_{1_h},r_{2_h}-p_{2_h}) \\=& \nonumber \sum_{j_1=1}^N\sum_{j_2=1}^N \Bigg[ \int_{I^1_{j_1}}\int_{I^2_{j_2}}(u_h)_tv\,dy\,dx  -\int_{I^1_{j_1}}\int_{I^2_{j_2}}p_{1_h}q_{1_h}\,dy\,dx  -\int_{I^1_{j_1}}\int_{I^2_{j_2}}p_{2_h}q_{2_h}\,dy\,dx \\&\nonumber
  +\norm{p_{1_h}}^2_{L^2(I^1_{j_1}\times I^2_{j_2})}  +\norm{p_{2_h}}^2_{L^2(I^1_{j_1}\times I^2_{j_2})} + \int_{I^1_{j_1}}\int_{I^2_{j_2}}(\Delta^1_{\frac{(\alpha_1-2)}{2}}q_{1_h})q_{1_h}\,dy\,dx \\&\nonumber
  + \int_{I^1_{j_1}}\int_{I^2_{j_2}}(\Delta^2_{\frac{(\alpha_2-2)}{2}}q_{2_h})q_{2_h}\,dy\,dx - \int_{I^1_{j_1}}\int_{I^2_{j_2}}(\Delta^1_{\frac{(\alpha_1-2)}{2}}q_{1_h})p_{1_h}\,dy\,dx \\&\nonumber
  -\int_{I^1_{j_1}}\int_{I^2_{j_2}}(\Delta^2_{\frac{(\alpha_2-2)}{2}}q_{2_h})p_{2_h}\,dy\,dx - \int_{I^1_{j_1}}\int_{I^2_{j_2}}(\Delta^1_{\frac{(\alpha_1-2)}{2}}q_{1_h})r_{1_h}\,dy\,dx \\&\nonumber
  -\int_{I^1_{j_1}}\int_{I^2_{j_2}}(\Delta^2_{\frac{(\alpha_2-2)}{2}}q_{2_h})r_{2_h}\,dy\,dx +\int_{I^1_{j_1}}\int_{I^2_{j_2}}q_{1_h}u_{h}\,dy\,dx +\int_{I^1_{j_1}}\int_{I^2_{j_2}}q_{2_h}u_{h}\,dy\,dx\\&\nonumber
  +\int_{I^1_{j_1}}\int_{I^2_{j_2}}q_{1_h}r_{1_h}\,dy\,dx +\int_{I^1_{j_1}}\int_{I^2_{j_2}}q_{2_h}r_{2_h}\,dy\,dx +\norm{r_{1_h}}^2_{L^2(I^1_{j_1}\times I^2_{j_2})} 
  +\norm{r_{2_h}}^2_{L^2(I^1_{j_1}\times I^2_{j_2})}\Bigg]\\&+
 \frac{1}{2}\sum\limits_{j_1=1}^{N_1-1} \llbracket r_{1_h} \rrbracket_{{j_1}+\frac{1}{2},y}^2 + \frac{1}{2}\sum\limits_{j_2=1}^{N_2-1} \llbracket r_{2_h} \rrbracket_{x,{j_2}+\frac{1}{2}}^2 +\frac{1}{2}(r_{1_h{_{\frac{1}{2},y}}}^+)^2+\frac{1}{2}(r_{2_h{_{x,\frac{1}{2}}}}^+)^2.
    \end{align}

\begin{lemma}\label{stablemmahd}
    The LDG scheme \eqref{LDGscheme1hd2}-\eqref{fluxhd22} is $L^2$-stable.
\end{lemma}
The proof of the above Lemma \ref{stablemmahd} follows in a similar way as in Lemma \ref{stablemma} with the similar choices of test functions in both the directions, so we are omitting the details.

Following the approach from the one-dimensional case, we define the projection operators as in \cite{lesaint1974finite, cockburn2001superconvergence}. On a rectangle $\Omega_2$, we introduce the operators
\begin{equation}\label{proj2d}
    \mathbb{P}v = \mathcal{P}_{x}\otimes\mathcal{P}_{y}v,\qquad \mathbb{P}^\pm v = \mathcal{P}^\pm_{x}\otimes\mathcal{P}^\pm_{y}v,
\end{equation}
where the subscripts $x$ and $y$ indicate that the one-dimensional projections $\mathcal{P}$ and $\mathcal{P}^\pm$ are defined in \eqref{projectionprop} with respect to the corresponding variables and $\otimes$ is a tensor product.
We enlisted some properties of the projection operators $\mathbb{P}$ and $\mathbb{P}^\pm$. Let $g\in H^1(\Omega_2,\mathcal{T})$ be a sufficiently smooth function. Then we have (see \cite{cockburn2001superconvergence})
\begin{enumerate}[label=\roman*)]\label{1-4}
    \item $L^2$-projection:
    \begin{equation}\label{projectionprophd}
     \int_{I^1_{j_1}}\int_{I^2_{j_2}}\Big(\mathbb{P}^\pm g(x,y) - g(x,y)\Big)v(x,y)\,dy\,dx=0,
    \end{equation}
for any $v\in (P^{k-1}(I^1_{j_1})\otimes P^k(I^2_{j_2}))\cup (P^{k}(I^1_{j_1})\otimes P^{k-1}(I^2_{j_2}))$.
\item At $x$-interface
\begin{equation}\label{projectionprophd2}
    \begin{split}
         \int_{I^2_{j_2}}\Big(\mathbb{P}^\pm g(x^\pm_{j_1\mp\frac{1}{2}},y) - g(x^\pm_{j_1\mp\frac{1}{2}},y)\Big)v(x^\pm_{j_1\mp\frac{1}{2}},y)\,dy=0, \quad \forall v\in \mathcal{Q}^k(I^1_{j_1}\times I^2_{j_2}).
    \end{split}
\end{equation}
\item At $y$-interface
\begin{equation}\label{projectionprophd3}
    \begin{split}
         \int_{I^1_{j_1}}\Big(\mathbb{P}^\pm g(x,x^\pm_{j_2\mp\frac{1}{2}}) - g(x,x^\pm_{j_2\mp\frac{1}{2}})\Big)v(x,x^\pm_{j_2\mp\frac{1}{2}})\,dx=0, \quad \forall v\in \mathcal{Q}^k(I^1_{j_1}\times I^2_{j_2}).
    \end{split}
\end{equation}
\item Following approximation property holds:
\begin{equation}\label{Proj_app2d}
    \norm{\mathbb{P}g-g}_{L^2(\Omega_2)} + h\norm{\mathbb{P}g-g}_{L^\infty(\Omega_2)} + h^{\frac{1}{2}}\norm{\mathbb{P}g-g}_{\Gamma_h}\leq Ch^{k+1},
\end{equation}
where $C$ is a constant depends only on $g$.
\end{enumerate}

For simplicity, we define a few notations related to the projection operators $\mathbb P$ and $\mathbb P^\pm$ as follows
\begin{align*}
    \mathbb{P}^\pm_hu &= \mathbb{P}^\pm U-u_h ,\quad \mathbb{P}_hu = \mathbb{P}U-u_h, \\  \mathbb{P}^\pm_hp_i &= \mathbb{P}^\pm P_i-p_{i_h} ,\quad \mathbb{P}_hp_i = \mathbb{P}P_i-p_{i_h}, \qquad i=1,2,\\
    \mathbb{P}^\pm_hq_i &= \mathbb{P}^\pm Q_i-q_{i_h} ,\quad \mathbb{P}_hq_i = \mathbb{P}Q_i-q_{i_h}, \qquad i=1,2,\\
   \mathbb{P}^\pm_hr_i &= \mathbb{P}^\pm R_i-r_{i_h} ,\quad \mathbb{P}_hr_i = \mathbb{P}R_i-r_{i_h}, \qquad i=1,2,\\
    \mathbb{P}^\pm_eU &= \mathbb{P}^\pm U-U ,\quad \mathbb{P}_eU = \mathbb{P}U-U,\\
    \mathbb{P}^\pm_ep_i &= \mathbb{P}^\pm P_i-P_i ,\quad \mathbb{P}_ep_i = \mathbb{P}P_i-P_i, \qquad i=1,2,\\
    \mathbb{P}^\pm_eq_i &= \mathbb{P}^\pm Q_i-Q_i ,\quad \mathbb{P}_eq_i = \mathbb{P}Q_i-Q_i, \qquad i=1,2,\\
   \mathbb{P}^\pm_er_i &= \mathbb{P}^\pm R_i-R_i ,\quad \mathbb{P}_er_i = \mathbb{P}R_i-R_i, \qquad i=1,2.
\end{align*}

To figure out the estimates for the nonlinear fluxes \( f_1 \) and \( f_2 \), we need an \emph{a priori} assumption assumption stronger than the one-dimensional case \cite{xu2007error}.  Specifically, for sufficiently small \( h \), we assume the following condition holds
\begin{equation}\label{prioriass2}
    \norm{U-u_h}_{L^2(\Omega_2)}\leq h^{\frac{3}{2}}.
\end{equation}
This stronger assumption is necessary due to the following inverse inequality in two dimensions, which are characterized by the bound
\begin{equation*}
    \norm{u}_{L^\infty(\Omega_2)}\leq Ch^{-1}\norm{u}_{L^2(\Omega_2)}.
\end{equation*}
However, \emph{a priori} assumption is not necessary for the linear fluxes $f_1$ and $f_2$.
Now we provide an error estimate for the LDG scheme \eqref{LDGscheme1hd2}-\eqref{fluxhd22} specifically tailored for two dimensions. However, we opt not to reiterate the complete proof presented in the previous section concerning one dimension. Instead, we present error equations and an energy inequality, which helps to follow the proof in a similar structure.

\begin{theorem}\label{erreqn2thm}
Let $U$ be an exact solution of the Cauchy problem \eqref{fkdvhd2}. We assume that $U$ is sufficiently smooth and nonlinear fluxes $f_1,f_2\in C^3$. Let $u_h$ be an approximate solution of $U$ obtained by the LDG scheme \eqref{LDGscheme1hd2}-\eqref{fluxhd22}. Let $W_2^k$ be the piecewise tensor product polynomials of degree $k\geq 2$.
Then, the following error estimate holds for the small enough $h$:
\begin{equation}\label{temp_err_1}
    \norm{U-u_h}_{L^2(\Omega_2)}\leq Ch^{k+\frac{1}{2}},
\end{equation}
where $C = C\big(T,k,\norm{U}_{H^k},|f^{(m)}_1|,|f^{(m)}_2|\big)$, $m=1,2,3$.
\end{theorem}
\begin{proof}
    We begin by deriving an error equation for the scheme \eqref{LDGscheme1hd2}-\eqref{fluxhd22}. Since $u_h$ is an approximate solution and $(U,P_1,P_2,Q_1,Q_2, R_1,R_2)$  satisfies the scheme \eqref{LDGscheme1hd2}-\eqref{fluxhd22}, we end up with the following error equation in each cell:

\begin{equation}\label{errequn2}
 \begin{split}
      \mathcal{U}_{j_1j_2}&(U,u_{h};v,s_1,s_2)- \mathcal{F}_{j_1j_2}(f_1,f_2;U,u_h;v)
     +\mathcal{P}_{j_1j_2}(P_1,p_{1_h},P_2,p_{2_h};v,w_1,w_2)\\&
     +\mathcal{Q}_{j_1j_2}(Q_1,q_{1_h},Q_2,q_{2_h};w_1,w_2,z_1,z_2)
     +\mathcal{R}_{j_1j_2}(R_1,r_{1_h},R_2,r_{2_h};z_1,z_2,s_1,s_2)=0,
\end{split}
\end{equation}
where 
\begin{align*}
     \mathcal{U}_{j_1j_2}&(U,u_{h};v,s_1,s_2)=
     \int_{I^1_{j_1}}\int_{I^2_{j_2}}(U-u_h)_tv\,dy\,dx 
     +\int_{I^1_{j_1}}\int_{I^2_{j_2}}(U-u_{h})(s_1)_{x}\,dy\,dx \\&
    - \int_{I_{j_2}^2}\big((U-\hat{u}_h) s^{-}_1\big)_{j_1+\frac{1}{2},y}\,dy
    +\int_{I_{j_2}^2}\big((U-\hat{u}_h) s^{+}_1\big)_{j_1-\frac{1}{2},y}\,dy
    +\int_{I^1_{j_1}}\int_{I^2_{j_2}}(U-u_{h})(s_2)_{y}\,dy\,dx \\&
    - \int_{I_{j_1}^1}\big((U-\hat{u}_h) s^{-}_2\big)_{x,j_2+\frac{1}{2}}\,dx
    +\int_{I_{j_1}^1}\big((U-\hat{u}_h) s^{+}_2\big)_{x,j_2-\frac{1}{2}}\,dx,
\end{align*}
\begin{align*}
         &\mathcal{F}_{j_1j_2}(f_1,f_2;U,u_h;v) =  \int_{I^1_{j_1}}\int_{I^2_{j_2}}\big(f_1(U)-f_1(u_h)\big)v_{x}\,dy\,dx + \int_{I^1_{j_1}}\int_{I^2_{j_2}}\big(f_2(U)-f_2(u_h)\big)v_{y}\,dy\,dx \\&\quad
        -\int_{I^2_{j_2}}\big((f_1(U)-\hat f_1) v^{-}\big)_{j_1+\frac{1}{2},y}\,dy  +\int_{I^2_{j_2}}\big((f_1(U)-\hat f_1)v^{+}\big)_{j_1-\frac{1}{2},y}\,dy \\& \quad- \int_{I^1_{j_1}}\big((f_2(U)-\hat f_2) v^{-}\big)_{x,j_2+\frac{1}{2}}\,dx  
        + \int_{I^1_{j_1}}\big((f_2(U)-\hat f_2) v^{+}\big)_{x,j_2-\frac{1}{2}}\,dx,
\end{align*}

\begin{align*}
         &\mathcal{P}_{j_1j_2}(P_1,p_{1_h},P_2,p_{2_h};v,w_1,w_2) = -\int_{I^1_{j_1}}\int_{I^2_{j_2}}(P_1 -p_{1_h})v_{x}\,dy\,dx 
     -\int_{I^1_{j_1}}\int_{I^2_{j_2}}(P_2 - p_{2_h})v_{y}\,dy\,dx \\&\quad + \int_{I^2_{j_2}}((P_1 - \hat p_{1_h}) v^{-})_{j_1+\frac{1}{2},y}\,dy
     - \int_{I^2_{j_2}}\big((P_1- \hat p_{1_h} )v^{+}\big)_{j_1-\frac{1}{2},y}\,dy + \int_{I^1_{j_1}}\big((P_2-  \hat p_{2_h}) v^{-}\big)_{x,j_2+\frac{1}{2}}\,dx \\&\quad
     - \int_{I^1_{j_1}}\big((P_2-  \hat p_{2_h}) v^{+}\big)_{x,j_2-\frac{1}{2}}\,dx +\int_{I^1_{j_1}}\int_{I^2_{j_2}}(P_1-p_{1_h})w_1\,dy\,dx  +\int_{I^1_{j_1}}\int_{I^2_{j_2}}(P_2-p_{2_h})w_2\,dy\,dx,
\end{align*}

\begin{align*}
     &\mathcal{Q}_{j_1j_2}(Q_1,q_{1_h},Q_2,q_{2_h};w_1,w_2,z_1,z_2) = - \int_{I^1_{j_1}}\int_{I^2_{j_2}}\Delta^1_{\frac{(\alpha_1-2)}{2}}(Q_1-q_{1_h})w_1\,dy\,dx \\& \quad
     - \int_{I^1_{j_1}}\int_{I^2_{j_2}}\Delta^2_{\frac{(\alpha_2-2)}{2}}(Q_2-q_{2_h})w_2\,dy\,dx +\int_{I^1_{j_1}}\int_{I^2_{j_2}}(Q_1-q_{1_h})z_1\,dy\,dx
    +\int_{I^1_{j_1}}\int_{I^2_{j_2}}(Q_2-q_{2_h})z_2\,dy\,dx,
\end{align*}
and
\begin{align*}
    &\mathcal{R}_{j_1j_2}(R_1,r_{1_h},R_2,r_{2_h};z_1,z_2,s_1,s_2) = \int_{I^1_{j_1}}\int_{I^2_{j_2}}(R_1-r_{1_h})(z_1)_{x}\,dy\,dx 
    -\int_{I_{j_2}^2}\big((R_1-\hat{r}_1) z^{-}_1\big)_{j_1+\frac{1}{2},y}\,dy\\&\quad
    +\int_{I_{j_2}^2}\big((R_2-\hat{r}_2) z^{+}_1\big)_{j_1-\frac{1}{2},y}\,dy
    +\int_{I^1_{j_1}}\int_{I^2_{j_2}}(R_2-r_{2_h})(z_2)_{y}\,dy\,dx -\int_{I_{j_1}^1}\big((R_2-\hat{r}_2 )z^{-}_2\big)_{x,j_2+\frac{1}{2}}\,dx\\ & \quad
    +\int_{I_{j_1}^1}\big((R_2-\hat{r}_2) z^{+}_2\big)_{x,j_2-\frac{1}{2}}\,dx
    +\int_{I^1_{j_1}}\int_{I^2_{j_2}}(R_1-r_{1_h})s_1\,dy\,dx
    +\int_{I^1_{j_1}}\int_{I^2_{j_2}}(R_2-r_{2_h})s_2\,dy\,dx,
\end{align*}
for all $v,w_1,w_2,z_1,z_2,s_1,s_2 \in W_2^k$. 
Thus from \eqref{compactformhd23} and \eqref{errequn2}, we obtain the error equation
\begin{align}\label{compactformhd2}
   \nonumber\mathcal B^2(U-u_h&,P_1-p_{1_h},P_2-p_{2_h},Q_1-q_{1_h},Q_2-q_{2_h},R_1-r_{1_h},R_2-r_{2_h};v,w_1,w_2,z_1,z_2,s_1,s_2)\\ \nonumber=&\sum_{j_1=1}^N\sum_{j_2=1}^N \Big[  \mathcal{U}_{j_1j_2}(U,u_{h};v,s_1,s_2)
     +\mathcal{P}_{j_1j_2}(P_1,p_{1_h},P_2,p_{2_h};v,w_1,w_2)\\& \nonumber\qquad \qquad
     +\mathcal{Q}_{j_1j_2}(Q_1,q_{1_h},Q_2,q_{2_h};w_1,w_2,z_1,z_2)
     +\mathcal{R}_{j_1j_2}(R_1,r_{1_h},R_2,r_{2_h};z_1,z_2,s_1,s_2)\Big]\\ =&
     \sum_{j_1=1}^N\sum_{j_2=1}^N\mathcal{F}_{j_1j_2}(f,g;U,u_h;v). 
\end{align}

We choose the test functions $(v,w_1,w_2,z_1,z_2,s_1,s_2) =(\mathbb{P}^-_hu,-\mathbb{P}_hq_1 + \mathbb{P}^+_hp_1+\mathbb{P}_h^+r_1,-\mathbb{P}_hq_2 + \mathbb{P}^+_hp_2+\mathbb{P}_h^+r_2, \mathbb{P}^-_hu+ \mathbb{P}_h^+r_1,\mathbb{P}^-_hu+ \mathbb{P}_h^+r_2 , \mathbb{P}_h^+r_1 - \mathbb{P}^+_hp_1, \mathbb{P}_h^+r_2 - \mathbb{P}^+_hp_2)$ in the equation \eqref{compactformhd2}, also note that $U-u_h = \mathbb{P}^\pm_hu - \mathbb{P}^\pm_eU $ (similarly for $P,Q,R$) and $\mathcal B^2$ is linear, we have the following energy equation
\begin{align}\label{energyinehd2}
       \nonumber \mathcal B^2\big(&\mathbb{P}^\pm_hu,\mathbb{P}^\pm_hp_1,\mathbb{P}^\pm_hp_2,\mathbb{P}^\pm_hq_1,\mathbb{P}^\pm_hq_2,\mathbb{P}^\pm_hr_1,\mathbb{P}^\pm_hr_2;\mathbb{P}^-_hu,-\mathbb{P}_hq_1 + \mathbb{P}^+_hp_1+\mathbb{P}_h^+r_1,-\mathbb{P}_hq_2 + \mathbb{P}^+_hp_2 \\ \nonumber &\qquad \qquad+\mathbb{P}_h^+r_2, \mathbb{P}^-_hu + \mathbb{P}_h^+r_1,\mathbb{P}^-_hu+ \mathbb{P}_h^+r_2 , \mathbb{P}_h^+r_1 - \mathbb{P}^+_hp_1, \mathbb{P}_h^+r_2 - \mathbb{P}^+_hp_2\big)\\ \nonumber=& \mathcal B^2\big(\mathbb{P}^\pm_eU,\mathbb{P}^\pm_eP_1,\mathbb{P}^\pm_eP_2,\mathbb{P}^\pm_eQ_1,\mathbb{P}^\pm_eQ_2,\mathbb{P}^\pm_eR_1,\mathbb{P}^\pm_eR_2;\mathbb{P}^-_hu,-\mathbb{P}_hq_1 + \mathbb{P}^+_hp_1+\mathbb{P}_h^+r_1, \\ \nonumber& \qquad\qquad-\mathbb{P}_hq_2 + \mathbb{P}^+_hp_2+\mathbb{P}_h^+r_2, \mathbb{P}^-_hu+ \mathbb{P}_h^+r_1,\mathbb{P}^-_hu+ \mathbb{P}_h^+r_2 , \mathbb{P}_h^+r_1 - \mathbb{P}^+_hp_1, \mathbb{P}_h^+r_2 - \mathbb{P}^+_hp_2\big) \\\qquad&+\sum_{j_1=1}^N\sum_{j_2=1}^N\mathcal{F}_{j_1j_2}(f,g;U,u_h;\mathbb{P}^-_hu). 
\end{align}

It is now clear that the linear part $\mathcal B^2$ corresponds to $\mathcal B_{\lambda}$, and the nonlinear part $\mathcal{F}_{j_1j_2}$ corresponds to $\mathcal F_i$ for the one-dimensional problem as outlined in the Section \ref{sec3}. The left hand side of above energy equation \eqref{energyinehd2} can be estimated by incorporating the identity \eqref{B2lemmahd} and the right hand side can be estimated by the projection properties \eqref{projectionprophd}-\eqref{Proj_app2d} and assumption \eqref{prioriass2} along with using the similar estimate for the nonlinear term $\mathcal F_i$.  Since the further algebraic manipulations are a straightforward extension from the one-dimensional case to obtain the estimate \eqref{temp_err_1}, we refrain from restating these bounds here. Hence this completes the proof.
\end{proof}
\section{Stability and Convergence analysis of fully-discrete LDG scheme}\label{sec5}
In this section, our aim is to derive a fully discrete LDG scheme for the fractional KdV equation \eqref{fkdv}. 
To do this, we use the Crank-Nicolson method for temporal discretization in \eqref{LDGscheme1} and then define its compact form for stability and convergence analysis.
The stability of the fully discrete scheme trivially follows from the spatial stability analysis. However, to perform error analysis, we need to define a new error equation for the fully discrete LDG scheme, utilizing the projection operators introduced in the previous sections. Deriving this error equation is more complex due to inclusion of average terms in time, as this plays a key role in the derivation to improve the accuracy.

We discretize the temporal domain as follows. Let $\tau$ be a small time step and set \(\{t_n = n\tau\}_{n=0}^{M}\) which is the partition of the given time interval \([0, T]\) with $t_M=T$, and \(u^n = u(t_n)\) and \(u^{n+\frac{1}{2}} = \frac{u^{n+1} + u^n}{2}\). The fully discrete LDG scheme is designed as follows: given \(u_h^n\), determine \(u_h^{n+1}\) such that the following system holds
\begin{equation}\label{LDGscheme1dis}
\begin{split}
     \left(u_h^{n+1},v\right)_{I_i} & = \left(u_h^{n},v\right)_{I_i} +\tau\left(f(u^{n+\frac{1}{2}}_h) + p^{n+1}_h,v^n_x\right)_{I_i} - \tau\left(\hat f_h^{n+\frac{1}{2}} v + \hat p^{n+1}_h v\right)|_{x_{i-\frac{1}{2}}^+}^{x_{i+\frac{1}{2}}^-},\\
              \left(p_h^{n+1},w\right)_{I_i} &= \left(\Delta_{\frac{(\alpha-2)}{2}}q_h^{n+1},w\right)_{I_i},\\
              \left(q_h^{n+1},z\right)_{I_i}&= -\left(r_h^{n+1},z_x\right)_{I_i} + \left(\hat r_h^{n+1} z\right)|_{x_{i-\frac{1}{2}}^+}^{x_{i+\frac{1}{2}}^-},\\
              \left(r_h^{n+1},s\right)_{I_i}&= -\left(u_h^{n+\frac{1}{2}},s_x\right)_{I_i} +\left(\hat u_h^{n+\frac{1}{2}} s\right)|_{x_{i-\frac{1}{2}}^+}^{x_{i+\frac{1}{2}}^-}.
\end{split}
\end{equation}
for all \(v, w, z,s \in V_h^k\) and $n=1,2,\cdots,M-1$ and set initial condition $u^0_h = \mathcal{P}^-U_0$. We use the similar numerical fluxes at interfaces and boundaries as defined for semi-discrete scheme \eqref{LDGscheme1}-\eqref{LF}. To proceed further, we define the compact form of the scheme as in the previous sections 
\begin{align}\label{Compact}
   \nonumber\mathcal{B}^n(&u_h^n,u_h^{n+1},p_h^{n+1},q_h^{n+1},r_h^{n+1}; v,w,z,s)\\& =  \sum\limits_{i=1}^N\Big[\left(\frac{u_h^{n+1}-u_h^n}{\tau}, v\right)  - \left(\left(f(u^{n+\frac{1}{2}}_h) +p_h^{n+1}\right),v_x\right)_{I_i}   \nonumber+\left(p_h^{n+1},w\right)_{I_i} -  \left(\Delta_{\frac{(\alpha-2)}{2}}q_h^{n+1},w\right)_{I_i}
        \\&\quad  + \left(q_h^{n+1},z\right)_{I_i} + \left(r_h^{n+1},z_x\right)_{I_i} 
         +\left(r_h^{n+1},s\right)_{I_i} +\left(u_h^{n+\frac{1}{2}},s_x\right)_{I_i}\Big] + \mathcal{IF}(u_h^{n+\frac{1}{2}},p_h^{n+1},r_h^{n+1};v,z,s),
\end{align}
where $\mathcal{IF}$ is defined by \eqref{Interfacesflux}. Note that the operator $\mathcal{B}^n$ defined above is a discrete analogue in time of $\mathcal{B}$ defined by \eqref{perturbation}.
If $f(U)=0$ in \eqref{Compact}, then we denote the linear part of $\mathcal{B}^n$  by $\mathcal{B}_0^n$.

\begin{lemma}
    The fully discrete scheme \eqref{LDGscheme1dis} is \(L^2\)-stable, and the solution \(u_h^n\) obtained by the scheme \eqref{LDGscheme1dis} satisfies
    \begin{equation}\label{stabl2}
        \|u_h^{n+1}\|_{L^2(\Omega)} \leq C \|u_h^{n}\|_{L^2(\Omega)}, \qquad \forall n.
    \end{equation}
\end{lemma}
\begin{proof}
Since $(u_h^{n+1}, p_h^{n+1}, q_h^{n+1}, r_h^{n+1})$ is a solution of scheme \eqref{LDGscheme1dis}, then $$\mathcal{B}^n(u^n_h,u^{n+1}_h, p^{n+1}_h, q^{n+1}_h, r^{n+1}_h; v, w, z, s)=0 \text{ for any }(v, w, z, s) \in V_h^k.$$

 We choose test functions \((v, w, z,s) = (u_h^{n+\frac{1}{2}}, -q_h^{n+1}+p_h^{n+1}+r_h^{n+1}, u_h^{n+\frac{1}{2}}+r_h^{n+1}, r_h^{n+1}-p_h^{n+1})\) in the above equation, then estimates \eqref{B_nlinear}-\eqref{perturbation6} imply
    \begin{align*}
        \left(\frac{u_h^{n+1} - u_h^n}{\tau}, u^{n+\frac{1}{2}}\right) \leq C\norm{ u_h^{n+\frac{1}{2}}}^2.
    \end{align*}
Again using the Young's inequality and taking small $\tau$ such that $1-C\tau>\frac{1}{2}$, we have 
\begin{equation*}
        \|u_h^{n+1}\|_{L^2(\Omega)} \leq C \|u_h^{n}\|_{L^2(\Omega)}, \qquad \forall n.
    \end{equation*}
    This completes the proof.
\end{proof}

In order to obtain the error estimates for fully discrete LDG scheme \eqref{LDGscheme1dis}, we define error terms at $n-$th time step incorporating the projections defined in \eqref{projectionprop}. Let $U$ be an exact solution of \eqref{fkdv} and $u_h^n$ is an approximate solution of $U$ at $t_n$ in $V_h^k$, then we define
\begin{equation}\label{unerror}
    u_h^{n} - U(t_{n}) = \mathcal{P}_e^-U(t_n) - \mathcal{P}_h^-u^n,
\end{equation}
where projection $\mathcal{P}_h^- u^n = \mathcal{P}^- U(t_n) - u_h^n$ and $\mathcal{P}_e^- U(t_n) = \mathcal{P}^- U(t_n) - U(t_n)$. Similar expressions can be defined for $p_h^n - P(t_n)$, $q_h^{n} - Q(t_{n})$ and $r_h^{n} - R(t_{n})$. Now we present the error term at the average time $t_{n+\frac{1}{2}}$ as follows
\begin{align}\label{un2error}
 u_h^{n+\frac{1}{2}} - U(t_{n+\frac{1}{2}}) =  (\mathcal{P}_e^-U)_{n+\frac{1}{2}} - (\mathcal{P}_h^-u)_{n+\frac{1}{2}} +\theta_{n+\frac{1}{2}},
\end{align}
where
\begin{align*}
         (\mathcal{P}_e^-U)_{n+\frac{1}{2}} := \frac{1}{2}\big(\mathcal{P}_e^-U(t_{n+1})+\mathcal{P}_e^-U(t_{n})\big), \quad (\mathcal{P}_h^-u)_{n+\frac{1}{2}} :=  \frac{1}{2}\big(\mathcal{P}_h^-u^{n+1}+\mathcal{P}_h^-u^n\big)
\end{align*}
and 
\begin{equation*}
     \theta_{n+\frac{1}{2}} :=  \frac{U(t_{n})+U(t_{n+1})}{2}-U(t_{n+\frac{1}{2}}), \quad \zeta_{n+\frac{1}{2}} :=  \frac{U(t_{n+1})-U(t_{n})}{\tau}-U_t(t_{n}).
\end{equation*}
We employ the Taylor’s formula with an integral remainder to derive the following \(L^2\)-norm estimate for \(\theta_{n+\frac{1}{2}}\):
\begin{equation}\label{estrho}
    \norm{\theta_{n+\frac{1}{2}}}_{L^2(\Omega)}^2 \leq C\tau^3 \int_{t_n}^{t_{n+1}} \norm{U_{tt}(\xi)}_{L^2(\Omega)}^2 \,d\xi,\quad \norm{\zeta_{n+\frac{1}{2}}}_{L^2(\Omega)}^2 \leq C\tau^3 \int_{t_n}^{t_{n+1}} \norm{U_{ttt}(\xi)}_{L^2(\Omega)}^2 \,d\xi.
\end{equation}

We present and prove the following error estimate, building on the analysis conducted for the semi-discrete case. Here, we focus on the error analysis for a general nonlinear flux \( f \) and achieve a sub-optimal rate. For linear flux, an optimal rate can be obtained using a similar approach as in the semi-discrete case (see Lemma \ref{Lemmalin_err}). We omit the detailed calculations as they follow directly from the results presented in previous section.

\begin{theorem}\label{errorthm}
    Let \( U \) be the exact solution of the fractional KdV equation \eqref{fkdv}, which is sufficiently smooth. Let \( u_h^n \) be the approximations of \( U \) obtained by the fully discrete LDG scheme \eqref{LDGscheme1dis} at time \( t_n \). For sufficiently small \( h \) and \( \tau \), the following convergence rate holds:
    \begin{equation}\label{rate}
        \norm{U(t_n) - u_h^n}_{L^2(\Omega)} = \mathcal{O}(h^{k+\frac{1}{2}} + \tau^2), \qquad n=0,1,\cdots,M,
    \end{equation}
    where the constant \( C \) is independent of \( \tau \) and \( h \).
\end{theorem}

\begin{proof}
Since \(u_h^n\) satisfies the scheme \eqref{LDGscheme1dis} and \(U\) is an smooth solution of \eqref{fkdv}, using Taylor's formula and \eqref{estrho}, we have
\begin{equation}\label{errtrm}
    \mathcal{B}^n(U(t_n),U(t_{n+\frac{1}{2}}),P(t_{n+1}),Q(t_{n+1}),R(t_{n+1});v,w,z,s) \leq C\tau^3 \int_{t_n}^{t_{n+1}} \norm{U_{ttt}(\xi)}_{L^2(\Omega)}^2 \,d\xi, \\
    \end{equation}
    and
 \begin{equation}\label{errtrm2}
    \mathcal{B}^n(u_h^n,u_h^{n+\frac{1}{2}},p_h^{n+1},q_h^{n+1},r_h^{n+1};v,w,z,s) =0.
\end{equation}
Hence, using bi-linearity of $\mathcal{B}_0^n$, and subtracting \eqref{errtrm} from \eqref{errtrm2}, also using \eqref{Nonlinearpart}, we have the following error equation
\begin{align}
    \nonumber \mathcal{O}(\tau^3)& = \mathcal{B}^n_0\big(u_h^n-U(t_n),u_h^{n+\frac{1}{2}}-U(t_{n+\frac{1}{2}}),p_h^{n+1}-P(t_{n+1}),q_h^{n+1}-Q(t_{n+1}),r_h^{n+1}-R(t_{n+1}); v,w,z,s\big)\\ \nonumber&\qquad + \sum\limits_{i=1}^N \mathcal{F}_i(f;U(t_{n+\frac{1}{2}}),u_h^{n+\frac{1}{2}},v)
\end{align}
which further implies using \eqref{unerror} and \eqref{un2error}
\begin{align}\label{erreqn}
  \nonumber\mathcal{B}^n_0(&\mathcal{P}_h^-u^n,(\mathcal{P}_h^-u)_{n+\frac{1}{2}},\mathcal{P}^+_hp^{n+1}, \mathcal{P}_hq^{n+1},\mathcal{P}^+_hr^{n+1}; v,w,z,s)-\mathcal{O}(\tau^3) =  \mathcal{B}^n_0\big(\mathcal{P}_e^-U(t_n),(\mathcal{P}_e^-U)_{n+\frac{1}{2}} +\theta_{n+\frac{1}{2}},\\&\qquad\mathcal{P}^+_eP(t_{n+1}), \mathcal{P}_eQ(t_{n+1}),\mathcal{P}^+_eR(t_{n+1}); v,w,z,s\big)+\sum\limits_{i=1}^N \mathcal{F}_i(f;U(t_{n+\frac{1}{2}}),u_h^{n+\frac{1}{2}},v).
\end{align}
We follow the approach of semi-discrete analysis \eqref{nonlinerr}-\eqref{nonlin2err} by choosing the appropriate test functions  $$(v,w,z,s) = \big((\mathcal{P}_h^-u)_{n+\frac{1}{2}}, -\mathcal{P}_h q^n+\mathcal{P}^+_h p^n +\mathcal{P}^+_h r^n, (\mathcal{P}_h^-u)_{n+\frac{1}{2}} + \mathcal{P}^+_h r^n,\mathcal{P}^+_h r^n -   \mathcal{P}^+_h p^n \big)$$ in \eqref{erreqn} so that \eqref{erreqn} becomes

\begin{align}\label{est1help}
    \nonumber\Big(\frac{\mathcal{P}_h^-u^{n+1}-\mathcal{P}_h^-u^n}{\tau}, (\mathcal{P}_h^-u)_{n+\frac{1}{2}}\Big) \nonumber
    &\leq \Big(\frac{\mathcal{P}_e^-U(t_{n+1})-\mathcal{P}_e^-U(t_n)}{\tau} + \theta_{N+\frac{1}{2}},(\mathcal{P}_h^-u)_{n+\frac{1}{2}}\Big) \\ &\qquad+C\norm{(\mathcal{P}_h^-u)_{n+\frac{1}{2}}}_{L^2(\Omega)}^2 + C(\Omega)h^{2k+1}+ C\tau^3.
\end{align}
We incorporate the Young's inequality, interpolation estimate \eqref{interpest} and estimate \eqref{estrho}, and also we have taken sufficiently small $\tau$ such that $1-C\tau>\frac{1}{2}$ to get
\begin{align}\label{est3help}
   \nonumber\norm{\mathcal{P}_h^-u^{n+1}}_{L^2(\Omega)}^2 &\leq  C(\Omega)\tau h^{2k+1}+ \frac{1+C\tau}{1-C\tau}\norm{\mathcal{P}_h^-u^{n}}_{L^2(\Omega)}^2  + C \tau^4 \int_{t_n}^{t_{n+1}} \norm{U_{tt}(\xi)}^2_{L^2(\Omega)}\,d\xi +C\tau^4 \\ \nonumber
   &\leq   C(\Omega)n\tau h^{2k+1}+  \left(\frac{1+C\tau}{1-C\tau}\right)^n\norm{\mathcal{P}_h^-u^{0}}_{L^2(\Omega)}^2 \\ \nonumber&\qquad + C \tau^4 \sum\limits_{k=0}^{n-1} \left(\frac{1+C\tau}{1-C\tau}\right)^{n-k} \int_{t_k}^{t_{k+1}} \norm{U_{tt}(\xi)}^2_{L^2(\Omega)}\,d\xi\\
   &\leq C(\Omega)T h^{2k+1}+ e^{4CT}\norm{\mathcal{P}_h^-u^{0}}_{L^2(\Omega)}^2  +e^{4CT} \tau^4  \int_{0}^{T} \norm{U_{tt}(\xi)}^2_{L^2(\Omega)}\,d\xi,
\end{align}
Since $U$ is smooth and $\mathcal{P}_h^-u^{0} = 0$, the estimate \eqref{est3help} becomes
\begin{equation*}
    \norm{\mathcal{P}_h^-u^{n+1}}_{L^2(\Omega)} \leq C(U_0,T) (h^{k+\frac{1}{2}}+\tau^2).
\end{equation*}
This completes the proof.

\end{proof}

 \begin{remark}
  We would like to remark that higher order time-stepping method can also be considered for the fully discrete LDG scheme. For instance, fourth order Runge-Kutta (RK) time discretization can be employed. The semi-discrete scheme \eqref{LDGscheme1} corresponds to an ODE system
         \begin{equation}\label{ddtldg}
          \frac{d}{dt}u_h = L_h u_h.
         \end{equation}
where $L_h$ is the LDG operator of the semi-discrete scheme \eqref{LDGscheme1}. Since the developed LDG scheme \eqref{LDGscheme1} is higher order accurate in space and achieved accuracy of order $k+1$, to achieve the better accuracy in time as well one needs to discretize time using the higher order scheme such as Runge-Kutta method.
With this, we consider the standard four stage explicit fourth order RK-LDG scheme which can be presented as follows:
\begin{align}\label{RK4scheme}
    u^{n+1} = P_4(\tau L_h)u^n,
\end{align}
where the operator $P_4$ is defined by
\begin{align*}
    P_4( \tau L_h) = I+\tau L_h +\frac{1}{2}(\tau L_h)^2 +\frac{1}{6}(\tau L_h)^3 + \frac{1}{24}(\tau L_h)^4.
\end{align*}
The reason behind introducing such higher order scheme is to verify our theoretical findings for higher degree polynomials as approximations in the numerical section. However, stability and convergence analysis for the fourth order Runge-Kutta LDG scheme has not been developed yet for the higher order equations. We refer to \cite{dwivedi2024local}, where the strong stability for fourth order Runge-Kutta LDG scheme is obtained for the Benjamin-Ono equation by considering the linear flux.
\end{remark}

\section{Numerical Examples}\label{sec6}
In this section, our aim is to validate the proposed LDG scheme \eqref{LDGscheme1} for solving the fractional KdV equation \eqref{fkdv} with $\alpha$ ranging between $1$ and $2$. For the numerical computation of the fractional integral $\Delta_{\frac{\alpha-2}{2}}$, we redefine ${}_{a}I_x^{2-\alpha}$ and ${}_{x}I_b^{2-\alpha}$ in the interval $\Omega$ by applying the linear transformations $t\mapsto \left(\frac{x+a}{2}+\frac{x-a}{2}\xi\right)$ and $t\mapsto \left(\frac{b+x}{2}+\frac{b-x}{2}\xi\right)$ to the definitions \eqref{leftfrac} and \eqref{rightfrac}, respectively. This results in:

\begin{align*}
    {}_{a}I_x^{2-\alpha} u(x)& =  \frac{1}{\Gamma({2-\alpha})}\left(\frac{x-a}{2}\right)^{2-\alpha}\int_{-1}^1(1-\xi)^{1-\alpha}u\left(\frac{x+a}{2}+\frac{x-a}{2}\xi\right)\,d\xi,\\
    {}_{x}I_b^{2-\alpha} u(x)& = \frac{1}{\Gamma({2-\alpha})}\left(\frac{b-x}{2}\right)^{2-\alpha}\int^{1}_{-1}(1+\xi)^{1-\alpha}u\left(\frac{b+x}{2}+\frac{b-x}{2}\xi\right)\,d\xi.
\end{align*}
With this setup, we can efficiently compute the fractional integral using the Gauss–Jacobi quadrature with weight functions $(1-\xi)^{1-\alpha}$ and $(1+\xi)^{1-\alpha}$, where the condition $1-\alpha>-1$ is satisfied, as required. For more details and algorithm, one may refer to \cite[Appendix A]{hesthaven2007nodal} and \cite{qiu2015nodal}. The order of convergence is defined for each intermediate step between element numbers $N_1$ and $N_2$ as
\begin{equation*}
    R_E = \frac{\ln(E(N_1))-\ln(E(N_2))}{\ln(N_2)-\ln(N_1)},
\end{equation*}
where the error $E$ can be seen as a function of number of elements $N$.

\subsection{Crank-Nicolson LDG Scheme}
We validate the proposed fully discrete Crank-Nicolson (CN) LDG scheme \eqref{LDGscheme1dis} through an illustration. The one soliton solution for KdV equation $U_t+\left(\frac{U^2}{2}\right)_x +U_{xxx}=0$ is given by \cite{dutta2015convergence,dwivedi2023stability}:
\begin{equation}\label{onesol}
     U(x,t) = 9\left(1-\tanh^2\left(\sqrt{\frac{3}{2}}(x-3t)\right)\right).
\end{equation}
The solution $U$ represents a single soliton wave moving with speed $3$ towards the right side. We verified our CN-LDG scheme by setting the initial data $u^0_h = U(x,-1)$. The numerical solution of fractional KdV \eqref{fkdv} with $\alpha=1.999$ is calculated on the uniform grid with time step $\tau=0.5h$ at time $T=2$ in the space interval $[-15,15]$. The Table \ref{tab:error_tableKDVCN} represents the expected optimal rate of convergence.
\begin{table}[htbp]
    \centering
    \begin{tabular}{|c|c|c|}
       \hline
 $N$ & $E$& $R_E$ \\
 \hline
 \hline 
 320  & 1.41e-01&\\
  &  & 1.97\\
 640 & 3.62e-02  &\\
   &  & 1.98 \\
 1280  & 9.16e-03 &\\
  &  & 2.00\\
 2560  & 2.29e-04 &\\
 \hline
    \end{tabular}
    \caption{Error and order of convergence for the fractional KdV equation \eqref{fkdv} with $\alpha=1.999$ taking $N$ elements and polynomial degree $k=1$ using the CN-LDG scheme \eqref{LDGscheme1dis}.}
    \label{tab:error_tableKDVCN}
\end{table}

\subsection{Higher Order LDG scheme}
To achieve the better accuracy in time, we perform the numerical experiments using the higher order time discretization method such as fourth order Runge-Kutta LDG scheme \eqref{RK4scheme}. 
Our focus lies on verifying the performance of scheme using a low storage explicit Runge-Kutta (LSERK) of fourth order \cite{carpenter1994fourth,hesthaven2007nodal} time discretization method of the form
\begin{align*}
    &r^{(0)} = u_h^n,\\
    &\text{for } j =1:5\\
    & \qquad \begin{cases}
        k^j = a_j k^{j-1} +\Delta t L_hr^{(j-1)},\\
        r^{(j)} = r^{(j-1)}+b_j k^{j},
    \end{cases}\\
    &u_h^{n+1} = r^{(5)},
\end{align*}
where the weighted coefficients $a_j$, $b_j$ and $c_j$ of the LSERK method given in \cite{hesthaven2007nodal}. The above iteration in competitive with the classical fourth-order method \eqref{RK4scheme} is considerably more efficient and accurate than \eqref{RK4scheme}, as it has the disadvantage that it requires four extra storage arrays.
\subsubsection{Compare with classical KdV}
In this numerical example, we explore the behavior of the LDG scheme \eqref{LDGscheme1}-\eqref{Boundaryfluxes2} for the fractional KdV equation \eqref{fkdv} by selecting the exponent $\alpha$ close to $2$. The flux function is considered as $f(U) = -3U^2$ and the exact solution of the KdV equation $U_t-3(U^2)_x+U_{xxx}=0$ is used in \cite{yan2002local}:
\begin{equation}\label{exactsolkdv}
    U(x,t) = -2 \sech^2(x-4t), \qquad x\in[-10,12],\quad t\in [0,1].
\end{equation}
Since for $\alpha = 2$ the equation \eqref{fkdv} corresponds to the generalized KdV equation, we have shown that the approximate solution obtained by the LDG scheme \eqref{LDGscheme1}-\eqref{Boundaryfluxes2} converges to the exact solution of the generalized KdV equation when $\alpha$ is chosen close to $2$.

We initialize our simulation with the initial condition at $t=0$ given by 
\begin{equation}\label{u0alpha2}
    U_0(x)= U(x,0) = -2 \sech^2(x), \qquad x\in[-10,12],
\end{equation}
and compare the results at final time $T=1$.
\begin{table}[htbp]
    \centering
    \begin{tabular}{|c|c|c|c|c|c|c|}
    \hline
        \multicolumn{7}{|c|}{$\alpha=1.999$} \\
        \hline
        \multirow{2}{*}{\textbf{$N/k$}} & \multicolumn{2}{c|}{\textbf{$k=1$}} & \multicolumn{2}{c|}{\textbf{$k=2$}} & \multicolumn{2}{c|}{\textbf{$k=3$}} \\
        \cline{2-7}
        & \textbf{$E$} & \textbf{$R_E$} & \textbf{$E$} & \textbf{$R_E$} & \textbf{$E$} & \textbf{$R_E$} \\
        \hline
          \textbf{$N=40$} & 2.02e-01 &  & 7.07e-02 &  & 1.13e-03& \\  
        &  & 3.30 &  & 2.89 & & 3.98\\  
         \textbf{$N=80$}& 2.04e-02 &  & 9.52e-03 &  &7.21e-05 & \\
        & & 2.55 &  & 2.95 & & 3.98  \\
         \textbf{$N=160$} & 3.46e-03 &  & 1.22e-03 &  & 4.57e-06& \\  
        &  & 2.28 &  & 2.98 & & 3.91 \\  
         \textbf{$N=320$}& 7.09e-04 &  & 1.54e-04 &  & 3.04e-07 & \\
        \hline
    \end{tabular}
    \caption{Error and order of convergence for fractional KdV equation \eqref{fkdv} with $\alpha$ close to $2$ taking $N$ elements and polynomial degree $k$.}
    \label{tab:error_tableKDV}
\end{table}

\begin{figure}
    \centering
    \includegraphics[width=0.9 \linewidth, height=8cm]{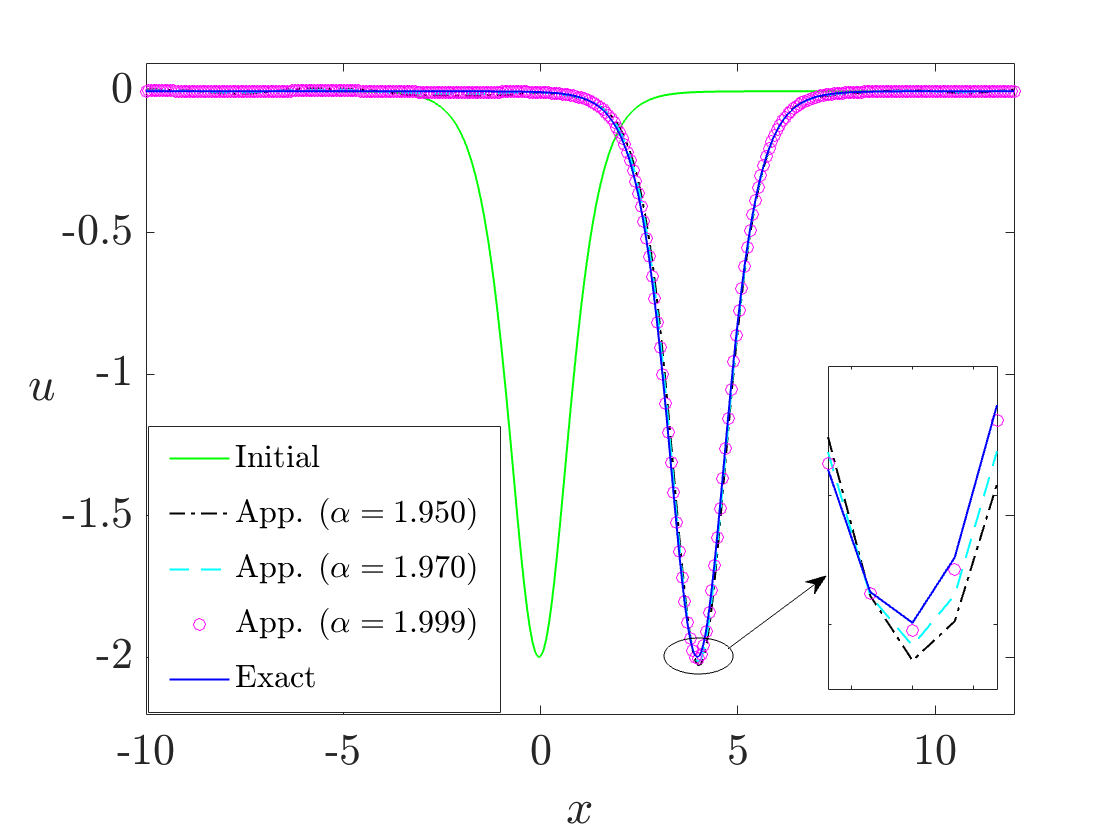}
    \caption{The exact solutions and approximate solution of \eqref{fkdv} at $T=1$ with $N=80$, $k=3$ and fractional exponent $\alpha=1.950$, 1.970 and 1.999.}
    \label{fig:fkdv_k3}
\end{figure}
This numerical experiment allows us to verify theoretical results obtained in previous sections for the LDG scheme \eqref{LDGscheme1}-\eqref{Boundaryfluxes2} and it also captures the dynamics of the KdV equation for fractional values of $\alpha$ close to $2$ for the initial condition $U_0$ defined in \eqref{u0alpha2}. The Table \ref{tab:error_tableKDV} represents the error rates which are optimal even for coarser grids, provided the CFL condition number is sufficiently small. Its graphical representation is depicted in Figure \ref{fig:fkdv_k3} with taking number of grid points $N=160$ and degree of polynomial $k=3$. From the Figure \ref{fig:fkdv_k3}, it is evident that as $\alpha$ tends to $2$, specifically, for $\alpha = 1.900, ~1.950, ~1.999$, approximate solutions converge to the exact solution \eqref{exactsolkdv}. We choose the polynomials of degree $k=1,2,3$, and it is observed that for higher degree polynomials, approximate solution converges more accurately. 

\subsubsection{Fractional Case}
For the fractional value of $\alpha$, we consider the following linear fractional KdV equation:
\begin{equation}\label{examfrac}
    \begin{cases}
        U_t(x,t)-(-\Delta)^{\alpha/2}U_x(x,t) = g(x,t), \qquad &\text{ in } [0,1]\times(0,0.01],\\
        U(x,0) = U_{0g}(x), \qquad &\text{ on } [0,1],
    \end{cases}
\end{equation}
with the initial condition $U_{0g}(x) = x^6(1-x)^6$. Here $g(x,t)$ is an additional source term, and we choose 
\begin{equation*}
    g(x,t) = e^{-t}\left(-U_{0g}(x)-(-\Delta)^{\alpha/2}(U_{0g})_x(x)\right),
\end{equation*}
to obtain the exact solution $U(x,t) = e^{-t}x^6(1-x)^6$. 
\begin{table}[htbp]
    \centering
    \begin{tabular}{|c|c|c|c|c|c|c|c|}
        \hline
        \multicolumn{8}{|c|}{\textbf{$\alpha$ is chosen in between 1 and 2}} \\
        \hline
        \multicolumn{2}{|c|}{\multirow{2}{*}{\textbf{$\alpha$$/N/k$}}} & \multicolumn{2}{c|}{\textbf{$k=1$}} & \multicolumn{2}{c|}{\textbf{$k=2$}} & \multicolumn{2}{c|}{\textbf{$k=3$}} \\
        \cline{3-8}
        \multicolumn{2}{|c|}{} & \textbf{$E$} & \textbf{$R_E$} & \textbf{$E$} & \textbf{$R_E$} & \textbf{$E$} & \textbf{$R_E$} \\
        \hline
        \textbf{$\alpha=1.1$} & \textbf{$N=20$} & 8.22e-03 &  & 5.16e-04 &  & 1.10e-05 & \\
        & &  & 1.99 &  & 2.98 &  & 3.93 \\
        \textbf{} & \textbf{$N=40$}& 2.05e-03 &  & 6.45e-05 &  & 7.23e-07 & \\
        & & & 2.00 &  &  3.00 & & 3.99 \\
        \textbf{} & \textbf{$N=80$}& 5.14e-04 &  & 8.07e-06 &  &4.52e-08 & \\
        & & & 2.00 &  & 3.00 & & 3.99 \\
        \textbf{} & \textbf{$N=160$} & 1.28e-04 &  & 1.01e-06 &  & 2.83e-09 & \\
        \hline
        \textbf{$\alpha=1.5$} & \textbf{$N=20$} & 8.21e-03 &  & 5.16e-04 &  & 1.11e-05 & \\
        & &  & 1.98 &  & 2.99 &  & 3.92\\
        \textbf{} & \textbf{$N=40$}& 2.06e-03 &  & 6.43e-05 &  &7.23e-07 & \\
        & & & 2.00 &  & 2.99 &  & 3.99\\
        \textbf{} & \textbf{$N=80$}& 5.14e-04 &  & 8.08e-06 &  & 4.51e-08 & \\
        & & & 2.00 &  & 3.00 &  & 4.00\\
        \textbf{} & \textbf{$N=160$} & 1.28e-04 &  & 1.01e-06 &  & 2.82e-09 & \\
        \hline
       \textbf{$\alpha=1.8$} & \textbf{$N=20$} & 8.28e-03 &  & 5.18e-04&  & 1.12e-05 & \\
        & &  & 1.96 &  & 2.99 &  & 3.89\\
        \textbf{} & \textbf{$N=40$}& 2.12e-03 &  & 6.45e-05 &  &7.22e-07 & \\
        & & & 1.98 &  & 3.01 &  & 3.97\\
        \textbf{} & \textbf{$N=80$}& 5.23e-04 &  & 8.09e-06 &  & 4.48e-08 & \\
        & & & 2.00 &  & 3.00 &  & 4.00\\
        \textbf{} & \textbf{$N=160$} & 1.35e-04 &  & 1.01e-06 &  & 2.80e-09 & \\
        \hline
    \end{tabular}
    \caption{Error and order of convergence for different values of $\alpha$ in example \eqref{examfrac} with the source term taking $N$ elements and polynomial degree $k$.}
    \label{tab:error_table_combined}
\end{table}
We verify the order of convergence of the scheme \eqref{LDGscheme1} with a very small final time $T=0.01$. Table \ref{tab:error_table_combined} displays the convergence rates and $L^2$-errors obtained by implementing the LDG scheme with the initial data $U_{0g}$ for $\alpha=1.1, 1.5$, and $1.8$ respectively.

Afterwards, we investigate the performance of the LDG scheme \eqref{LDGscheme1} across different fractional values of $\alpha$, utilizing smooth initial data and without a source term. 
We examine the nonlinear fractional KdV equation \eqref{fkdv} incorporating the nonlinear flux $f(U) = \frac{1}{2}U^2$. The simulation is extended to a final time $T=1$ incorporating smooth initial conditions with the Dirichlet boundary conditions as described in \cite{dutta2021operator}:
\begin{equation}\label{smoothfrac}
    V_0(x) = U(x,0) =  \frac{1}{4} \sin(x), \qquad x\in(-2\pi,2\pi).
\end{equation}
\begin{figure}
    \centering
    \includegraphics[width=1 \linewidth, height=8cm]{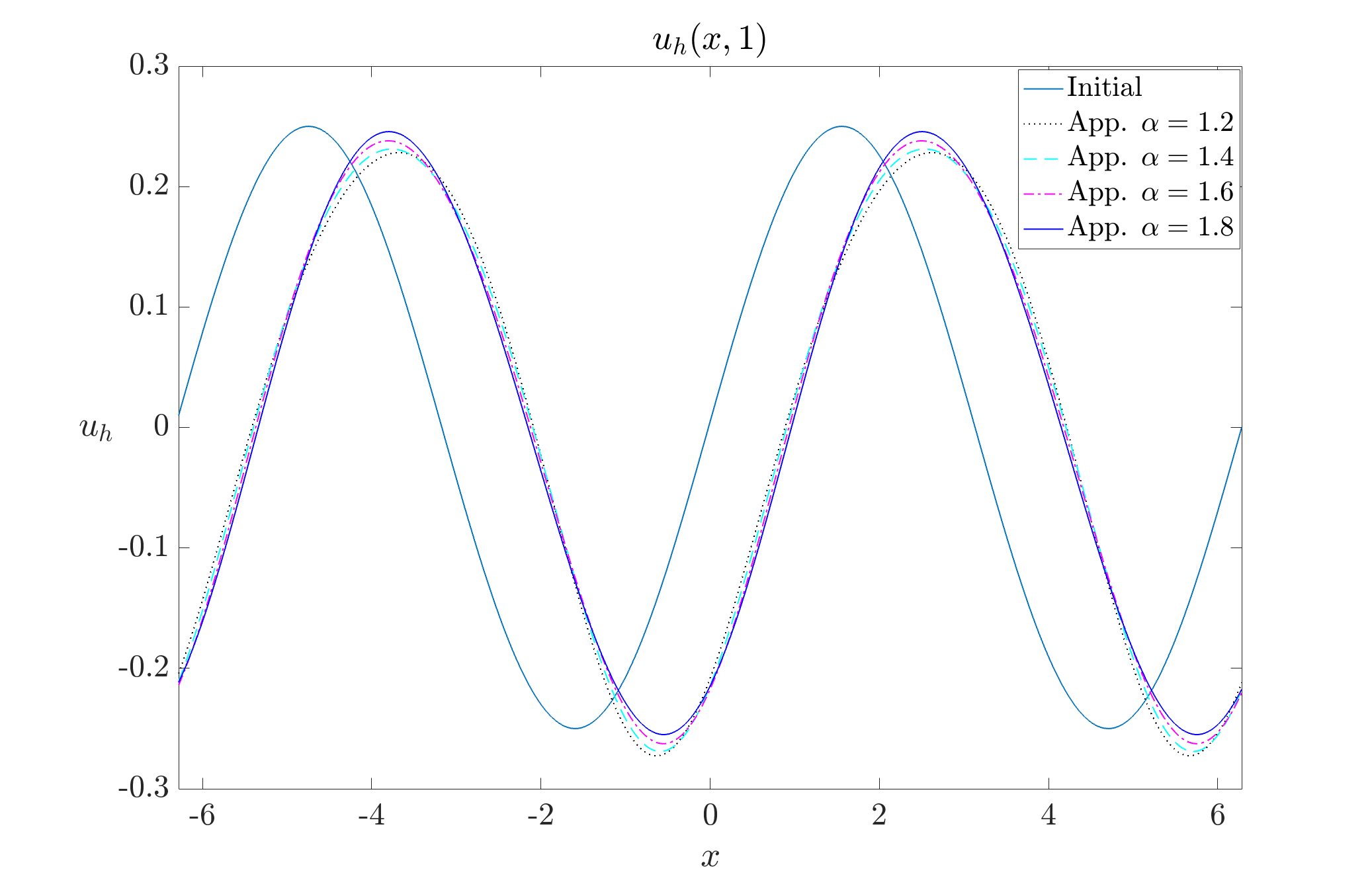}
    \caption{Approximate solution of fractional KdV equation at $T=1$ with $N=320$, $k=3$ for the fractional values $\alpha= 1.2, 1.4, 1.6 $ and 1.8, by choosing smooth initial condition $V_0$ in \eqref{smoothfrac}.}
    \label{fig:fkdv_smooth}
\end{figure}
\begin{figure}
    \centering
    \includegraphics[width=1 \linewidth, height=11cm]{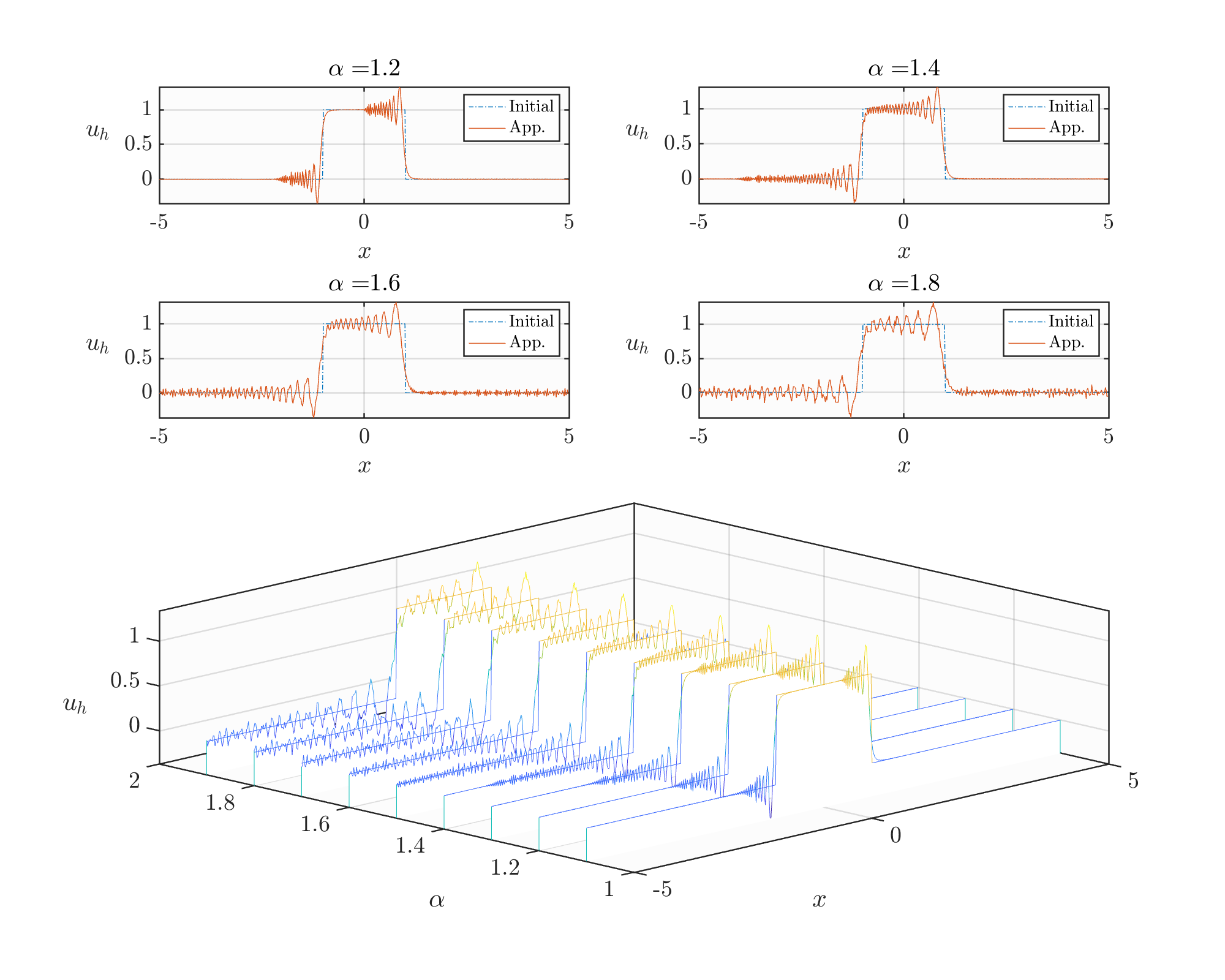}
    \caption{Approximate solution at $T=0.001$ with $N=320$, $k=3$ and initial condition $W_0$ of \eqref{fkdv}.}
    \label{fig:fkdv_L2}
\end{figure}
Due to unavailability of the exact solution, our focus lies solely on observing the behavior of the approximate solution over time. In Figure \ref{fig:fkdv_smooth}, we observe the expected rightward movement of the solution, consistent with the positive sign of the higher derivative for initial condition $V_0$ in \eqref{smoothfrac}. 

Furthermore, we carry out the numerical investigations for non-smooth initial data which is defined in the interval $[-5,5]$ as
\[
W_0(x)= U(x,0) = \begin{cases}
    1, & -1 \leq x \leq 1,\\
    0, & \text{otherwise}.
\end{cases}
\]
Since the initial data $W_0$ exhibits a jump discontinuity at two points, it belongs to $L^2(\Omega)$ but not to any other Sobolev spaces of positive indices. In this case, an exact solution is unavailable of the fractional KdV equation \eqref{fkdv}. However, our primary objective is to understand the qualitative behavior at the discontinuity of the solution produced by our scheme, rather than comparing it or determining its convergence rate. In Figure \ref{fig:fkdv_L2} and \ref{fig:fkdv_L25}, we use the grid points $N = 320$ and final time $T=0.001$ and $T=0.005$ respectively. We observe that the dispersion term forces the solution to evolve into traveling waves. Even at early time stages, approximate solution breaks up in many oscillation waves at the discontinuities. Furthermore, we note that as $\alpha$ increases between $1$ and $2$, the height of the traveling waves increases but the number of oscillation decreases near discontinuity.

 \begin{figure}
    \centering
    \includegraphics[width=1 \linewidth, height=12cm]{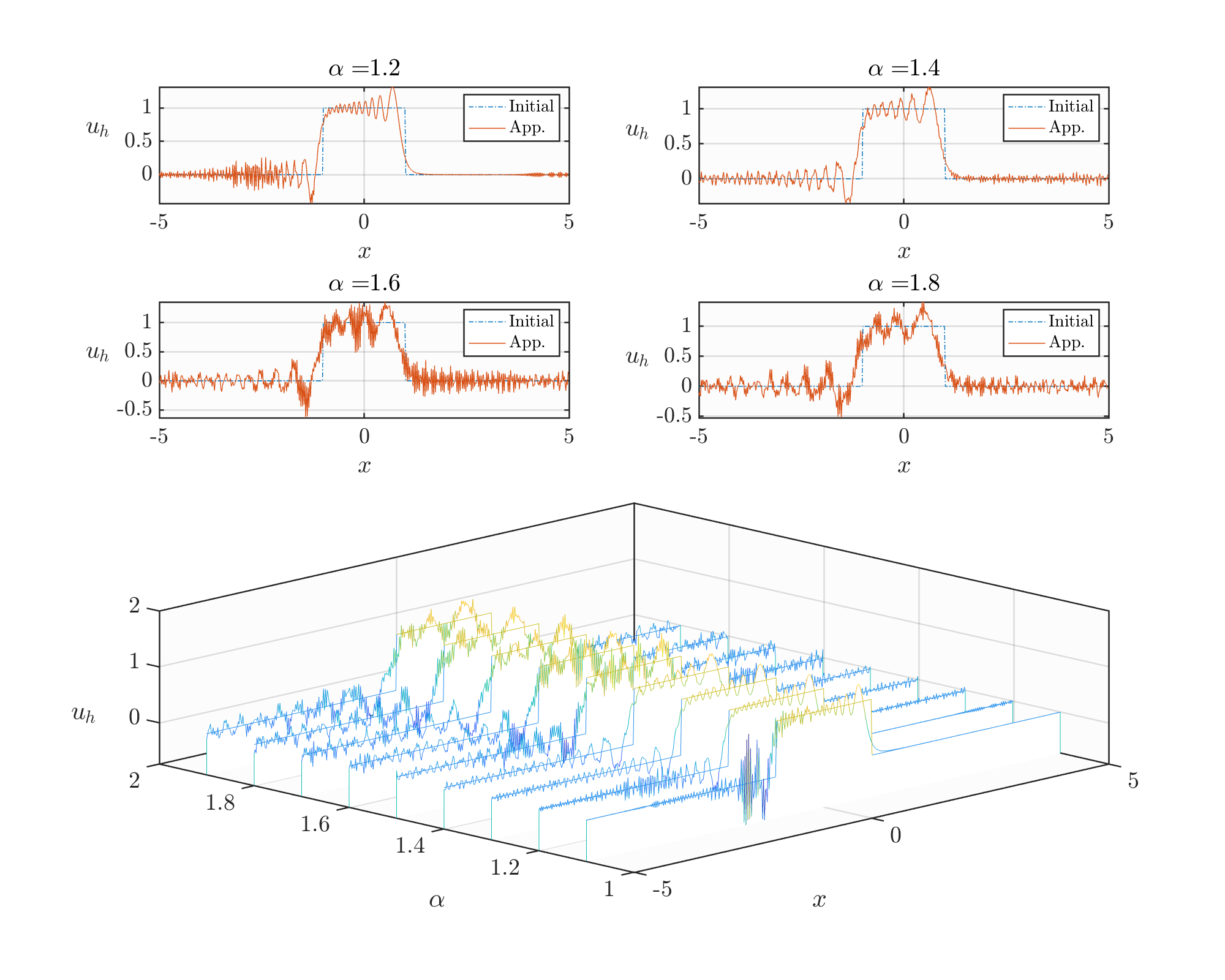}
    \caption{Approximate solution at $T=0.005$ with $N=320$, $k=3$ and initial condition $W_0$ of \eqref{fkdv}.}
    \label{fig:fkdv_L25}
\end{figure}

\subsubsection{Examples of multiple dimensional equation}
We consider the following equation in two space dimensions with the source term $g_d$:
\begin{equation}\label{fkdvh2}
\begin{cases}
     U_t- (-\Delta)_{1}^{\alpha_1/2}U_{x}-(-\Delta)_{2}^{\alpha_2/2}U_{y}=g_d(\mathbf{x},t),  \quad \mathbf{x}:=(x,y) \in [-1,1]^2,~t\in(0,0.001],\\
     U_{0_d}(\mathbf{x}) = \sin(\pi x)\sin(\pi y), \quad (x,y) \in \mathbb{R}^2,
\end{cases}
\end{equation}
where $1<\alpha_1,\alpha_2<2$. Since we do not have the exact solution for \eqref{fkdvh2}, we choose the source term $g_d(\mathbf{x},t)$ as
\begin{equation*}
   g_d(\mathbf{x},t) = e^{-t}\left(-U_{0_d}(\mathbf{x})-\pi^3\sin(\pi y)\Delta^1_{\frac{(\alpha_1-2)}{2}}\cos(\pi x)-\pi^3\sin(\pi x)\Delta^2_{\frac{(\alpha_2-2)}{2}}\cos(\pi y)\right)
\end{equation*}
such that we have the exact solution $U(\mathbf{x},t) = e^{-t}\sin(\pi x)\sin(\pi y)$. 
\begin{table}[htbp]
    \centering
    \begin{tabular}{|c|c|c|c|c|c|c|c|}
        \hline
        \multicolumn{8}{|c|}{\textbf{$(\alpha_1,\alpha_2)$ is chosen in between (1,1) and (2,2)}} \\
        \hline
        \multicolumn{2}{|c|}{\multirow{2}{*}{\textbf{$(\alpha_1,\alpha_2)$$/N/k$}}} & \multicolumn{2}{c|}{\textbf{$k=1$}} & \multicolumn{2}{c|}{\textbf{$k=2$}} & \multicolumn{2}{c|}{\textbf{$k=3$}} \\
        \cline{3-8}
        \multicolumn{2}{|c|}{} & \textbf{$E$} & \textbf{$R_E$} & \textbf{$E$} & \textbf{$R_E$} & \textbf{$E$} & \textbf{$R_E$} \\
        \hline
        \textbf{$(\alpha_1,\alpha_2)=(1.1,1.1)$} & \textbf{$N=18$} & 5.91e-02 &  &1.42e-02 &  & 8.51e-03 & \\
        & &  & 1.80 &  & 3.07 &  & 3.89 \\
        \textbf{} & \textbf{$N=32$}& 2.09e-02 &  & 2.40e-03 &  & 9.11e-04 & \\
        & & & 2.00 &  &  3.04 & & 3.80 \\
        \textbf{} & \textbf{$N=50$}& 8.56e-03 &  & 6.12e-04 &  & 1.67e-04 & \\
        & & & 1.97 &  & 3.00 & & 4.07 \\
        \textbf{} & \textbf{$N=72$} & 4.16e-03 &  & 2.08e-04 &  & 3.78e-05 & \\
        \hline
        \textbf{$(\alpha_1,\alpha_2)=(1.4,1.6)$} & \textbf{$N=18$} & 5.59e-02 &  & 1.31e-02 &  & 8.32e-03 & \\
        & &  & 1.98 &  & 2.94 &  & 3.88\\
        \textbf{} & \textbf{$N=32$}& 1.95e-02 &  & 2.38e-03 &  &8.92e-04 & \\
        & & & 2.17 &  & 3.05 &  & 4.07\\
        \textbf{} & \textbf{$N=50$}& 7.39e-03 &  & 6.09e-04 &  & 1.52e-04 & \\
        & & & 1.99 &  & 3.04 &  & 3.98\\
        \textbf{} & \textbf{$N=72$} & 3.57e-03 &  & 2.03e-04 &  &3.49e-05 & \\
        \hline
       \textbf{$(\alpha_1,\alpha_2)=(1.8,1.8)$} & \textbf{$N=18$} & 5.37e-02 &  & 1.27e-02&  & 8.30e-03& \\
        & &  & 1.97 &  & 2.94 &  & 3.93\\
        \textbf{} & \textbf{$N=32$}& 1.73e-02 &  & 2.34e-03 &  &8.66e-04 & \\
        & & & 1.99 &  & 3.03 &  & 3.99\\
        \textbf{} & \textbf{$N=50$}& 7.15e-03 &  & 6.07e-04 &  & 1.46e-04 & \\
        & & & 2.00 &  & 2.98 &  & 4.00\\
        \textbf{} & \textbf{$N=72$} & 3.45e-03 &  & 2.05e-04 &  & 3.41e-05 & \\
        \hline
    \end{tabular}
    \caption{Error and order of convergence for different values of $(\alpha_1,\alpha_2)$ for the 2D example \eqref{fkdvh2} with source term at time $T=0.001$ taking $N$ elements and polynomial degree $k$.}
    \label{tab:error_table_combined2d}
\end{table}
In this example, we choose number of triangles as $N=2m^2$, $m=3,4,\cdots$. The $L^2$-errors and numerical order of accuracy are depicted in Table \ref{tab:error_table_combined2d}. For various values of $(\alpha_1,\alpha_2)$, we have obtained the optimal rates of convergence for polynomials with degree up to $3$.

Next we consider the following equation without the source term
\begin{equation}\label{fkdvh2dws}
\begin{cases}
     U_t- (-\Delta)_{1}^{\alpha_1/2}U_{x}-(-\Delta)_{2}^{\alpha_2/2}U_{y}=0,  \quad \mathbf{x}:=(x,y) \in [-\pi,\pi]^2,~t\in(0,0.1],\\
     U_{0_d}(\mathbf{x}) = \sin(x)\sin(y), \quad (x,y) \in [-\pi,\pi]^2.
\end{cases}
\end{equation}
Whenever $(\alpha_1,\alpha_2)=(2,2)$, the equation \eqref{fkdvh2dws} becomes $U_t+U_{xxx}+U_{yyy}=0$, and it has the exact solution which is given by $U(\mathbf{x},t) = \sin(x+t)\sin(y+t)$. We wish to investigate the approximate solution of \eqref{fkdvh2dws} whenever $(\alpha_1,\alpha_2)$ is close to $(2,2)$.  
\begin{table}[htbp]
    \centering
    \begin{tabular}{|c|c|c|c|c|c|c|}
    \hline
        \multicolumn{7}{|c|}{$(\alpha_1,\alpha_2)=(1.999,1.999)$} \\
        \hline
        \multirow{2}{*}{\textbf{$N/k$}} & \multicolumn{2}{c|}{\textbf{$k=1$}} & \multicolumn{2}{c|}{\textbf{$k=2$}} & \multicolumn{2}{c|}{\textbf{$k=3$}} \\
        \cline{2-7}
        & \textbf{$E$} & \textbf{$R_E$} & \textbf{$E$} & \textbf{$R_E$} & \textbf{$E$} & \textbf{$R_E$} \\
        \hline
          \textbf{$N=32$} & 5.97e-01 &  & 5.74e-02 &  & 3.20e-02& \\  
        &  & 1.88 &  & 2.78 & & 3.72\\  
         \textbf{$N=50$}& 2.58e-01 &  & 1.66e-02 &  & 6.07e-03 & \\
        & & 2.10 &  & 2.81 & & 3.88  \\
         \textbf{$N=72$} & 1.20e-01 &  & 5.98e-03 &  & 1.47e-03& \\  
        &  & 2.06 &  & 2.91 & & 3.94 \\  
         \textbf{$N=98$}& 6.39e-02 &  & 2.44e-03 &  & 4.36e-04 & \\
        \hline
    \end{tabular}
    \caption{Error and order of convergence for $(\alpha_1,\alpha_2)$ close to (2,2) to compare equation \eqref{fkdvh2dws} with the equation $U_t+U_{xxx}+U_{yyy}=0$ at time $T=0.1$ taking $N$ elements and polynomial degree $k$.}
    \label{tab:error_tableKDV2d} 
\end{table}
The Table \ref{tab:error_tableKDV2d} provides $L^2$-errors by using the number of triangles $N=2m^2,~m=4,5,\cdots$. It is evident that the proposed LDG scheme with $P^k$ elements provides a uniform $(k+1)$-th order of accuracy even for the coarser meshes.
\section{Concluding remarks}\label{sec7}
We have developed a stable LDG scheme for the fractional KdV equation \eqref{fkdv} and \eqref{fkdvhd2} in one space dimension and multiple space dimensions. Although we have achieved  a theoretical order of convergence of $k+\frac{1}{2}$ for the nonlinear flux, the numerical experiments demonstrate the optimal order of convergence $\mathcal{O}(h^{k+1})$. Through extensive numerical experiments covering various values of $\alpha$, we have demonstrated the efficiency of the LDG scheme for one dimensional and multiple space dimensional problems. In cases where exact solutions were not available to validate the LDG scheme, we have introduced an additional source term to obtain the exact solutions, allowing us to obtain the convergence rates. 

Indeed, while we have established the error estimates for the spatial discretization, obtaining error estimates for the fully discrete scheme with higher order time-stepping schemes remains an ongoing task. This endeavor demands additional insights and techniques, which we plan to address in our future work. Furthermore, extending the stability analysis of the fully discrete scheme to the nonlinear case is another important aspect that we intend to explore.
Our future work will focus on proving the stability and obtaining error estimates in the nonlinear setup of the higher order fully discrete scheme.

\section*{Declaration} 
We have not used any data to conduct this work. The authors declare that this work does not have any conflicts of interest.

\section*{Acknowledgements}
The authors would like to thank Chi-Wang Shu for his insightful comments during the execution of this work.



\end{document}